\documentclass[11pt,reqno,nosumlimits]{amsart}%
\usepackage{graphicx, subcaption, color}
\usepackage{tikz-cd}
\usepackage{mathtools, amsfonts, amssymb, amsthm}
\usepackage{enumerate, url}
\usepackage[left=1in, right=1in, top=1in, bottom=0.75in]{geometry}
\usetikzlibrary{trees}
\usetikzlibrary{graphs}
\usetikzlibrary{matrix}
\usepackage[normalem]{ulem}
\usetikzlibrary{decorations.markings}
\usetikzlibrary{calc,fadings,decorations.pathreplacing}
\usepackage{tkz-graph}
\usepackage{pgfplots}
\usepackage[normalem]{ulem}

\tikzset{
    partial ellipse/.style args={#1:#2:#3}{
        insert path={+ (#1:#3) arc (#1:#2:#3)}
    }
}

\usepackage{caption}
\newtheorem{theorem}{Theorem}[section]
\newtheorem{proposition}[theorem]{Proposition}
\newtheorem{proposition/definition}[theorem]{Proposition/Definition}
\newtheorem{lemma}[theorem]{Lemma}
\newtheorem{corollary}[theorem]{Corollary}

\theoremstyle{definition}
\newtheorem{definition}[theorem]{Definition}

\newtheorem{example}[theorem]{Example}

\theoremstyle{remark}

\DeclareMathOperator{\W}{W}
\DeclareMathOperator{\GHZ}{GHZ}
\DeclareMathOperator{\Gr}{Gr}
\DeclareMathOperator{\GL}{GL}

\DeclareMathOperator{\Sub}{Sub}
\DeclareMathOperator{\Seg}{Seg}
\DeclareMathOperator{\sgn}{\varepsilon}
\DeclareMathOperator{\In}{\textsc{in}}
\DeclareMathOperator{\Out}{\textsc{out}}
\DeclareMathOperator{\tns}{\textsc{tns}}

\DeclareMathOperator{\rank}{rank}
\DeclareMathOperator{\brank}{\overline{\rank}}
\DeclareMathOperator{\mrank}{\mu rank}

\newcommand{\tp}{{\scriptscriptstyle\mathsf{T}}}

\newcommand{\blue}[1]{{\color{blue}#1}}

\tikzset{->-/.style={decoration={
  markings,
  mark=at position #1 with {\arrow [ultra thick]{>}}},postaction={decorate}}}
\tikzset{every node/.style={font=\scriptsize}}

\tikzset{->>-/.style={decoration={
  markings,
  mark=at position #1 with {\arrow [thin]{>}}},postaction={decorate}}}
\tikzset{every node/.style={font=\scriptsize}}

\begin{document}
\title{Tensor network ranks}
\author[K.~Ye]{Ke~Ye}
\address{Computational and Applied Mathematics Initiative, Department of Statistics, University of Chicago, Chicago, IL 60637-1514.}
\email{kye@galton.uchicago.edu}
\author[L.-H.~Lim]{Lek-Heng~Lim}
\address{Computational and Applied Mathematics Initiative, Department of Statistics,
University of Chicago, Chicago, IL 60637-1514.}
\email{lekheng@galton.uchicago.edu}

\begin{abstract}
In problems involving approximation, completion, denoising, dimension reduction, estimation, interpolation, modeling, order reduction, regression, etc, we argue that the near-universal practice of assuming that a function, matrix, or tensor (which we will see are all the same object in this context) has \emph{low rank} may be ill-justified. There are many natural instances where the object in question has high rank with respect to the classical notions of rank: matrix rank, tensor rank, multilinear rank --- the latter two being the most straightforward generalizations of the former. To remedy this, we show that one may vastly expand these classical notions of ranks: Given any undirected graph $G$, there is a notion of $G$-rank associated with $G$, which provides us with as many different kinds of ranks as there are undirected graphs. In particular, the popular tensor network states in physics (e.g., \textsc{mps}, \textsc{ttns}, \textsc{peps}) may be regarded as functions of a specific $G$-rank for various choices of $G$. Among other things, we will see that a function, matrix, or tensor may have very high matrix, tensor, or multilinear rank and yet very low $G$-rank for some $G$. In fact the difference is in the orders of magnitudes and the gaps between $G$-ranks and these classical ranks are arbitrarily large for some important objects in computer science, mathematics, and physics. Furthermore, we show that there is a $G$ such that almost every tensor has $G$-rank exponentially lower than its rank or the dimension of its ambient space.
\end{abstract}



\maketitle

\section{Introduction}\label{sec:intro}

A universal problem in science and engineering is to find a function from some given data. The function may be a solution to a PDE with given boundary/initial data or a target function to be learned from a training set of data. In modern applications, one frequently encounters situations where the function lives in some state space or hypothesis space of prohibitively high dimension --- a consequence of requiring very high accuracy solutions or having very large training sets. A common remedy with newfound popularity is to assume that the function has low rank, i.e., may be expressed as a sum of a small number of separable terms. But such a low-rank assumption often has weak or no justification; rank is chosen only because there is no other standard alternative. Taking a leaf from the enormously successful idea of \emph{tensor networks} in physics \cite{CMV, FNW, HND, MBRTV, Orus, OR, SDV, VC, VWPC, V1, V2, WH1993, White1992}, we define a notion of $G$-rank for any undirected graph $G$.  Like tensor rank and multilinear rank, which are extensions of matrix rank to higher order, $G$-ranks contain matrix rank as a special case.

Our definition of $G$-ranks shows that every tensor network ---  tensor trains, matrix product states, tree tensor network states, star tensor network states, complete graph tensor network states,  projected entangled pair states, multiscale entanglement renormalization ansatz, etc --- is nothing more than a set of functions/tensors of some $G$-rank for some undirected graph $G$. It becomes straightforward to explain the effectiveness of tensor networks: They serve as a set of `low $G$-rank functions' that can be used for various purposes (as an ansatz, a regression function, etc). The flexibility of choosing $G$ based on the underlying problem can provide a substantial computational advantage --- a function with high rank or high $H$-rank for a graph $H$ can have much lower $G$-rank for another suitably chosen graph $G$. We will elaborate on these in the rest of this introduction, starting with an informal discussion of tensor networks and $G$-ranks, followed by an outline of our main results.

The best known low-rank decomposition is  the  \emph{matrix rank decomposition}
\begin{equation}\label{eq:matd}
f(x,y) = \sum_{i=1}^r \varphi_i(x) \psi_i(y)
\end{equation}
that arises in common matrix decompositions such as \textsc{lu}, \textsc{qr}, \textsc{evd}, \textsc{svd}, Cholesky,  Jordan, Schur, etc --- each differing in the choice of additional structures on the factors $\varphi_i$ and $\psi_i$. In higher order, say, order three for notational simplicity, \eqref{eq:matd} generalizes as \emph{tensor rank decomposition},
\begin{equation}\label{eq:td}
f(x,y,z) = \sum_{i=1}^r \varphi_i(x) \psi_i(y) \theta_i(z),
\end{equation}
or as \emph{multilinear rank decomposition}
\begin{equation}\label{eq:md}
f(x,y,z) = \sum_{i,j,k=1}^{r_1,r_2,r_3} \varphi_i(x) \psi_j(y) \theta_k(z).
\end{equation}
Like \eqref{eq:matd}, \eqref{eq:td} and \eqref{eq:md}  decompose a function $f$ into a sum of products of factors $\varphi_i$, $\psi_j$, $\theta_k$, simpler functions that depend on fewer variables than $f$.  This simple idea is ubiquitous, underlying the separation-of-variables technique in partial differential equations \cite{svpde} and special functions \cite{svbook}, fast Fourier transforms \cite{svfft},  tensor product splines \cite{Chui} in approximation theory, mean field approximations \cite{StatPhy} in statistical physics, na\"{\i}ve Bayes model \cite{LC1} and tensor product kernels \cite{HSS} in machine learning, blind multilinear identification \cite{LC2} in signal processing.

The decompositions \eqref{eq:td} and \eqref{eq:md} can be inadequate when modeling more complicated interactions, calling for \emph{tensor network decompositions}. Some of the most popular ones include 
\emph{matrix product states} (\textsc{mps}) \cite{ITEH1987},
\[
f(x,y,z) = \sum_{i,j,k=1}^{r_1,r_2,r_3} \varphi_{ij}(x) \psi_{jk}(y) \theta_{ki}(z),
\]
\emph{tree tensor network states} (\textsc{ttns}) \cite{SDV},
\[
f(x,y,z,w) = \sum_{i,j,k=1}^{r_1,r_2,r_3} \varphi_{ijk}(x) \psi_{i}(y) \theta_{j}(z) \pi_{k}(w),
\]
\emph{tensor train}\footnote{We are aware that tensor trains \cite{O} have long been known in physics \cite{Anderson1959, White1992, WH1993} and are often called matrix product states \emph{with open boundary conditions} \cite{Orus}. What we called matrix product states are known more precisely as matrix product states \emph{with periodic boundary conditions} \cite{Orus}. We thank our colleagues in physics for pointing this out to us on many occasions. In our article, we use the terms \textsc{tt} and \textsc{mps} merely for the convenience of easy distinction between the two types of \textsc{mps}. We will say more about our  nomenclature after Definition~\ref{def:specialtns}.} (\textsc{tt}) \cite{Anderson1959},
\[
f(x,y,z,u,v) = \sum_{i,j,k,l=1}^{r_1,r_2,r_3,r_4}\varphi_{i}(x) \psi_{ij}(y) \theta_{jk}(z)\pi_{kl}(u)\rho_{l}(v),
\]
and \emph{projected entangled pair states} (\textsc{peps}) \cite{VC},\label{pg:peps}
\[
f(x,y,z,u,v,w) =\sum_{i,j,k,l,m,n,o=1}^{r_1,r_2,r_3,r_4,r_5,r_6,r_7} \varphi_{ij}(x) \psi_{jkl}(y) \theta_{lm}(z) \pi_{mn}(u) \rho_{nko}(v) \sigma_{oi}(w),
\]
among many others. Note that all these decompositions, including those in \eqref{eq:matd}, \eqref{eq:td}, \eqref{eq:md}, are of the same nature --- they decompose a function into a sum of \emph{separable functions}. Just as \eqref{eq:td} and \eqref{eq:md} differ in how the factors are indexed, tensor network decompositions differ from each other and from \eqref{eq:td} and \eqref{eq:md} in how the factors are indexed.
Every tensor network  decomposition is defined by an undirected graph $G$ that determines the indexing of the factors. The graphs associated with \textsc{mps}, \textsc{ttns}, \textsc{tt}, and \textsc{peps} are shown in Figure~\ref{fig:networks}. 
The decompositions above represent the simplest non-trivial instance for each tensor network --- they can become arbitrarily complicated with increasing \emph{order}, i.e., the number of arguments of the function $f$ or, equivalently, the number of vertices in the corresponding graphs. In Section~\ref{sec:tns}, we will formally define tensor network states in a mathematically rigorous and, more importantly, coordinate-free manner --- the importance of the latter  stems from the avoidance of a complicated mess of indices, evident even in the simplest instance of \textsc{peps} above. For now, a \emph{tensor network state} is an $f$ that has a tensor network decomposition corresponding to a given graph $G$, and a \emph{tensor network} corresponding to $G$ is the set of all such functions.
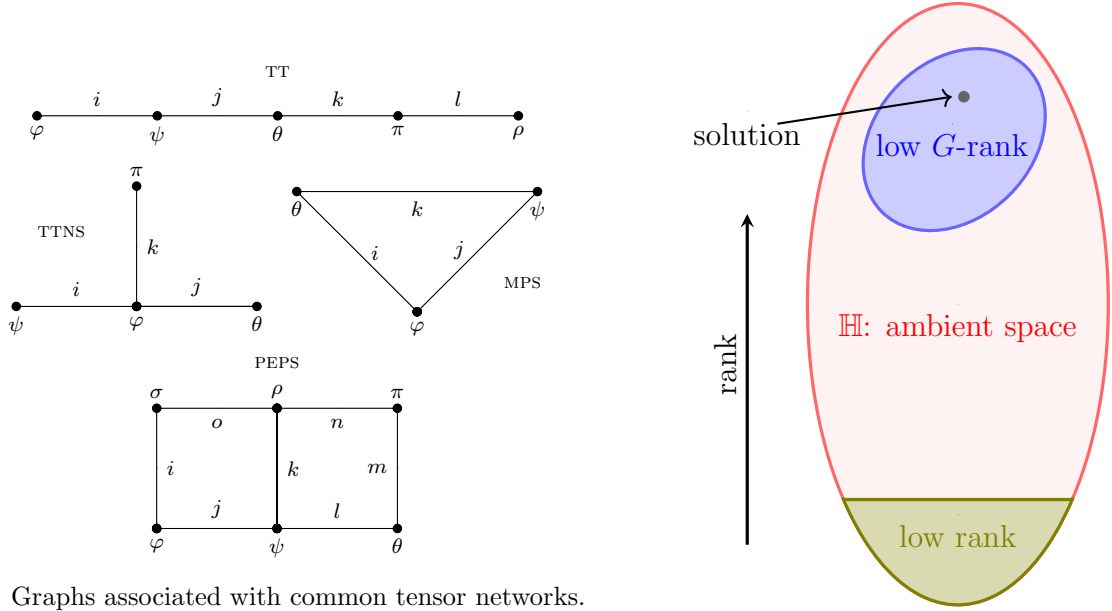
\begin{figure}
\centering
\begin{subfigure}{.64\textwidth}
\centering
\begin{tikzpicture}[scale=0.8]
\filldraw
     (4,1) node[align=center, below] {\textsc{tt}}
     (0,0) circle (2pt) node[align=center, below] {$\varphi$}
 -- (1,0) circle (0pt) node[align=center, above] {$i$}
 -- (2,0) circle (2pt)  node[align=center, below] {$\psi$}  
  -- (3,0) circle (0pt) node[align=center, above] {$j$}
 -- (4,0) circle (2pt)  node[align=center, below] {$\theta$}
  -- (5,0) circle (0pt) node[align=center, above] {$k$}
   -- (6,0) circle (2pt) node[align=center, below] {$\pi$}
    -- (7,0) circle (0pt) node[align=center, above] {$l$}
 -- (8,0) circle (2pt) node[align=center, below] {$\rho$} 
-- cycle;
\end{tikzpicture}
\begin{tikzpicture}[scale=0.8]
\filldraw
     (0.75,1.5) node[align=center, below] {\textsc{ttns}}
     (0,0) circle (2pt) node[align=center, below] {$\psi$}
      -- (1,0) circle (0pt) node[align=center, above] {$i$}
 -- (2,0) circle (2pt)  node[align=center, below] {$\varphi$}  
 -- (4,0) circle (2pt)  node[align=center, below] {$\theta$}  
  -- (3,0) circle (0pt) node[align=center, above] {$j$}
 -- (2,0) circle (2pt) node[align=center, below]{} 
  -- (2,1) circle (0pt) node[align=center, right] {$k$}
 -- (2,2) circle (2pt) node[align=center, above] {$\pi$} 
-- cycle;
\end{tikzpicture}
\begin{tikzpicture}[scale=0.8]
\filldraw
    (1.75,0.75) node[align=center, below] {\textsc{mps}}
    (0,0) circle (2pt) node[align=center, below] {$\varphi$}
     -- (0.7,0.7) circle (0pt) node[align=center, above] {$j$}
 -- (2,2) circle (2pt)  node[align=center, below] {$\psi$}  
  -- (0,2) circle (0pt) node[align=center, below] {$k$}
 -- (-2,2) circle (2pt)  node[align=center, below] {$\theta$}  
  -- (-0.7,0.7) circle (0pt) node[align=center, above] {$i$}
 -- (0,0) circle (2pt) node[align=center, below]{}
-- cycle;
\end{tikzpicture}
\begin{tikzpicture}[scale=0.8]
\filldraw
    (2,2.5) node[align=center, above] {\textsc{peps}}
    (0,0) circle (2pt) node[align=center, below] {$\varphi$}
     -- (1,0) circle (0pt) node[align=center, above] {$j$}
 --(2,0) circle (2pt)  node[align=center, below] {$\psi$}  
  -- (3,0) circle (0pt) node[align=center, above] {$l$}
 --(4,0) circle (2pt) node[align=center, below] {$\theta$}
  -- (4,1) circle (0pt) node[align=center, left] {$m$}
 --(4,2) circle (2pt)  node[align=center, above] {$\pi$}
   -- (3,2) circle (0pt) node[align=center, below] {$n$}
 -- (2,2) circle (2pt)  node[align=center, above] {$\rho$}
    -- (2,1) circle (0pt) node[align=center, right] {$k$}
  --(2,0) circle (2pt) node[align=center, below]{}
   -- (2,2) circle (2pt) node[align=center, above]{}
       -- (1,2) circle (0pt) node[align=center, below] {$o$}
 -- (0,2) circle (2pt)  node[align=center, above] {$\sigma$}
     -- (0,1) circle (0pt) node[align=center, right] {$i$}
--   (0,0) circle (2pt) node[align=center, below]{}
 -- cycle;
\end{tikzpicture}
\caption{Graphs associated with common tensor networks.}
\label{fig:networks}
\end{subfigure}
\begin{subfigure}{.35\textwidth}
\begin{tikzpicture}[rotate = 90, scale = 0.8]

\draw[very thick,-{stealth[scale=1]},black,shift={(5.5 cm,-1.5cm)}] (-7.5,6) to  (-2,6);
     \filldraw[color =black!60,fill=black!60, very thick,shift={(5.5 cm,0 cm)}]
      (-4.5,4.5) circle (0pt) node[align=center, above,rotate=90]{\large\textcolor{black}{rank}}
     (0,1) circle (1pt) node[align=center, above]{};

    \filldraw[color=red!60,fill=red!5, very thick] (2,1) ellipse (5cm and 2.5cm);

\filldraw[color=blue!60,fill=blue!20, very thick,rotate =45] (4.1,-2.6) ellipse (1.3cm and 1.7cm);

     \filldraw[shift={(6.2 cm,2 cm)}]     
     (-1.2,-0.9) circle (0pt) node[align=center, below] {\large\blue{low $G$-rank}};
     
     \filldraw
     (2,1) circle (0pt) node[align=center, below] {\large\textcolor{red}{$\mathbb{H}$: ambient space}};

     \draw[thick,-{>[scale=1]},black,shift={(5.5 cm,1 cm)}] (-0.5,3) to  (-0.05,0.1);

     \filldraw[color =black!60,fill=black!60, very thick,shift={(5.5 cm,0 cm)}]
      (-0.3,4.5) circle (0pt) node[align=center, below]{\large\textcolor{black}{solution}}
     (-0.05,0.9) circle (2pt) node[align=center, above]{};

\filldraw[color=green!50!red,fill=green!50!red!30, very thick]
(2,1) [partial ellipse=130:230:5cm and 2.5cm];

\draw[very thick,color=green!50!red] (-1.25,2.9) to  (-1.25,-0.9);

\filldraw[shift={(-7 cm,-0.5 cm)}]
     (5.5,1.5) circle (0pt) node[align=center, below] {\large \textcolor{green!50!red}{low rank}};
\end{tikzpicture}
\caption{$G$-rank versus rank.}
\label{fig:schematic}
\end{subfigure}
\caption{Every undirected graph $G$ defines a $G$-rank. Solution that we seek may have high rank but low $G$-rank.}
\end{figure}

The minimum $r$ in \eqref{eq:matd} gives us the matrix rank of $f$; the minimum $r$ in \eqref{eq:td} and the minimum $(r_1,r_2, r_3)$ in \eqref{eq:md} give us the \emph{tensor rank} and \emph{multilinear rank} of $f$ respectively. Informally, the tensor network rank or $G$-rank of $f$ may be similarly defined by requiring some form of minimality for $(r_1,\dots,r_c)$ in the other decompositions for \textsc{mps}, \textsc{tt}, \textsc{ttns}, \textsc{peps} (with an appropriate graph $G$ in each case). Note that this is no longer so straightforward since (i) $\mathbb{N}^c$ is not an ordered set when $c > 1$; (ii) it is not clear that any function would have such a decomposition for an arbitrary $G$.

We will show in Section~\ref{sec:Grank} that any $d$-variate function or  $d$-tensor  has a $G$-rank for any undirected connected graph $G$ with $d$ vertices. While this has been defined in special cases, particularly when $G$ is a path graph (\textsc{tt}-rank \cite{HK2009}) or more generally when $G$ is a tree (hierarchical rank \cite[Chapter~11]{Hackbusch} or tree rank \cite{BSU}), we show that the notion is well-defined for \emph{any} undirected connected graph $G$: Given any  $d$ vector spaces $\mathbb{V}_1,\dots,\mathbb{V}_d$ of arbitrary dimensions, there is a class of tensor network states associated with $G$, as well as a $G$-rank for any $T \in \mathbb{V}_1 \otimes \dots \otimes \mathbb{V}_d$; or equivalently, for any function $f \in L^2(X_1 \times \dots \times X_d)$; or, for those accustomed to working in terms of coordinates, for any hypermatrix $A \in \mathbb{C}^{n_1 \times \dots \times n_d}$. See Section~\ref{sec:tns} for a discussion on the relations between these objects ($d$-tensors, $d$-variate functions, $d$-hypermatrices).

Formalizing the notions of tensor networks and $G$-ranks provides several advantages, the most important of which is that it allows one to develop a rich calculus for working with tensor networks: deleting vertices, removing edges, restricting to subgraphs, taking unions of graphs, restricting to subspaces, taking intersections  of tensor network states, etc. We develop some of these basic techniques and properties in Sections~\ref{sec:calculus} and  \ref{sec:prop}, deferring to \cite{dtn} the more involved properties that are not needed for the rest of this article. Among other advantages, the notion of $G$-rank also sheds light on existing  methods in scientific computing: In hindsight, the algorithm  in \cite{ying2016} is one that approximates a given tensor network state by those of low $G$-rank. 

The results in Section~\ref{sec:low} may be viewed as the main impetus for tensor networks (as we pointed out earlier, these are `low $G$-rank tensors' for various choices of $G$):
\begin{itemize}
\item a tensor may have very high matrix, tensor, or multilinear rank and yet very low $G$-rank;
\item a tensor may have very high $H$-rank and very low $G$-rank for $G \ne H$;
\end{itemize}
We will exhibit an explicit example where tensor rank, multilinear rank, matrix rank, and \textsc{tt}-rank are all $O(n^2)$ but whose \textsc{mps}-rank is $O(n)$. In Section~\ref{sec:prop}, we will see that there is a choice of $G$ such that in a space of dimension $O(n^d)$, \emph{almost every} tensor  has $G$-rank $O\bigl(n(d-1)\bigr)$ and tensor rank $O\bigl(n^d/(nd - d +1)\bigr)$, i.e., for that particular $G$, almost every tensor has $G$-rank that is exponentially lower than its tensor rank or the dimension of its ambient space.

We will study in detail the simplest and most common $G$-ranks: \textsc{tt}-rank ($G$ is a path graph) in Section~\ref{sec:tt}, \textsc{ttns}-rank ($G$ is a tree) in Section~\ref{sec:ttns}, \textsc{mps}-rank ($G$ is a cyclic graph) in Section~\ref{sec:mps}, paying particular attention to questions of uniqueness, existence of best low $G$-rank approximations, polynomial-time computability, dimensions, generic and maximal $G$-ranks, etc.

Some other insights that may be worth highlighting include:
\begin{itemize}
\item Any tensor network state is the contraction of a rank-one tensor (Section~\ref{sec:tns}).

\item $G$-rank is polynomial-time computable when $G$ is acyclic (Section~\ref{sec:ttns}).

\item A best low $G$-rank approximation always exists if $G$ is acyclic (Section~\ref{sec:ttns}) but may not necessarily exist if $G$ contains a cycle\footnote{This was established in \cite{LQY} for $G = C_3$, a $3$-cycle; we show that it holds for any $G$ that contains a $d$-cycle, $d \ge 3$.} (Section~\ref{sec:mps}). 

\item $G$-ranks are distinct from tensor rank and multilinear rank in that neither is a special case of the other (Section~\ref{sec:compare}) but $G$-ranks may be regarded as an `interpolant' between tensor rank and multilinear rank (Section~\ref{sec:Grank}).

\end{itemize}
In Section~\ref{sec:eg}, we determine $G$-ranks of decomposable tensors, decomposable symmetric and skew-symmetric tensors, monomials, W state, GKZ state, and the structure tensor of matrix-matrix product for various choices of $G$.

\section{Tensor network states}\label{sec:tns}

We have left the function spaces in the decompositions in Section~\ref{sec:intro} unspecified. In physics applications where tensor networks were first studied \cite{Orus}, they are often assumed to be Hilbert spaces. For concreteness we may assume that they are all $L^2$-spaces, e.g., in \textsc{mps} we have $\varphi_{ij} \in L^2(X)$, $\psi_{jk} \in L^2(Y)$,  $\theta_{ki} \in L^2(Z)$ for all $i,j,k$ and  $f \in L^2(X \times Y \times Z) = L^2(X) \otimes L^2(Y) \otimes L^2(Z)$, although we may also allow for other function spaces that admit tensor product.

The reason we are not concern with the precise type of function space is that in this article we limit ourselves to finite-dimensional spaces, i.e., $X, Y, Z, \dots$ are finite sets and $x,y,z,\dots$ are discrete variables that take a finite number of values. In this case, it is customary\footnote{A vector $a \in \mathbb{C}^n$ is a function $f : \{1,\dots,n\} \to \mathbb{C}$ with $f(i) = a_i$; and a matrix/hypermatrix $A \in \mathbb{C}^{n_1 \times \dots \times n_d}$ is a function $f : \{1,\dots, n_1\} \times \dots \times \{1,\dots, n_d\} \to \mathbb{C}$ with $f(i_1,\dots, i_d) = a_{i_1 \cdots i_d}$ \cite{HLA}.} to identify $L^2(X) \cong \mathbb{C}^m$, $L^2(Y) \cong \mathbb{C}^n$, $L^2(Z) \cong \mathbb{C}^p$, $L^2(X \times Y \times Z) \cong \mathbb{C}^{m \times n \times p}$, where $m = \# X$, $n = \#Y$, $p = \#Z$, and write an \textsc{mps} decomposition as a decomposition
\begin{equation}\label{eq:A}
A_{\textsc{mps}} = \sum_{i,j,k=1}^{r_1,r_2,r_3} a_{ij} \otimes b_{jk} \otimes c_{ki},
\end{equation}
where $A \in \mathbb{C}^{m \times n \times p}$, $a_{ij} \in \mathbb{C}^m$, $b_{jk} \in \mathbb{C}^n$, $c_{kl} \in \mathbb{C}^p$ for all $i,j,k$. In order words, in finite dimension, the function $f$ is represented by a hypermatrix $A$ and the factor functions $\varphi_{ij}$, $\psi_{jk}$, $\theta_{ki}$ are represented by factor vectors $a_{ij}$, $b_{jk}$, $c_{kl}$ respectively. Henceforth we will use the word \emph{factor} regardless of whether it is a function or a vector.

The same applies to other tensor networks when the spaces are finite-dimensional --- they may all be regarded as decompositions of hypermatrices into sums of tensor products of vectors. For easy reference, we list the \textsc{ttns}, \textsc{tt}, \textsc{peps} decompositions below:
\begin{align}
A_{\textsc{ttns}}  &= \sum_{i,j,k=1}^{r_1,r_2,r_3} a_{ijk} \otimes b_{i} \otimes c_{j} \otimes d_{k},  \label{eq:B} \\
A_{\textsc{tt}}  &= \sum_{i,j,k,l=1}^{r_1,r_2,r_3,r_4} a_{i} \otimes b_{ij} \otimes c_{jk} \otimes d_{kl} \otimes e_{l},  \label{eq:C}  \\
A_{\textsc{peps}}  &=\sum_{i,j,k,l,m,n,o=1}^{r_1,r_2,r_3,r_4,r_5,r_6,r_7} a_{ij} \otimes b_{jkl} \otimes c_{lm} \otimes d_{mn} \otimes e_{nko} \otimes f_{oi},  \label{eq:D}
\end{align}
Note that $A_{\textsc{ttns}}$, $A_{\textsc{tt}}$, $A_{\textsc{peps}}$ are hypermatrices of orders $4$, $5$, $7$ respectively. In particular we observe that the simplest nontrivial instance of \textsc{peps} already involve a tensor of order $7$, which is a reason tensor network decompositions are more difficult than the well-studied decompositions associated with tensor rank and multilinear rank, where order-$3$ tensors already capture most of their essence.

From these examples, it is not difficult to infer the general definition of a tensor network decomposition \emph{in coordinates}. Take any undirected graph $G = (V,E)$ and assign a positive integer weight to each edge. Then a tensor network decomposition associated with $G$ may be constructed from the correspondence in  Table~\ref{tab:corr}.
\begin{table}[h!]
\centering
\begin{tabular}{l|l|l}
\emph{Graph} & \emph{Tensor} (function/hypermatrix) & \emph{Notation} \\ \hline
vertices & factors & $\varphi, \psi, \theta, \dots$/$a,b,c, \dots$ \\
edges & contraction indices  & $i,j,k, \dots $ \\
degree of vertex & number of indices in each factor & $n_1,\dots,n_d$ \\
weight of edge & upper limit of summation & $r_1,\dots,r_c$ \\
number of vertices & order of tensor & $d$ \\
number of edges & number of indices contracted & $c$
\end{tabular}
\bigskip
\caption{How a graph determines a tensor network decomposition.}
\label{tab:corr}
\end{table}

As we can see from even the simplest instance of \textsc{peps} above, a coordinate-dependent approach quickly run up against an impenetrable wall of indices. Aside from having to keep track of a large number of indices and their summation limits, we also run out of characters for labeling them (e.g., the functional form of the simplest instance of \textsc{peps} on p.~\pageref{pg:peps} already uses up 20 Roman and Greek alphabets), requiring even messier sub-indices.  We may observe that the label of a factor, i.e., $\varphi, \psi, \theta, \dots$ in the case of functions and $a, b, c, \dots$ in the case of vectors, plays no role in the decompositions --- only its indices matter. This is the impetus behind physicists' Dirac notation, in which \textsc{mps}, \textsc{ttns}, \textsc{tt}, \textsc{peps} are expressed as
\begin{align*}
A_{\textsc{mps}} &= \sum_{i,j,k=1}^{r_1,r_2,r_3} \lvert i, j\rangle  \lvert j, k\rangle  \lvert k, i\rangle, \\
A_{\textsc{ttns}}  &= \sum_{i,j,k=1}^{r_1,r_2,r_3} \lvert i, j, k\rangle  \lvert i\rangle  \lvert j\rangle  \lvert k\rangle, \\
A_{\textsc{tt}}  &= \sum_{i,j,k,l=1}^{r_1,r_2,r_3,r_4} \lvert i\rangle  \lvert i, j\rangle  \lvert j, k\rangle  \lvert k, l\rangle  \lvert l\rangle, \\
A_{\textsc{peps}} &=\sum_{i,j,k,l,m,n,o=1}^{r_1,r_2,r_3,r_4,r_5,r_6,r_7} \lvert i, j\rangle  \lvert j, k, l\rangle  \lvert l, m\rangle  \lvert m, n\rangle  \lvert n, k, o\rangle  \lvert o, i\rangle, 
\end{align*}
respectively. While this notation is slightly more economical, it does not circumvent the problem of indices.
With this in mind, we will adopt a modern \emph{coordinate-free} definition of tensor networks similar to the one in \cite{LQY} that by and large avoids the issue of indices.

Let $G = (V,E)$ be an undirected graph where the set of $d$ vertices and the set of $c$ edges are labeled respectively by
\begin{equation}\label{eq:graphlabels}
V  = \{ 1, \dots, d\} \quad \text{and}\quad
E  = \bigl\{ \{i_1, j_1\},\dots,\{i_c, j_c\} \bigr\} \subseteq \binom{V}{2}.
\end{equation}
We will assign arbitrary directions to the edges:
\[
\underline{E} = \bigl\{ (i_1,j_1),\dots,(i_c,j_c) \bigr\} \subseteq V \times V
\]
but still denote the resulting directed graph  $G$ for the following reason: Tensor network states depend only on the undirected graph structure of $G$ --- two directed graphs with the same underlying undirected graph give isomorphic tensor network states \cite{LQY}. For each $i \in V$, let
\[
\In(i)  = \bigl\{j \in \{1,\dots,c\}: (j,i) \in \underline{E}\bigr\},\qquad
\Out(i) = \bigl\{j\in \{1,\dots,c\} : (i,j) \in \underline{E}\},
\]
i.e., the sets of vertices pointing into and out of $i$ respectively. As usual, for a directed edge $(i,j)$, we will call $i$ its head and $j$ its tail.

The recipe for constructing tensor network states is easy to describe informally: Given any graph $G=(V,E)$, assign arbitrary directions to the edges to obtain $\underline{E}$; attach a vector space $\mathbb{V}_i$ to  each vertex $i$; attach a covector space  $\mathbb{E}_j^*$ to the head and a vector space  $\mathbb{E}_k$ to the tail of each directed edge $(j,k)$; do this for all vertices in $V$ and all directed edges in $\underline{E}$; contract along all edges to obtain a tensor in $\mathbb{V}_1 \otimes \dots \otimes \mathbb{V}_d$. The set of all tensors obtained this way form the tensor network states associated with $G$. We make this recipe precise in the following.

We will work over $\mathbb{C}$ for convenience although the discussions in this article will also apply to $\mathbb{R}$. We will also restrict ourselves mostly to finite-dimensional vector spaces as our study here is undertaken with a view towards computations and in computational applications of tensor networks, infinite-dimensional spaces are invariably approximated by finite-dimensional ones. 

Let $\mathbb{V}_1,\dots,\mathbb{V}_d$ be complex vector spaces with $\dim \mathbb{V}_i = n_i$, $i=1,\dots,d$. Let $\mathbb{E}_1,\dots, \mathbb{E}_c$ be complex vector spaces with $\dim \mathbb{E}_j = r_j$, $j = 1,\dots,c$. We denote the dual space of $\mathbb{V}_i$ by $\mathbb{V}^*_i$ (and that of $\mathbb{E}_j$ by $\mathbb{E}_j^*$). For each $i \in V$, consider the tensor product space 
\begin{equation}\label{eq:basic}
\Bigl(\bigotimes_{j\in \In(i)} \mathbb{E}_j \Bigr) \otimes \mathbb{V}_i \otimes \Bigl( \bigotimes_{j\in \Out(i)} \mathbb{E}^*_j \Bigr) 
\end{equation}
and the contraction map 
\begin{equation}\label{eq:con}
\kappa_G : \bigotimes_{i=1}^d \Bigl[\Bigl(\bigotimes_{j\in \In(i)} \mathbb{E}_j \Bigr) \otimes \mathbb{V}_i \otimes \Bigl( \bigotimes_{j\in \Out(i)} \mathbb{E}^*_j \Bigr) \Bigr] \to \bigotimes_{i=1}^d \mathbb{V}_i,
\end{equation}
defined by contracting factors in $\mathbb{E}_j$ with factors in $\mathbb{E}^*_j$. Since any directed edge $(i,j)$ must point out of a vertex $i$ and into a vertex $j$, each copy of $ \mathbb{E}_j^*$ is paired with one and only one copy of $\mathbb{E}_j$, i.e., the contraction is well-defined. 
\begin{definition}\label{def:tns}
A tensor in $\mathbb{V}_1\otimes \dots \otimes \mathbb{V}_d$ that can be written as $\kappa_G(T_1\otimes \cdots \otimes T_d)$ where  
\[
T_i\in \Bigl(\bigotimes_{j\in \In(i)} \mathbb{E}_j \Bigr) \otimes \mathbb{V}_i \otimes \Bigl( \bigotimes_{j\in \Out(i)} \mathbb{E}^*_j \Bigr),\qquad i=1,\dots,d,
\]
and $\kappa_G$ as in \eqref{eq:con}, is called a \emph{tensor network state}  associated to the undirected graph $G$ and vector spaces $\mathbb{V}_1,\dots, \mathbb{V}_d$, $\mathbb{E}_1,\dots, \mathbb{E}_c$. The set of all such tensor network states is called the \emph{tensor network} and denoted
\begin{multline*}
\tns(G;\mathbb{E}_1,\dots,\mathbb{E}_c;\mathbb{V}_1,\dots, \mathbb{V}_d)
\coloneqq \Bigl\{ \kappa_G(T_1\otimes \dots \otimes T_d) \in \mathbb{V}_1 \otimes \dots \otimes \mathbb{V}_d : \\
T_i \in \Bigl(\bigotimes_{j\in \In(i)} \mathbb{E}_j \Bigr) \otimes \mathbb{V}_i \otimes \Bigl( \bigotimes_{j\in \Out(i)} \mathbb{E}^*_j\Bigr) ,\; i=1,\dots,d
\Bigr\}.
\end{multline*}
We will always require that $\mathbb{E}_1,\dots,\mathbb{E}_c$ be finite-dimensional but $\mathbb{V}_1,\dots,\mathbb{V}_d$ may be of any dimensions, finite or infinite.
Since a vector space is determined up to isomorphism by its dimension, when the vector spaces $\mathbb{E}_1,\dots,\mathbb{E}_c$ are unimportant (these play the role of contraction indices), we will simply denote the tensor network by
$\tns(G; r_1,\dots,r_c; \mathbb{V}_1,\dots, \mathbb{V}_d)$; or, if the vector spaces  $\mathbb{V}_1,\dots,\mathbb{V}_d$ are also unimportant and finite-dimensional, we will denote it by
$\tns(G; r_1,\dots,r_c; n_1,\dots,n_d)$. As before, $n_i = \dim \mathbb{V}_i$ and $r_j = \dim \mathbb{E}_j$.
\end{definition}

While we have restricted Definition~\ref{def:tns} to tensor products of vector spaces $\mathbb{V}_1\otimes \dots \otimes \mathbb{V}_d$  for the purpose of this article, the definition works with any types of mathematical objects with a notion of tensor product: $\mathbb{V}_1,\dots,\mathbb{V}_d$ may be modules or algebras, Hilbert or Banach spaces,  von Neumann or $C^*$-algebras, Hilbert $C^*$-modules, etc. In fact we will need to use  Definition~\ref{def:tns} in the form where $\mathbb{V}_1,\dots,\mathbb{V}_d$ are vector bundles in Section~\ref{sec:calculus}.

Since they will be appearing with some frequency, we will introduce abbreviated notations for the \emph{inspace} and \emph{outspace} appearing in \eqref{eq:basic}: For each vertex $i=1,\dots, d$, set 
\begin{equation}\label{eq:inout}
\mathbb{I}_{i} \coloneqq \bigotimes_{j\in \In(i)} \mathbb{E}_j\qquad \text{and}\qquad  \mathbb{O}_{i} \coloneqq \bigotimes_{j\in \Out(i)} \mathbb{E}^*_j.
\end{equation}

Note that the image of every contraction map $ \kappa_G(T_1\otimes \dots \otimes T_d) $ gives a decomposition like the ones we saw in \eqref{eq:A}--\eqref{eq:D}. We call such a decomposition a \emph{tensor network decomposition} associated with $G$. A tensor $T \in  \mathbb{V}_1 \otimes \dots \otimes \mathbb{V}_d $ is said to be $G$-\emph{decomposable} if it can be expressed as $T = \kappa_G(T_1 \otimes \dots T_d)$ for some $r_1,\dots, r_c \in \mathbb{N}$; a fundamental result here  (see Theorem~\ref{thm:every tensor is a TNS}) is that:
\begin{quote}
\emph{Given any $G$ and any $\mathbb{V}_1,\dots, \mathbb{V}_d$, every tensor in $\mathbb{V}_1 \otimes \dots \otimes \mathbb{V}_d$ is $G$-decomposable  when $r_1,\dots,r_c$ are sufficiently large.}
\end{quote}
The tensor network $\tns(G; r_1,\dots,r_c; \mathbb{V}_1,\dots,\mathbb{V}_d)$ is simply the set of all $G$-decomposable tensors for a fixed choice of  $r_1,\dots, r_c $. A second fundamental result (see Definition~\ref{def:tnr} and discussions thereafter) is that:
\begin{quote}
\emph{Given any $G$ and any $T \in \mathbb{V}_1 \otimes \dots \otimes \mathbb{V}_d$, there is minimum choice of $r_1,\dots,r_c$ such that $T \in \tns(G; r_1,\dots,r_c; \mathbb{V}_1,\dots,\mathbb{V}_d)$.}
\end{quote}

The undirected graph $G$ can be extremely general. We impose no restriction on $G$ --- self-loops, multiple edges, disconnected graphs, etc --- are all permitted. However, we highlight the following:
\begin{description}
\item[Self-loops] Suppose a vertex $i$ has a self-loop, i.e., an edge $e$ from $i$ to itself. Let $\mathbb{E}_e$ be the vector space attached to $e$. Then by definition   $\mathbb{E}_e$ and $\mathbb{E}^*_e$ must both appear in the inspace and  outspace of $i$ and upon contraction they serve no role in the tensor network state; e.g., for $C_1$, the single vertex graph with one self-loop, $\kappa_{C_1}(\mathbb{E}_e \otimes \mathbb{V}_i \otimes \mathbb{E}_e^* )= \mathbb{V}_i$. Hence self-loops in $G$ have no effect on the tensor network states defined by $G$.

\item[Multiple edges] Multiple edges $e_1,\dots,e_m$ with vector spaces $\mathbb{E}_1,\dots,\mathbb{E}_m$ attached have the same effect as a single edge $e$ with the vector space $\mathbb{E}_1 \otimes \dots \otimes \mathbb{E}_m$ attached, or equivalently, multiple edges $e_1,\dots,e_m$ with edge weights $r_1,\dots,r_m$ have the same effect as a single edge with edge weight $r_1\cdots r_m$.

\item[Degree-zero vertices] If $G$ contains  a vertex of degree zero, i.e., an isolated vertex not connected to the rest of the graph, then by Definition~\ref{def:tns},
\begin{equation}\label{eq:isolate}
\tns(G; r_1,\dots,r_c; n_1,\dots,n_d) = \{0\}.
\end{equation}

\item[Weight-one edges] If $G$ contains an edge of weight one, i.e., a one-dimensional vector space is attached to that edge, then by Definition~\ref{def:tns}, that edge may be dropped. See Proposition~\ref{prop:remove an edge} for details.
\end{description}
In particular, allowing for a multigraph adds nothing to the definition of tensor network states and we may assume  that $G$ is always  a simple graph, i.e., no self-loops or multiple edges. However degree-zero vertices and  weight-one edges will be permitted since they are convenient in proofs.

\begin{definition}\label{def:specialtns}
Tensor network states associated to specific types of graphs are given special names. The most common ones are as follows:
\begin{enumerate}[\upshape (i)]
\item\label{tt} if $G$ is a path graph, then tensor network states associated to $G$ are variously called \emph{tensor trains} (\textsc{tt}) \cite{O}, \emph{linear tensor network} \cite{VanLoan}, \emph{concatenated tensor network states} \cite{HND}, \emph{Heisenberg chains} \cite{WH1993, White1992}, or \emph{matrix product states with open boundary conditions};
\item if $G$ is a star graph, then they are called \emph{star tensor network states} (\textsc{stns}) \cite{CMV};
\item if $G$ is a tree graph, then they are called \emph{tree tensor network states} (\textsc{ttns}) \cite{SDV} or \emph{hierarchical tensors} \cite{Hackbusch,HK2009,BSU};
\item\label{mps} if $G$ is a cycle graph, then they are called \emph{matrix product states} (\textsc{mps}) \cite{FNW,OR} or, more precisely, \emph{matrix product states with periodic boundary conditions};
\item if $G$ is a product of $d\ge 2$ path graphs, then they are called $d$-dimensional \emph{projected entangled pair states} (\textsc{peps}) \cite{VC,VWPC};
\item if $G$ is a complete graph, then they are called \emph{complete graph tensor network states} (\textsc{ctns}) \cite{MBRTV}.
\end{enumerate}
\end{definition}
We will use the term tensor trains as its acronym \textsc{tt} reminds us that they are a special case of tree tensor network states \textsc{ttns}. This is a matter of nomenclature convenience. As we can see from the references in \eqref{tt}, the notion has been rediscovered many times. The original sources for what we call tensor trains are \cite{WH1993, White1992} where they are called Heisenberg chains. In fact, as we have seen, tensor trains are also special cases of matrix product states. In some sources \cite{Orus}, the tensor trains in \eqref{tt} are called ``matrix product states with open boundary conditions'' and  the matrix product states in \eqref{mps} are called ``matrix product states with periodic conditions.''
\begin{figure}[h]
\centering
\begin{tikzpicture}
\filldraw
    (0,0) circle (2pt) node[align=center, below]{}
 --(1,0) circle (2pt)  node[align=center, below]{}
 -- cycle;
\end{tikzpicture}
\qquad
\begin{tikzpicture}
\filldraw
    (0,0) circle (2pt) node[align=center, below]{}
 --(1,0) circle (2pt)  node[align=center, below]{}
 --(2,0) circle (2pt) node[align=center, below]{}
 -- cycle;
\end{tikzpicture}
\qquad
\begin{tikzpicture}
\filldraw
    (0,0) circle (2pt) node[align=center, below]{}
 --(1,0) circle (2pt)  node[align=center, below]{}
 --(2,0) circle (2pt) node[align=center, below]{}
 --(3,0) circle (2pt)  node[align=center, below]{}
 -- cycle;
\end{tikzpicture}
\qquad
\begin{tikzpicture}
\filldraw
    (0,0) circle (2pt) node[align=center, below]{}
 --(1,0) circle (2pt)  node[align=center, below]{}
 --(2,0) circle (2pt) node[align=center, below]{}
 --(3,0) circle (2pt)  node[align=center, below]{}
 -- (4,0) circle (2pt)  node[align=center, below]{}
 -- cycle;
\end{tikzpicture}
\caption{Path graphs $P_2$, $P_3$, $P_4$, $P_5$. Each gives a \textsc{tt} decomposition.}
\label{fig:P}
\end{figure}
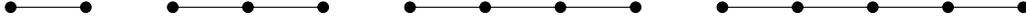
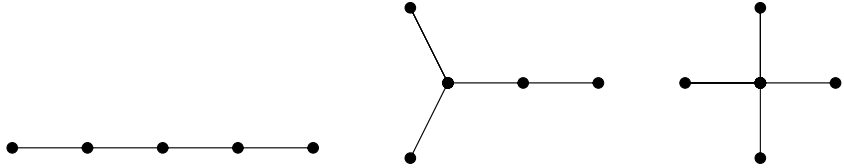
\begin{figure}[h]
\centering
\begin{tikzpicture}
\filldraw
    (0,0) circle (2pt) node[align=center, below]{}
 --(1,0) circle (2pt)  node[align=center, below]{}
 --(2,0) circle (2pt) node[align=center, below]{}
 --(3,0) circle (2pt)  node[align=center, below]{}
 -- (4,0) circle (2pt)  node[align=center, below]{}
 -- cycle;
\end{tikzpicture}
\qquad
\begin{tikzpicture}
\filldraw
    (0.5,-1) circle (2pt) node[align=center, left]{}
 --(1,0) circle (2pt)  node[align=center, left]{}  
 --(0.5,1) circle (2pt)  node[align=center, left]{}
    --(1,0) circle (2pt)  node[align=center, below]{}
 --(2,0) circle (2pt) node[align=center, below]{}
 --(3,0) circle (2pt)  node[align=center, below]{}
 -- cycle;
\end{tikzpicture}
\qquad
\begin{tikzpicture}
\filldraw
    (1,-1) circle (2pt) node[align=center, left] {}
 --(1,0) circle (2pt)  node[align=center, below] {}  
 --(0,0) circle (2pt)  node[align=center, below] {}  
 --(1,0) circle (2pt)  node[align=center, left] {}  
 --(1,1) circle (2pt)  node[align=center, left] {}
    --(1,0) circle (2pt)  node[align=center, below] {}
 --(2,0) circle (2pt) node[align=center, below] {}
 -- cycle;
\end{tikzpicture}
\caption{All five-vertex trees, including $P_5$ and $S_5$. Each gives a \textsc{ttns} decomposition.}
\label{fig:T}
\end{figure}

\begin{figure}[h]
\SetGraphUnit{3}
\SetVertexSimple[MinSize = 8pt]
\centering
\scalebox{0.5}{
\begin{tikzpicture}[rotate = 90]
  \Vertices{circle}{A,B,C}
  \Edges(A,B,C,A)
\end{tikzpicture}}
\qquad
\scalebox{0.5}{
\begin{tikzpicture}[rotate = 45]
  \Vertices{circle}{A,B,C,D}
  \Edges(A,B,C,D,A)
\end{tikzpicture}}
\qquad
\SetVertexSimple[MinSize = 10pt]
\scalebox{0.4}{
\begin{tikzpicture}[rotate = 18]
  \Vertices{circle}{A,B,C,D,E}
  \Edges(A,B,C,D,E,A)
\end{tikzpicture}}
\qquad
\scalebox{0.4}{
\begin{tikzpicture}[rotate = 0]
  \Vertices{circle}{A,B,C,D,E,F}
  \Edges(A,B,C,D,E,F,A)
\end{tikzpicture}}
\caption{Cycle graphs $C_3$, $C_4$, $C_5$, $C_6$. Each gives an \textsc{mps} decomposition.}
\label{fig:C}
\end{figure}
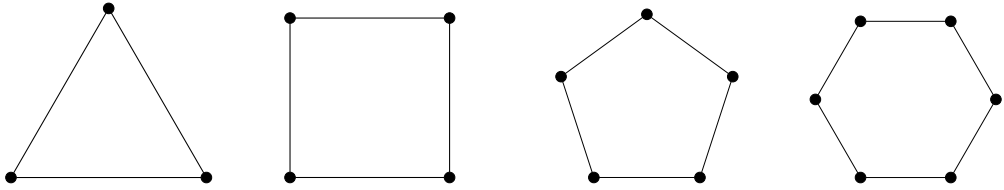

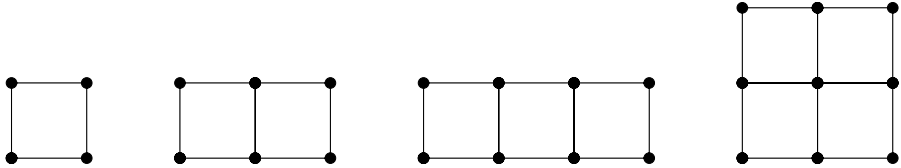
\begin{figure}[h]
\begin{tikzpicture}
\filldraw
    (0,0) circle (2pt) node[align=center, left]{}
 --(1,0) circle (2pt)  node[align=center, left]{}  
 --(1,1) circle (2pt)  node[align=center, left]{}
 --(0,1) circle (2pt)  node[align=center, left]{}
--    (0,0) circle (2pt) node[align=center, left]{}
 -- cycle;
\end{tikzpicture}
\qquad
\begin{tikzpicture}
\filldraw
    (0,0) circle (2pt) node[align=center, left]{}
 --(1,0) circle (2pt)  node[align=center, left]{}  
 --(2,0) circle (2pt)  node[align=center, left]{}
 --(2,1) circle (2pt)  node[align=center, left]{}
 --(1,1) circle (2pt)  node[align=center, left]{}
 --(1,0) circle (2pt)  node[align=center, left]{}
 --(1,1) circle (2pt)  node[align=center, left]{}
 --(0,1) circle (2pt)  node[align=center, left]{}
--    (0,0) circle (2pt) node[align=center, left]{}
 -- cycle;
\end{tikzpicture}
\qquad
\begin{tikzpicture}
\filldraw
    (0,0) circle (2pt) node[align=center, left]{}
 --(1,0) circle (2pt)  node[align=center, left]{}  
 --(2,0) circle (2pt)  node[align=center, left]{}
 --(3,0) circle (2pt)  node[align=center, left]{}
 --(3,1) circle (2pt)  node[align=center, left]{}
 --(2,1) circle (2pt)  node[align=center, left]{}
  --(2,0) circle (2pt)  node[align=center, left]{}
 --(2,1) circle (2pt)  node[align=center, left]{}
 --(1,1) circle (2pt)  node[align=center, left]{}
  --(1,0) circle (2pt)  node[align=center, left]{}
 --(1,1) circle (2pt)  node[align=center, left]{}
 --(0,1) circle (2pt)  node[align=center, left]{}
--  (0,0) circle (2pt) node[align=center, left]{}
 -- cycle;
\end{tikzpicture}
\qquad
\begin{tikzpicture}
\filldraw
    (0,0) circle (2pt) node[align=center, left]{}
 --(1,0) circle (2pt)  node[align=center, left]{}  
 --(2,0) circle (2pt)  node[align=center, left]{}
 --(2,1) circle (2pt)  node[align=center, left]{}
 --(2,2) circle (2pt)  node[align=center, left]{}
 --(1,2) circle (2pt)  node[align=center, left]{}
 --(1,1) circle (2pt)  node[align=center, left]{}
  --(2,1) circle (2pt)  node[align=center, left]{}
 --(0,1) circle (2pt)  node[align=center, left]{}
 --(1,1) circle (2pt)  node[align=center, left]{}
  --(1,0) circle (2pt)  node[align=center, left]{}
 --(1,1) circle (2pt)  node[align=center, left]{}
  --(1,2) circle (2pt)  node[align=center, left]{}
  --(0,2) circle (2pt)  node[align=center, left]{}
    --(0,1) circle (2pt)  node[align=center, left]{}
  --(0,0) circle (2pt)  node[align=center, left]{}
 -- cycle;
\end{tikzpicture}
\qquad
\caption{Four products of path graphs. Each gives a \textsc{peps} decomposition.}
\end{figure}

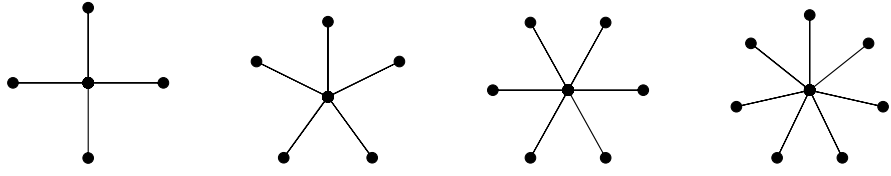
\begin{figure}[h]
\begin{tikzpicture}
\filldraw
    (0,0) circle (2pt) node[align=center, left]{}
 --(1,0) circle (2pt)  node[align=center, left]{}  
 --(0,0) circle (2pt) node[align=center, left]{}
 --(-1,0) circle (2pt)  node[align=center, left]{}  
  --(0,0) circle (2pt) node[align=center, left]{}
   --(0,1) circle (2pt) node[align=center, left]{}
 --(0,0) circle (2pt) node[align=center, left]{}
 --(0,-1) circle (2pt) node[align=center, left]{}
-- cycle;
\end{tikzpicture}
\qquad
\begin{tikzpicture}[rotate=90]
\filldraw
    (0,0) circle (2pt) node[align=center, left]{}
 --(1,0) circle (2pt)  node[align=center, left]{}  
  --(0,0) circle (2pt)  node[align=center, left]{}  
 --(0.47,0.95) circle (2pt)  node[align=center, left]{}  
 --(0,0) circle (2pt) node[align=center, left]{}
 --(-0.81,0.588) circle (2pt)  node[align=center, left]{}  
  --(0,0) circle (2pt) node[align=center, left]{}
 --(-0.81,-0.588) circle (2pt)  node[align=center, left]{}  
   --(0,0) circle (2pt)  node[align=center, left]{}  
 --(0.47,-0.95) circle (2pt)  node[align=center, left]{}  
 --(0,0) circle (2pt)  node[align=center, left]{}  
-- cycle;
\end{tikzpicture}
\qquad
\begin{tikzpicture}
\filldraw
    (0,0) circle (2pt) node[align=center, left]{}
 --(1,0) circle (2pt)  node[align=center, left]{}  
  --(0,0) circle (2pt)  node[align=center, left]{}  
 --(1/2,0.9) circle (2pt)  node[align=center, left]{}  
 --(0,0) circle (2pt) node[align=center, left]{}
 --(-1/2,0.9) circle (2pt)  node[align=center, left]{}  
  --(0,0) circle (2pt) node[align=center, left]{}
 --(-1,0) circle (2pt)  node[align=center, left]{}  
   --(0,0) circle (2pt)  node[align=center, left]{}  
 --(-1/2,-0.9) circle (2pt)  node[align=center, left]{}  
 --(0,0) circle (2pt)  node[align=center, left]{}  
 --(1/2,-0.9) circle (2pt)  node[align=center, left]{}  
-- cycle;
\end{tikzpicture}
\qquad
\begin{tikzpicture}[rotate = 90]
\filldraw
    (0,0) circle (2pt) node[align=center, left]{}
 --(1,0) circle (2pt)  node[align=center, left]{}  
 --(0,0) circle (2pt)  node[align=center, left]{}  
  --(0.623,0.782) circle (2pt)  node[align=center, left]{}  
 --(0,0) circle (2pt)  node[align=center, left]{}  
 --(-0.22,0.975) circle (2pt) node[align=center, left]{}
 --(0,0) circle (2pt)  node[align=center, left]{}  
  --(-0.9,-0.4337) circle (2pt)  node[align=center, left]{}  
   --(0,0) circle (2pt)  node[align=center, left]{}  
 --(-0.9,0.4337) circle (2pt) node[align=center, left]{}
 --(0,0) circle (2pt)  node[align=center, left]{}  
  --(-0.22,-0.975) circle (2pt)  node[align=center, left]{}  
  --(0,0) circle (2pt)  node[align=center, left]{}  
    --(0.623,-0.782) circle (2pt)  node[align=center, left]{}  
-- cycle;
\end{tikzpicture}
\caption{Star graphs $S_5,S_6,S_7, S_8$. Each gives a \textsc{stns} decomposition.}
\label{fig: star graph}
\end{figure}

\begin{figure}[h]
\SetGraphUnit{3}
\SetVertexSimple[MinSize = 8pt]
\centering
\scalebox{0.5}{
\begin{tikzpicture}[rotate = 45]
  \Vertices{circle}{A,B,C,D}
  \Edges(A,B,C,D,A,C,D,B)
\end{tikzpicture}}
\qquad
\SetVertexSimple[MinSize = 10pt]
\scalebox{0.4}{
\begin{tikzpicture}[rotate = 18]
  \Vertices{circle}{A,B,C,D,E}
  \Edges(A,B,C,D,E,A,C,E,B,D,A)
\end{tikzpicture}}
\qquad
\SetVertexSimple[MinSize = 10pt]
\scalebox{0.4}{
\begin{tikzpicture}[rotate = 18]
  \Vertices{circle}{A,B,C,D,E,F}
  \Edges(A,B,C,D,E,F,A,C,E,A,D,B,F,C,A,D,F)
\end{tikzpicture}}
\qquad
\SetVertexSimple[MinSize = 10pt]
\scalebox{0.4}{
\begin{tikzpicture}[rotate = 18]
  \Vertices{circle}{A,B,C,D,E,F,G}
  \Edges(A,B,C,D,E,F,G,A,C,E,G,B,D,F,A,D,G,C,F,B,E,A)
\end{tikzpicture}}
\caption{Complete graphs $K_4$, $K_5$, $K_6$, $K_7$. Each gives a \textsc{ctns} decomposition.}
\label{fig:K}
\end{figure}
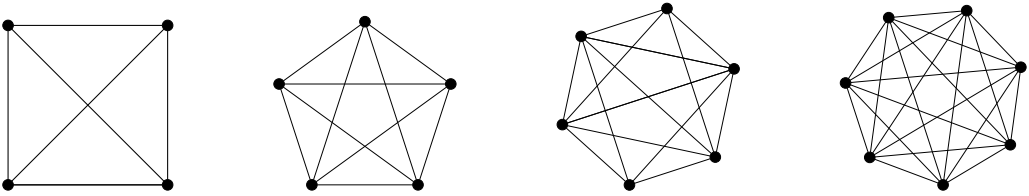

To illustrate Definition~\ref{def:tns} for readers unfamiliar with multilinear algebraic manipulations, we will work out the \textsc{mps} decomposition for the $3$-vertex graph in Figure~\ref{fig:networks} and  the \textsc{peps} decomposition from the $6$-vertex graph in Figure~\ref{fig:networks} in full details.
\begin{example}[\textsc{mps}]\label{eg:mps}
Let $C_3$ be the $3$-vertex  cycle graph for \textsc{mps} in Figure~\ref{fig:networks}. We attach vector spaces $\mathbb{A},\mathbb{B}, \mathbb{C}$ to the vertices labeled $x,y,z$ respectively and vector spaces $\mathbb{D},\mathbb{E}, \mathbb{F}$ to the edges labeled $i,j,k$ respectively.  $\tns(C_3; \mathbb{D},\mathbb{E},\mathbb{F}; \mathbb{A},\mathbb{B},\mathbb{C})$, the set of \textsc{mps} tensor network states corresponding to $C_3$, is obtained as follows. First, assign arbitrary directions to the edges, say, $x \xrightarrow{j} y \xrightarrow{k} z \xrightarrow{i} x$. Next, consider tensors
\[
T_1 \in \mathbb{D} \otimes \mathbb{A} \otimes  \mathbb{E}^*,\qquad
T_2 \in \mathbb{E} \otimes \mathbb{B} \otimes \mathbb{F}^*,\qquad
T_3 \in \mathbb{F} \otimes \mathbb{C} \otimes \mathbb{D}^*.
\]
An  \textsc{mps} tensor network state is obtained by contracting factors in $\mathbb{D}$, $\mathbb{E}$, $\mathbb{F}$ with those in $\mathbb{D}^*$, $\mathbb{E}^*$, $\mathbb{F}^*$ respectively,  giving us $\kappa_{C_3} (T_1\otimes T_2 \otimes T_3)\in \mathbb{A} \otimes \mathbb{B} \otimes \mathbb{C}$. Let
\[
\dim \mathbb{A}= n_1,\quad \dim \mathbb{B}= n_2,\quad \dim \mathbb{C}= n_3,\quad \dim \mathbb{D}= r_1,\quad \dim \mathbb{E}= r_2,\quad \dim \mathbb{F} = r_3.
\]
Let $\{d_1,\dots, d_{r_1}\}$, $\{e_1,\dots, e_{r_2}\}$, $\{f_1,\dots, f_{r_3}\}$ be bases on $\mathbb{D}$, $\mathbb{E}$, $\mathbb{F}$; and  $\{d_1^* ,\dots, d_{r_1}^*\}$, $\{e_1^*,\dots, e^*_{r_2}\}$, $\{f_1^* ,\dots, f^*_{r_3}\}$ be the corresponding dual bases on $\mathbb{D}^*$, $\mathbb{E}^*$, $\mathbb{F}^*$. Then
\[
T_1 = \sum_{i, j = 1}^{r_1, r_2} d_i \otimes a_{i j} \otimes e^*_j, \quad
T_2 = \sum_{j, k = 1}^{r_2, r_3} e_j \otimes b_{j k} \otimes f^*_k, \quad
T_3 = \sum_{k, i = 1}^{r_3, r_1} f_k \otimes c_{k i} \otimes d^*_i.
\]
We will derive the expression for $T_1$ for illustration: Let $a_1,\dots,a_{n_1}$  be a basis of $\mathbb{A}$. Then a tensor in $\mathbb{D}\otimes \mathbb{A}\otimes \mathbb{E}^*$ has the form
\[
T_1 = \sum_{i,k,j = 1}^{r_1,n_1,r_2} \alpha_{i k j }d_i \otimes a_k \otimes e^*_j,
\]
for some coefficients $\alpha_{i k j }\in \mathbb{C}$. We may then express $T_1$ as 
\[
T_1 = \sum_{i, j = 1}^{r_1,r_2} d_i \otimes \Bigl(\sum_{k=1}^{n_1} \alpha_{i k j }a_k \Bigr)\otimes e^*_j  = \sum_{i,j = 1}^{r_1,r_2} d_i \otimes a_{i j} \otimes e^*_j,
\]
where $a_{i j } \coloneqq \sum_{k=1}^{n_1} \alpha_{i k j }a_k $.  Finally we obtain the \textsc{mps} decomposition as
\begin{align*}
\kappa_{C_3} (T_1 \otimes T_2 \otimes T_3) &= \sum_{i,i',j,j',k,k' =1}^{r_1,r_1,r_2,r_2,r_3,r_3}\kappa \bigl( (d_i \otimes a_{i j}  \otimes e^*_j) \otimes (e_{j'} \otimes b_{j' k} \otimes f^*_{k}) \otimes (f_{k'} \otimes c_{k' i'} \otimes d^*_{i'}) \bigr)\\
& =  \sum_{i,i',j,j',k,k' =1}^{r_1,r_1,r_2,r_2,r_3,r_3}  d^*_{i'} (d_i)\, e^*_j(e_{j'})\, f^*_k(f_{k'})  \cdot a_{i j}  \otimes b_{j' k} \otimes c_{k' i'}  \\
& = \sum_{i,i',j,j',k,k' =1}^{r_1,r_1,r_2,r_2,r_3,r_3}  (\delta_{i,i'} \delta_{j,j'} \delta_{k,k'} ) \cdot a_{i j}  \otimes b_{j' k} \otimes c_{k' i'} 
 = \sum_{i,j,k =1}^{r_1,r_2,r_3} a_{i j}  \otimes b_{j k} \otimes c_{k i },
\end{align*}
where $\delta_{i,i'}$ denotes the Kronecker delta.
\end{example}

\begin{example}[\textsc{peps}]
Let $G$ be the $6$-vertex graph for \textsc{peps} in Figure~\ref{fig:networks}. We attach vector spaces $\mathbb{V}_1,\dots, \mathbb{V}_6$ to the vertices labeled $x,y,z,u,v,w$ and vector spaces $\mathbb{E}_1,\dots, \mathbb{E}_7$ to the edges labeled $i,j,k,l,m,n,o$ respectively.  $\tns(G; \mathbb{E}_1,\dots,\mathbb{E}_7; \mathbb{V}_1,\dots,\mathbb{V}_6)$, the set of \textsc{peps} tensor network states, is obtained as follows. First, assign arbitrary directions to the edges, say, $x \xrightarrow{j} y \xrightarrow{l} z \xrightarrow{m} u \xrightarrow{n} v \xrightarrow{k} y$ and $v \xrightarrow{o} w \xrightarrow{i} x$. Next, consider tensors $T_i\in \bigl(\bigotimes_{j\in \In(i)} \mathbb{E}_j \bigr) \otimes \mathbb{V}_i \otimes \bigl( \bigotimes_{j\in \Out(i)} \mathbb{E}^*_j\bigr)$, i.e.,
\begin{align*}
T_1 &\in \mathbb{E}_1\otimes \mathbb{V}_1\otimes  \mathbb{E}^*_2,&
T_2 &\in (\mathbb{E}_2 \otimes \mathbb{E}_3) \otimes \mathbb{V}_2\otimes \mathbb{E}^*_4,&
T_3 &\in \mathbb{E}_4 \otimes \mathbb{V}_3\otimes \mathbb{E}^*_5,\\
T_4 &\in \mathbb{E}_5 \otimes \mathbb{V}_4\otimes \mathbb{E}^*_6,&
T_5 &\in \mathbb{E}_6 \otimes \mathbb{V}_5\otimes (\mathbb{E}^*_3\otimes \mathbb{E}^*_7),&
T_6 &\in \mathbb{E}_7 \otimes \mathbb{V}_6\otimes \mathbb{E}^*_1.
\end{align*}
Finally, we contract factors in $\mathbb{E}_j$ with those in $\mathbb{E}_j^*$, $j =1,\dots,6$,  giving us $\kappa_G (T_1\otimes \dots \otimes T_6)\in \mathbb{V}_1 \otimes \dots \otimes \mathbb{V}_6$. If we choose bases on  $\mathbb{E}_1,\dots,\mathbb{E}_7$, then we obtain the expression for a \textsc{peps} tensor network state in coordinates,
\[
\kappa_G (T_1\otimes \dots \otimes T_6) =\sum_{i,j,k,l,m,n,o=1}^{r_1,r_2,r_3,r_4,r_5,r_6,r_7} a_{ij} \otimes b_{jkl} \otimes c_{lm} \otimes d_{mn} \otimes e_{nko} \otimes f_{oi},
\]
as in Example~\ref{eg:mps}; here $r_j = \dim \mathbb{E}_j$.
\end{example}

We end this section with a simple observation.
\begin{proposition}
For any undirected graph $G$ with $d$ vertices and $c$ edges,
$\tns(G; r_1,\dots,r_c; \mathbb{V}_1,\dots, \mathbb{V}_d)$ is an irreducible constructible subset of $\mathbb{V}_1\otimes \dots \otimes \mathbb{V}_d$.
\end{proposition}
\begin{proof}
Let $\mathbb{U}_i = \mathbb{I}_i \otimes \mathbb{V}_i \otimes \mathbb{O}_i$ for $i =1,\dots, d$. Let $X$ be the irreducible variety of decomposable tensors in $\mathbb{U}_1 \otimes \dots \otimes \mathbb{U}_d $, i.e., 
$X = \{  T_1 \otimes \dots \otimes T_d \in \mathbb{U}_1 \otimes \dots \otimes \mathbb{U}_d  :  \; T_i \in  \mathbb{U}_i,\; i=1,\dots,d \}$.
Since $X$ is irreducible and $\kappa_G$ is a morphism between two varieties, the image 
$\kappa_G(X) = \tns(G; r_1,\dots,r_c; \mathbb{V}_1,\dots, \mathbb{V}_d)$
must be irreducible and constructible.
\end{proof}
The proof of this proposition also reveals the following illuminating insight, which in retrospect should have been obvious from \eqref{eq:con} and Definition~\ref{def:tns}.
\begin{corollary}
Every tensor network state is a tensor contraction of a rank-one tensor. 
\end{corollary}

\section{Calculus of tensor networks}\label{sec:calculus}

Let $\mathbb{N}$ and $\mathbb{N}_0$ denote the set of positive and nonnegative integers respectively. We will introduce some basic tools for manipulating tensor network states. We begin by introducing the notion of criticality, which will allow for various reductions of tensor network states.
\begin{definition}
Let $\tns(G; r_1,\dots,r_c; n_1,\dots,n_d)$ be a tensor network and the notations be as in Definition~\ref{def:tns}. Set
\[
m_i \coloneqq \prod_{j\in \In(i)\cup \Out(i)} r_j,\qquad i =1,\dots,d.
\]
A vertex $i \in V$  is called \emph{subcritical} if $n_i < m_i$, \emph{critical} if $n_i = m_i$, and \emph{supercritical} if $n_i > m_i$.
We say that $\tns(G; r_1,\dots,r_c; n_1,\dots,n_d)$ is 
\begin{enumerate}[\upshape (i)]
\item \emph{subcritical} if $n_i \le m_i$ for all $i=1,\dots,d$, and at least one inequality is strict;
\item \emph{critical} if $n_i = m_i$ for all $i=1,\dots,d$;
\item \emph{supercritical} if $n_i \ge m_i$ for all $i=1,\dots,d$, and at least one inequality is strict.
\end{enumerate}
\end{definition}

Let $\mathbb{V}$ be a $n$-dimensional vector space. For $ k =1,\dots, n$, we let $\Gr(k,\mathbb{V})$ denote the Grassmannian of $k$-dimensional subspaces of $\mathbb{V}$. For the special case, $\mathbb{V} = \mathbb{C}^n$, we write $\Gr(k,n)$ for the Grassmannian of $k$-planes in $\mathbb{C}^n$.

Let $(r_1,\dots,r_c) \in \mathbb{N}^c$ and other notations be as in Definition~\ref{def:tns}. The \emph{tautological vector bundle} on $\Gr(k,\mathbb{V})$, denoted $\mathcal{S}$, is the vector bundle whose base space is $\Gr(k,\mathbb{V})$ and whose fiber over $[\mathbb{W}]\in \Gr(k,\mathbb{V})$ is simply the $k$-dimensional linear subspace $\mathbb{W} \subseteq \mathbb{V}$.  For any $k_1,\dots,k_c \in \mathbb{N}$,  the \emph{tensor network bundle}, denoted
\[
\tns(G;r_1,\dots,r_c;\mathcal{S}_1,\dots, \mathcal{S}_d),
\]
is the fiber bundle over the base space $\Gr(k_1,\mathbb{V}_1)\times \dots \times \Gr(k_d,\mathbb{V}_d)$ whose fiber over a point $([\mathbb{W}_1], \dots,[\mathbb{W}_d])\in \Gr(k_1,\mathbb{V}_1)\times \dots \times \Gr(k_d,\mathbb{V}_d)$ is $ \tns(G; r_1,\dots,r_c; \mathbb{W}_1,\dots, \mathbb{W}_d)$.

We will need the following results from \cite[Propositions~3 and 4]{LQY}, reproduced here for easy reference.
\begin{proposition}[Reduction of degree-one subcritical vertices]\label{prop:reduction valence one}
Let $ (r_1,\dots, r_c) \in \mathbb{N}^c$ and $G = (V,E)$ be a graph. Let $i \in V$ be a vertex of degree one adjacent to the vertex $j \in V$. If $i$ is subcritical or critical, then we have the following reduction:
\[
\tns(G;r_1,r_2,\dots,r_c;\mathbb{V}_1,\mathbb{V}_2,\mathbb{V}_3,\dots, \mathbb{V}_d) = \tns(G'; r_2,\dots,r_c; \mathbb{V}_1\otimes \mathbb{V}_2, \mathbb{V}_3,\dots, \mathbb{V}_d),
\]
where $G' = (V\setminus \{i\}, E \setminus \{i,j\})$, i.e., the graph obtained by removing the vertex $i$ and edge $\{i,j\}$ from $G$, and $ (r_2,\dots,r_c) \in \mathbb{N}^{c-1}$. Alternatively, we may write
\[
\tns(G;r_1,r_2,\dots,r_c;n_1,n_2, n_3,\dots, n_d) = \tns(G'; r_2,\dots,r_c; n_1n_2, n_3,\dots, n_d).
\]
\end{proposition}

\begin{proposition}[Reduction of supercritical vertices]\label{prop:reduction supercritical}
Let $n_i = \dim \mathbb{V}_i$, $p_i = \min \{n_i, m_i\}$,  and $\mathcal{S}_i$ be the tautological vector bundle on $\Gr(p_i,\mathbb{V}_i)$ for $i=1,\dots,d$.
Then the map
\[
\pi : \tns(G;r_1,\dots,r_c;\mathcal{S}_1,\dots, \mathcal{S}_d) \to  \tns(G;r_1,\dots,r_c;\mathbb{V}_1,\dots, \mathbb{V}_d), \quad ([\mathbb{W}_1],\dots, [\mathbb{W}_d],T) \mapsto T,
\]
is a surjective birational map.
\end{proposition}
Immediate consequences of Proposition~\ref{prop:reduction supercritical} are a bound on the multilinear rank of tensor network states and a reduction formula for the dimension of a tensor network.
\begin{corollary}\label{cor:dimension supercritical}
Let the notations be as above. Then for any $ (r_1,\dots, r_c) \in \mathbb{N}^c$, we have
\[
\tns(G; r_1,\dots,r_c; \mathbb{V}_1,\dots, \mathbb{V}_d) \subseteq \Sub_{p_1,\dots,p_d}(\mathbb{V}_1,\dots, \mathbb{V}_d)
\] 
and
\[
\dim \tns(G; r_1,\dots,r_c; n_1,\dots, n_d) = \dim \tns(G; r_1,\dots,r_c; p_1,\dots, p_d) + \sum_{i=1}^d p_i(n_i-p_i). 
\]
\end{corollary}
Note that by Proposition~\ref{prop:reduction supercritical}, all tensor network states can be reduced to one that is either critical or subcritical.

The next proposition is useful for describing when we are allowed to remove an edge from the graph while keeping the tensor network unchanged. The reader may notice a resemblance to  Proposition~\ref{prop:reduction valence one}, which is about collapsing two vertices into one and thus results in a reduction in the total number of vertices, but the goal of Proposition~\ref{prop:remove an edge} is to remove an edge while leaving the total number of vertices  unchanged.
\begin{proposition}[Edge removal]\label{prop:remove an edge}
Let $G = (V,E)$ be a graph with $d$ vertices and $c$ edges. Let $(r_1,\dots, r_c) \in \mathbb{N}^c$ and $(n_1,\dots, n_d) \in \mathbb{N}^d$. Suppose that the edge $e  \in E$ has weight $r_1 = 1$ and $G' = (V,E\setminus \{e\})$ is the graph obtained by removing the edge $e$ from $G$ and suppose $G'$ has no isolated vertices. Then
\[
\tns(G;1,r_2,\dots, r_d;\mathbb{V}_1,\dots, \mathbb{V}_d) \simeq \tns(G';r_2,\dots,r_d; \mathbb{V}_1,\dots, \mathbb{V}_d). 
\]
\end{proposition}
\begin{proof}
Assume  without loss of generality that $e = \{1, 2\}$. By definition, 
\begin{align*}
\tns(G;1,r_2,\dots, r_d;\mathbb{V}_1,\dots, \mathbb{V}_d) &=\{ \kappa_{G} (T_1\otimes \cdots \otimes T_d): T_i \in \mathbb{I}_i \otimes \mathbb{V}_i \otimes \mathbb{O}_i, \; i =1,\dots,d\}, \\
\tns(G';r_2,\dots, r_d;\mathbb{V}_1,\dots, \mathbb{V}_d) &= \{ \kappa_{G'} (T'_1\otimes \cdots \otimes T'_d): T'_i \in \mathbb{I}'_i \otimes \mathbb{V}_i \otimes \mathbb{O}'_i, \; i =1,\dots,d\},
\end{align*}
where $\mathbb{I}_i$, $\mathbb{I}'_i$ are inspaces, $\mathbb{O}_i$, $\mathbb{O}'_i$ are outspaces, as defined in \eqref{eq:inout}, and $\kappa_G$, $\kappa_{G'}$ are contraction maps, as defined in \eqref{eq:con}, associated to $G$ and $G'$ respectively. Since $r_1 = 1$,  $\mathbb{E}_1 \simeq \mathbb{C}$  and so contributes nothing\footnote{A one-dimensional vector space is isomorphic to the field of scalars $\mathbb{C}$ and $\mathbb{C} \otimes \mathbb{E} = \mathbb{E}$ for any complex  vector space $\mathbb{E}$.}  to the factors $\mathbb{I}_1 \otimes \mathbb{V}_1 \otimes \mathbb{O}_1$, $\mathbb{I}_2 \otimes \mathbb{V}_2 \otimes \mathbb{O}_2$, and thus $\mathbb{I}_i \otimes \mathbb{V}_i \otimes \mathbb{O}_i \simeq \mathbb{I}'_i \otimes \mathbb{V}_i \otimes \mathbb{O}'_i$ for $i = 1, 2$.
On the other hand,  $\mathbb{I}_i \otimes \mathbb{V}_i \otimes \mathbb{O}_i = \mathbb{I}'_i \otimes \mathbb{V}_i \otimes \mathbb{O}'_i$ for $i =3,\dots,d$. Therefore the images of the contraction map must be isomorphic, as required.
\end{proof}
The assumption that an isolated vertex does not arise in $G'$ upon removing the edge $e$ is necessary because of \eqref{eq:isolate}.  An immediate consequence of Proposition~\ref{prop:remove an edge} is that tensor trains are  a special case of matrix product states since
\begin{equation}\label{eq:tt=mps}
\tns(C_d;r_1,r_2,\dots,r_{d-1},1;n_1,\dots,n_d) =  \tns(P_d ;r_1,\dots,r_{d-1};n_1,\dots,n_d),
\end{equation}
where $C_d$ is the cycle graph with $d$ vertices, the edge with weight $1$ is adjacent to the vertex $1$ and $d$, and $P_d$ is the path graph with $d$ vertices. 

We end the section with a result about restriction of tensor network states to subspaces of tensors, which will be crucial for an important property of tensor network rank established in Theorem~\ref{thm:$G$-rank subspace}.
\begin{lemma}[Restriction lemma]\label{lemma: subspace TNS}
Let $G$ be a graph with $d$ vertices and $c$ edges. Let $\mathbb{V}_1,\dots, \mathbb{V}_d$ be vector spaces and $\mathbb{W}_i \subseteq \mathbb{V}_i$ be subspaces, $i=1,\dots, d$. Then for any $ (r_1,\dots,r_c) \in \mathbb{N}^c$,
\begin{equation}\label{eq:subsp}
\tns(G; r_1,\dots, r_c; \mathbb{V}_1,\dots, \mathbb{V}_d) \cap \mathbb{W}_1\otimes \dots \otimes \mathbb{W}_d = \tns(G; r_1,\dots, r_c; \mathbb{W}_1,\dots, \mathbb{W}_d).
\end{equation}
In particular, we always have
\begin{equation}\label{eq:subsp2}
 \tns(G; r_1,\dots, r_c; \mathbb{W}_1,\dots, \mathbb{W}_d) \subseteq \tns(G; r_1,\dots, r_c; \mathbb{V}_1,\dots, \mathbb{V}_d).
\end{equation}
\end{lemma}
\begin{proof}
It is obvious that `$\supseteq$' holds in \eqref{eq:subsp} and it remains to show `$\subseteq$'.
Let $\mathbb{E}_j$ be a vector space of dimension $r_j$,  $j=1,\dots, c$. Orient $G$ arbitrarily and let the inspace $\mathbb{I}_i$ and outspace $\mathbb{O}_i$ be as defined  in \eqref{eq:inout} for vertices $i =1,\dots,d$.
We obtain two commutative diagrams:
\[
\begin{tikzcd}
\prod_{i=1}^d \mathbb{I}_{i} \otimes \mathbb{W}_i \otimes 
\mathbb{O}_{i}  \arrow{r}{\Psi'} \arrow[swap]{d}{\kappa''} &\bigotimes_{i=1 }^d \mathbb{I}_{i} \otimes \mathbb{W}_i \otimes \mathbb{O}_{i} \arrow{r}{\Phi} \arrow[swap]{d}{\kappa'} 
&\bigotimes_{i=1 }^d \mathbb{I}_{i} \otimes \mathbb{V}_i \otimes \mathbb{O}_{i}  \arrow{d}{\kappa}  \\
\tns(G; r_1,\dots, r_c; \mathbb{W}_1,\dots, \mathbb{W}_d) \arrow{r}{\psi'}  &\mathbb{W}_1\otimes \dots \otimes \mathbb{W}_d \arrow{r}{\phi} & \mathbb{V}_1\otimes \dots \otimes \mathbb{V}_d
\end{tikzcd}
\]
and
\[
\begin{tikzcd}
\prod_{i=1}^d  \mathbb{I}_{i} \otimes \mathbb{V}_i \otimes \mathbb{O}_{i}  \arrow{r}{\Psi} \arrow[swap]{d}{\kappa'''} &\bigotimes_{i=1 }^d \mathbb{I}_{i} \otimes \mathbb{V}_i \otimes \mathbb{O}_{i}\arrow[swap]{d}{\kappa} \\
\tns(G; r_1,\dots, r_c; \mathbb{V}_1,\dots, \mathbb{V}_d) \arrow{r}{\psi}  & \mathbb{V}_1\otimes \dots \otimes \mathbb{V}_d,
\end{tikzcd}
\]
where $\Psi'$ sends $(x_1,\dots, x_d)$, $x_i \in \mathbb{I}_{i} \otimes \mathbb{W}_i \otimes 
\mathbb{O}_{i} ,i=1,\dots, d$, to  $x_1\otimes \cdots \otimes x_d$ and $\Psi$ is defined similarly;
$\psi',\psi$ are inclusions of tensor network states into their respective ambient spaces; $\Phi,\phi$ are inclusions induced by $\mathbb{W}_i \subseteq \mathbb{V}_i$, $i=1,\dots, d$; $\kappa'$ is the restriction of $\kappa \coloneqq \kappa_G$; $\kappa''$ the composition of $\kappa'$ and $\Psi'$; and $\kappa'''$ the composition of $\kappa$ and $\Psi$.

For each $i=1,\dots, d$, write $\mathbb{V}_i^+ = \mathbb{W}_i$ and decompose $\mathbb{V}_i$ into a direct sum
\[
\mathbb{V}_i = \mathbb{V}_i^- \oplus \mathbb{V}_i^+
\]
for some linear subspace $\mathbb{V}_i^- \subseteq \mathbb{V}_i$. These give us the decomposition
\[
\prod_{i=1}^d \mathbb{I}_{i} \otimes \mathbb{V}_i \otimes \mathbb{O}_{i} = \prod_{i=1}^d  \mathbb{I}_{i} \otimes (\mathbb{V}^-_i\oplus \mathbb{V}^+_i) \otimes \mathbb{O}_{i}
= \prod_{i=1}^d (\mathbb{I}_{i} \otimes \mathbb{V}^-_i \otimes \mathbb{O}_{i}) \oplus (\mathbb{I}_{i} \otimes \mathbb{V}^+_i \otimes \mathbb{O}_{i}).
\]
Now we may write an element $T_1 \otimes \dots \otimes T_d \in \Psi \bigl(\prod_{i=1}^d \mathbb{I}_{i} \otimes \mathbb{V}_i \otimes \mathbb{O}_{i}\bigr)$ as 
\[
T_1 \otimes \dots \otimes T_d = (T_1^- + T_1^+) \otimes \dots \otimes (T_d^- + T_d^+) = \sum_{\pm } T_1^{\pm} \otimes \dots\otimes T_d^{\pm}, 
\]
where $T_i^- \in \mathbb{I}_{i} \otimes \mathbb{V}^-_i \otimes \mathbb{O}_{i}$ and $ T_i^+ \in  \mathbb{I}_{i} \otimes \mathbb{V}^+_i \otimes \mathbb{O}_{i}$.
Since $T_1^{\pm} \otimes \dots \otimes T_d^{\pm} \in \Psi\bigl(\prod_{i=1}^d \mathbb{I}_{i} \otimes \mathbb{V}^{\pm}_i \otimes \mathbb{O}_{i}\bigr)$,
\[
\kappa(T_1 \otimes \dots \otimes T_d)\in \sum_{\pm } \kappa\Bigl(\Psi\Bigl(\prod_{i=1}^d \mathbb{I}_{i} \otimes \mathbb{V}^{\pm}_i \otimes \mathbb{O}_{i}\Bigr)\Bigr) = \sum_{\pm } \tns(G; r_1,\dots,r_c; \mathbb{V}^{\pm}_1,\dots, \mathbb{V}^{\pm}_d).
\]
Therefore $\kappa(T_1 \otimes \dots \otimes T_d) \in  \mathbb{V}^+_1\otimes \dots \otimes \mathbb{V}^+_d$ implies that $T_i\in \mathbb{V}^+_i = \mathbb{W}_i$ for all $i=1,\dots, d$ and hence $\kappa(T_1 \otimes \dots \otimes T_d)\in \tns(G; r_1,\dots,r_c; \mathbb{W}_1,\dots, \mathbb{W}_d)$.
\end{proof}

\section{$G$-ranks of tensors}\label{sec:Grank}

The main goal of this article is to show that there is a natural notion of rank for tensor network with respect to any connected graph $G$. We start by reminding our readers of the classical notions of tensor rank and multilinear rank, with a small twist   --- instead of first defining tensor and multilinear ranks and then defining the respective sets they cut out, i.e., secant quasiprojective variety and subspace variety, we will reverse the order of these definitions. This approach will be consistent with how we define tensor network ranks later. The results in this and subsequent sections require that the vector spaces $\mathbb{V}_1,\dots,\mathbb{V}_d$ be finite-dimensional.

The \emph{Segre variety} is the set of all decomposable tensors,
\[
\Seg(\mathbb{V}_1, \dots,\mathbb{V}_d) \coloneqq \{ T \in \mathbb{V}_1 \otimes \dots \otimes \mathbb{V}_d : T = v_1 \otimes \dots \otimes v_d,\; v_i \in \mathbb{V}_i\} .
\]
The $r$-\emph{secant quasiprojective variety} of the Segre variety is
\[
s_r\bigl(\Seg(\mathbb{V}_1, \dots,\mathbb{V}_d)\bigr) \coloneqq \Bigl\{T \in \mathbb{V}_1 \otimes \dots \otimes \mathbb{V}_d :  T = \sum_{i=1}^r T_i , \; T_i \in \Seg(\mathbb{V}_1, \dots,\mathbb{V}_d) \Bigr\},
\]
and its closure is the $r$-\emph{secant variety} of the Segre variety,
\[
\sigma_r\bigl(\Seg(\mathbb{V}_1, \dots,\mathbb{V}_d)\bigr) \coloneqq \overline{
s_r\bigl(\Seg(\mathbb{V}_1, \dots,\mathbb{V}_d)\bigr)}.
\]
The $(r_1,\dots,r_d)$-\emph{subspace variety} \cite{L} is the set
\[
\Sub_{r_1,\dots,r_d}(\mathbb{V}_1,\dots, \mathbb{V}_d) \coloneqq \{T \in  \mathbb{V}_1 \otimes \dots \otimes \mathbb{V}_d :  T\in \mathbb{W}_1\otimes \dots \otimes \mathbb{W}_d, \; \mathbb{W}_i \subseteq \mathbb{V}_i, \; \dim \mathbb{W}_i  = r_i  \}.
\]
The \emph{tensor rank} or just rank \cite{Hi1} of a tensor $T \in \mathbb{V}_1 \otimes \dots \otimes \mathbb{V}_d$ is 
\[
\rank(T) \coloneqq \min \bigl\{ r\in \mathbb{N}_0 : T  \in 
s_r\bigl(\Seg(\mathbb{V}_1, \dots,\mathbb{V}_d)\bigr)  \bigr\},
\]
its \emph{border rank} \cite{L} is
\[
\brank(T) \coloneqq \min \bigl\{ r\in \mathbb{N}_0 : T  \in 
\sigma_r\bigl(\Seg(\mathbb{V}_1, \dots,\mathbb{V}_d)\bigr)  \bigr\},
\]
and its \emph{multilinear rank} \cite{Hi1, DSL, L} is
\[
\mrank(T) \coloneqq \min \bigl\{(r_1,\dots, r_d) \in \mathbb{N}_0^d : T \in 
\Sub_{r_1,\dots,r_d}(\mathbb{V}_1,\dots, \mathbb{V}_d)  \bigr\}.
\]
Note that $\rank(T) = 0$ iff  $\mrank(T) = (0,\dots,0)$ iff $T = 0$ and that $\rank(T) = 1$ iff  $\mrank(T) = (1,\dots,1)$. Thus
\begin{align*}
\Seg(\mathbb{V}_1, \dots,\mathbb{V}_d) &= \{ T \in \mathbb{V}_1 \otimes \dots \otimes \mathbb{V}_d : \rank(T) \le 1 \} = \Sub_{1,\dots,1}(\mathbb{V}_1,\dots, \mathbb{V}_d),\\
s_r\bigl(\Seg(\mathbb{V}_1, \dots,\mathbb{V}_d)\bigr) &= \{T \in \mathbb{V}_1 \otimes \dots \otimes \mathbb{V}_d : \rank(T) \le r \},\\
\Sub_{r_1,\dots,r_d}(\mathbb{V}_1,\dots, \mathbb{V}_d) &= \{T \in  \mathbb{V}_1 \otimes \dots \otimes \mathbb{V}_d : \mrank(T) \le (r_1,\dots,r_d) \}.
\end{align*}

When the vector spaces are unimportant or when we choose coordinates and represent tensors as hypermatrices, we write
\begin{align*}
\Seg(n_1, \dots, n_d) &= \{ T \in \mathbb{C}^{n_1 \times \dots \times n_d} : \rank(T) \le 1 \},\\
s_r(n_1, \dots,n_d)\bigr) &= \{T \in \mathbb{C}^{n_1 \times \dots \times n_d} : \rank(T) \le r \},\\
\Sub_{r_1,\dots,r_d}(n_1,\dots, n_d) &= \{T \in  \mathbb{C}^{n_1 \times \dots \times n_d} : \mrank(T) \le (r_1,\dots,r_d) \}.
\end{align*}
The dimension of a subspace variety is given by
\begin{equation}\label{eq:dimSub}
\dim \Sub_{r_1,\dots,r_d}(n_1,\dots, n_d) = \sum_{i=1}^d r_i(n_i-r_i) + \prod_{j=1}^d r_j.
\end{equation}

Unlike tensor rank and multilinear rank, the existence of a tensor network rank is not obvious and will be established in the following. A tensor network $\tns(G;r_1,\dots, r_c;\mathbb{V}_1,\dots,\mathbb{V}_d)$ is defined for any graph $G$ although it is trivial  when $G$ contains an isolated vertex (see \eqref{eq:isolate}). However,  tensor network ranks or $G$-ranks will require the stronger condition that $G$ be \emph{connected}.
\begin{theorem}[Every tensor is a tensor network state]\label{thm:every tensor is a TNS}
Let $T\in \mathbb{V}_1\otimes \dots \otimes \mathbb{V}_d$ and let $G$ be a connected graph with $d$ vertices and $c$ edges. Then there exists $(r_1,\dots, r_c) \in \mathbb{N}^c$ such that 
\[
T \in \tns(G; r_1,\dots,r_c ; \mathbb{V}_1,\dots, \mathbb{V}_d).
\]
In fact, we may choose $r_1 = \dots = r_c = \rank(T)$, the tensor rank of $T$.
\end{theorem}
\begin{proof}
Let $r = \rank(T)$. Then there exist $v_{1}^{(i)},\dots,v_{r}^{(i)}\in \mathbb{V}_i$, $i=1,\dots, d$, such that
\[
T = \sum_{p=1}^r v_p^{(1)}\otimes \dots \otimes v_p^{(d)}.
\]
Let us take $r_1 = \dots = r_c = r$ and for each $i=1,\dots, d$, let
\[
T_i = \sum_{p=1}^r   \Bigl(\bigotimes_{j\in \In(i)}e_{p}^{(j)} \Bigr)\otimes  v^{(i)}_{p} \otimes \Bigl(\bigotimes_{j \in \Out(i)}e_{p}^{(j)\ast}\Bigr),
\]
where $e_{1}^{(j)},\dots, e_{r}^{(j)}\in \mathbb{E}_{j}$ are a basis with dual basis $e_{1}^{(j)\ast},\dots, e_{r}^{(j)\ast} \in \mathbb{E}^*_{j}$, i.e., $e_{p}^{(j)\ast}(e_{q}^{(j)}) = \delta_{pq}$ for $p,q = 1,\dots,r$ and $j=1,\dots,d$. In addition, we set $e_{p}^{(0)} = e_{p}^{(d+1)} = 1\in \mathbb{C}$ to be one-dimensional vectors (i.e., scalars), $p=1,\dots, r$. We claim that upon contraction,
\[
\kappa_G (T_1\otimes \cdots \otimes T_d) = T.
\]
To see this, observe that for each $i =1, \dots, d$, there exists a unique $h$ such that whenever $j\in \In(i) \cap \Out(h)$, $e_{p}^{(j)}$ and $e_{q}^{(j)\ast}$ contract to give $\delta_{pq}$; so the summand vanishes except when $p = q$.
This together with the assumption that $G$ is connected  implies that $\kappa_G (T_1\otimes \cdots \otimes T_d)$ reduces to a sum of terms of the form $v^{(1)}_p \otimes \cdots \otimes v^{(d)}_p$ for $p = 1, \dots, r$,
which is of course is just $T$.
\end{proof}
As an example to illustrate the above proof, let $d = 3$ and $G = P_3$, the path graph with three vertices. Let $e_1,\dots,e_r$ be a basis of $\mathbb{E}_1$ and let $e_1^*,\dots, e_r^*$ be the dual basis. Let $f_1,\dots, f_r$ be a basis of $\mathbb{E}_2$ and let $f_1^*,\dots, f_r^*$ be the dual basis. Given a tensor
\[
T = \sum_{p=1}^r u_p\otimes v_p \otimes w_p \in \mathbb{V}_1 \otimes \mathbb{V}_2 \otimes \mathbb{V}_3,
\]
consider 
\[
T_1  = \sum_{p=1}^r u_p \otimes e_p^* \in \mathbb{V}_1\otimes \mathbb{E}^*_1,\;\;
T_2 = \sum_{p=1}^r e_p \otimes v_p \otimes f_p^* \in \mathbb{E}_1\otimes \mathbb{V}_2 \otimes \mathbb{E}_2^*,\;\;
T_3 = \sum_{p=1}^r f_p \otimes w_p \in \mathbb{E}_2\otimes \mathbb{V}_3.
\]
Now observe that a nonzero term in $\kappa_G (T_1\otimes T_2 \otimes T_3)$ must come from contracting $e_p^*$ with $e_p$ and $f^*_{p}$ with $f_p$, showing that 
$T = \kappa_G ( T_1\otimes T_2\otimes T_2) \in \tns(G; r,r; \mathbb{V}_1,\mathbb{V}_2, \mathbb{V}_3)$.

By Theorem~\ref{thm:every tensor is a TNS} and Corollary~\ref{cor:dimension supercritical}, we obtain the following general inclusion relations (independent of $G$) between a tensor network and the sets of rank-$r$ tensors and multilinear rank-$(r_1,\dots,r_d)$ tensors.
\begin{corollary}\label{cor:intermediate}
Let $G$ be a connected graph with $d$ vertices and $c$ edges. Let $\mathbb{V}_1,\dots, \mathbb{V}_d$ be vector spaces of dimensions $n_1,\dots,n_d$. Then 
\[
s_r \coloneqq \{T \in \mathbb{V}_1\otimes \dots \otimes \mathbb{V}_d : \rank(T) \le r\}  \subseteq \tns(G; \underbrace{r,\dots,r}_{c}; \mathbb{V}_1,\dots, \mathbb{V}_d). 
\]
Let $(r_1,\dots,r_c)\in \mathbb{N}^c$ and let $(p_1,\dots,p_d)\in \mathbb{N}^d$ be given by
\[
p_i: =\min \Bigl\{ \prod_{j\in \In(i)\cup \Out(i)} r_j, n_i \Bigr\}, \qquad i =1,\dots, d.
\]
Then 
\[
\tns(G; r_1,\dots,r_c; \mathbb{V}_1,\dots, \mathbb{V}_d) \subseteq \Sub_{p_1,\dots,p_d} (\mathbb{V}_1,\dots,\mathbb{V}_d)
\]
In particular, if we let $(p_1,\dots,p_d)\in \mathbb{N}^d$ where $p_i \coloneqq\min\{ r^{b_i}, n_i \}$ and $b_i \coloneqq \# \In(i)\cup \Out(i)$, then
\[
s_r \subseteq \tns(G; r,\dots,r; \mathbb{V}_1,\dots, \mathbb{V}_d) \subseteq \Sub_{p_1,\dots,p_d}(\mathbb{V}_1, \dots, \mathbb{V}_d).
\]
\end{corollary}

Since $\mathbb{N}^c$ is a partially ordered set, in fact, a \emph{lattice} \cite{Lattice}, with respect to the usual partial order
\[
(r_1,\dots,r_c) \le (s_1,\dots, s_c) \qquad \text{iff} \qquad r_1 \le s_1, \dots, r_c \le s_c.
\]
For a non-empty subset $S\subset L$, a partially ordered set, we denote the set of \emph{minimal elements} of $S$ by $\min(S)$. For example, if $S = \{(1,2), (2,1), (2,2) \} \subset \mathbb{N}^2$, then $\min(S) = \{(1,2), (2,1)\}$. By Theorem~\ref{thm:every tensor is a TNS}, for any graph $G$, any vector spaces $\mathbb{V}_1,\dots,\mathbb{V}_d$, and any tensor $T \in \mathbb{V}_1\otimes \dots \otimes \mathbb{V}_d$,
\[
 \{ (r_1,\dots,r_c ) \in \mathbb{N}^c : T \in \tns(G; r_1,\dots,r_c; \mathbb{V}_1,\dots,\mathbb{V}_d) \} \ne \varnothing.
\]
Hence we may define tensor network rank with respect to a given graph $G$, called $G$-rank for short, as the set-valued function
\begin{align*}
\rank_G: \mathbb{V}_1\otimes \dots \otimes \mathbb{V}_d &\to 2^{\mathbb{N}^c}, \\
T &\mapsto \min \{ (r_1,\dots,r_c ) \in \mathbb{N}^c : T \in \tns(G; r_1,\dots,r_c; \mathbb{V}_1,\dots,\mathbb{V}_d) \},
\end{align*}
where $2^{\mathbb{N}^c}$ is the power set of all subsets of $\mathbb{N}^c$. Note that by Theorem~\ref{thm:every tensor is a TNS}, $\rank_G(T)$ will always be a finite subset of $\mathbb{N}^c$.

Nevertheless, following convention, we prefer to have $\rank_G(T)$  be an element as opposed to a subset of $\mathbb{N}^c$. So we will define $G$-rank to be any minimal element as opposed to the set of all minimal elements.
\begin{definition}[Tensor network rank and maximal rank]\label{def:tnr}
Let $G$ be a graph with $d$ vertices and $c$ edges. We say that $(r_1,\dots, r_c) \in \mathbb{N}^c$ is a $G$-\emph{rank} of $T\in \mathbb{V}_1\otimes \dots \otimes \mathbb{V}_d$, denoted by
\[
\rank_G (T) = (r_1,\dots, r_c),
\]
if $(r_1,\dots, r_c)$ is minimal such that $T\in \tns(G; r_1,\dots,r_c; \mathbb{V}_1,\dots, \mathbb{V}_d)$, i.e., 
\[
T\in \tns(G; r'_1,\dots,r'_c; \mathbb{V}_1,\dots, \mathbb{V}_d)~\text{and}~r_1 \ge r'_1,\dots, r_c \ge r'_c \quad \Longrightarrow \quad r'_1 = r_1,\dots, r'_c = r_c.
\]
A \emph{$G$-rank decomposition} of $T$ is its expression as element of $ \tns(G; r_1,\dots,r_c; \mathbb{V}_1,\dots, \mathbb{V}_d)$ where  $\rank_G (T) = (r_1,\dots, r_c)$.
We say that $(m_1,\dots, m_c)\in \mathbb{N}^c$ is a \emph{maximal $G$-rank} of $\mathbb{V}_1\otimes \dots \otimes \mathbb{V}_d$ if $(m_1,\dots, m_c)$ is minimal such that every $T \in \mathbb{V}_1\otimes \dots \otimes \mathbb{V}_d$ has rank not more than $(m_1,\dots, m_c)$.
\end{definition}

Definition~\ref{def:tnr} says nothing about the uniqueness of a minimal $(r_1,\dots, r_c) \in \mathbb{N}^c$. We will see later that for an acyclic graph $G$, the minimal $(r_1,\dots, r_c)$ is unique. We will see an extensive list of examples in Section~\ref{sec:eg} where we compute, for various $G$'s, the $G$-ranks of a number of special tensors: decomposable tensors, monomials (viewed as a symmetric tensor and thus a tensor), the noncommutative determinant and permanent, the \textsc{w} and \textsc{ghs} states, and the structure tensor of matrix-matrix product.

It follows from Definition~\ref{def:tnr} that
\begin{equation}\label{eq:set}
\tns(G; r_1,\dots,r_c; \mathbb{V}_1,\dots, \mathbb{V}_d) = \{ T \in \mathbb{V}_1\otimes \dots \otimes \mathbb{V}_d : \rank_G(T) \le (r_1,\dots,r_c) \}.
\end{equation}
By Corollary~\ref{cor:intermediate}, $G$-rank may be viewed as an `interpolant' between tensor rank and multilinear rank; although in Section~\ref{sec:compare}, we will see that they are strictly distinct notions --- tensor and multilinear ranks are not special cases of $G$-ranks for specific choices of $G$.
Since the set in \eqref{eq:set} is in general not closed \cite{LQY}, we  let
\[
\overline{\tns}(G; r; \mathbb{V}_1,\dots, \mathbb{V}_d) \coloneqq \overline{\tns(G; r; \mathbb{V}_1,\dots, \mathbb{V}_d)}
\]
denote its Zariski closure. With this, we obtain $G$-rank analogues of border rank and generic rank.
\begin{definition}[Tensor network border rank and generic rank]\label{def:tnbgr}
Let $G$ be a graph with $d$ vertices and $c$ edges. We say that $(r_1,\dots, r_c) \in \mathbb{N}^c$ is a \emph{border $G$-rank} of $T\in \mathbb{V}_1\otimes \dots \otimes \mathbb{V}_d$, denoted by
\[
\overline{\rank}_G(T) = (r_1,\dots,r_c),
\]
if $(r_1,\dots, r_c)$ is minimal such that $T\in \overline{\tns}(G; r_1,\dots,r_c; \mathbb{V}_1,\dots, \mathbb{V}_d)$. We say that $(r_1,\dots, r_c)\in \mathbb{N}^c$ is a \emph{generic $G$-rank} of $\mathbb{V}_1\otimes \dots \otimes \mathbb{V}_d$ if $(g_1,\dots, g_c)$ is minimal such that every $T \in \mathbb{V}_1\otimes \dots \otimes \mathbb{V}_d$ has border rank not more than $(g_1,\dots, g_c)$.
\end{definition}

Observe that by Definitions~\ref{def:tnr} and \ref{def:tnbgr}, we have
\begin{align*}
\tns(G; r_1,\dots,r_c; \mathbb{V}_1,\dots, \mathbb{V}_d) &=\{ T \in \mathbb{V}_1\otimes \dots \otimes \mathbb{V}_d : \rank_G(T) \le (r_1,\dots,r_c) \},\\
\overline{\tns}(G; r_1,\dots,r_c; \mathbb{V}_1,\dots, \mathbb{V}_d) &=\{ T \in \mathbb{V}_1\otimes \dots \otimes \mathbb{V}_d : \overline{\rank}_G(T) \le (r_1,\dots,r_c) \},\\
\overline{\tns}(G; g_1,\dots,g_c; \mathbb{V}_1,\dots, \mathbb{V}_d) &= \mathbb{V}_1\otimes \dots \otimes \mathbb{V}_d = 
\tns(G; m_1,\dots,m_c; \mathbb{V}_1,\dots, \mathbb{V}_d),
\end{align*}
for any $(r_1,\dots,r_c) \in \mathbb{N}^c$ and  where $( m_1,\dots,m_c)$ and $(g_1,\dots,g_c)$ are respectively a maximal and a generic $G$-rank of  $\mathbb{V}_1\otimes \dots \otimes \mathbb{V}_d$. Note also our use of the indefinite article --- \emph{a} border/generic/maximal $G$-rank --- since these are not in general unique if $G$ is not acyclic. A more pedantic definition would be in terms of set-valued functions as in the discussion before Definition~\ref{def:tnr}.

Following Definition~\ref{def:specialtns}, when $G$ is a path graph, tree, cycle graph, product of path graphs, or a graph obtained from gluing trees and cycle graphs along edges, then we may also use the terms \textsc{tt}-rank, \textsc{ttns}-rank, \textsc{mps}-rank, \textsc{peps}-rank, or \textsc{ctns}-rank to describe the respective $G$-ranks, and likewise for their respective generic $G$-rank and border $G$-rank. The terms \emph{hierarchical rank} \cite[Chapter $11$]{Hackbusch} and \emph{tree rank} \cite{BSU} have also been used for \textsc{ttns}-rank.  Discussions of \textsc{tt}-rank, \textsc{ttns}-rank, \textsc{mps}-rank will be deferred to Sections~\ref{sec:tt}, \ref{sec:ttns}, and \ref{sec:mps} respectively. We will also compute many examples of $G$-ranks and border $G$-ranks for important tensors arising from algebraic computational complexity and quatum mechanics in Section~\ref{sec:eg}.

\section{Tensor network ranks can be much smaller than matrix, tensor, and multilinear ranks}\label{sec:low}

Our first result regarding tensor network ranks may be viewed as the main impetus for tensor networks --- we show that a tensor may have arbitrarily high tensor rank or multilinear rank and yet arbitrarily low $G$-rank for some graph $G$, in the sense that there is an arbitrarily large gap between the two ranks. The same applies to tensor network ranks corresponding to two different graphs. For all  our comparisons in this section, a single example --- the structure tensor for matrix-matrix product --- suffices to demonstrate the gaps in various ranks. We will see more examples in Section~\ref{sec:eg}. In Theorem~\ref{thm:expsmall}, we will exhibit a graph $G$ such that \emph{almost every} tensor has exponentially small $G$-rank compared to its tensor rank or the dimension of its ambient space.

\subsection{Comparison with tensor and multilinear ranks}

In this and the next subsection, we will compare different ranks $(r_1,\dots,r_c) \in \mathbb{N}^c$ and $(s_1,\dots,s_{c'}) \in \mathbb{N}^{c'}$  by a simple comparison of their $1$-norms $r_1 + \dots + r_c$ and $s_1 + \dots + s_{c'}$. In Section~\ref{sec:comparedim}, we will compare the actual dimensions of the sets of tensors with these ranks.
\begin{theorem}\label{thm:compare}
For any $d \ge 3$, there exists a tensor $T \in \mathbb{V}_1\otimes \dots \otimes \mathbb{V}_d$ such that  for some connected graph $G$ with $d$ vertices and $c$ edges,
\begin{enumerate}[\upshape (i)]
\item the tensor rank $\rank(T) = r$ is much larger than the $G$-rank $\rank_G(T) = (r_1,\dots, r_c)$ in the sense that
\[
r \gg r_1 + \dots + r_c;
\]

\item the multilinear rank $\mrank(T) = (s_1,\dots,s_d)$ is much larger than the $G$-rank $\rank_G(T) = (r_1,\dots, r_c)$ in the sense that
\[
s_1 + \dots + s_d \gg r_1 + \dots + r_c;
\]

\item for some graph $H$ with $d$ vertices and $c'$ edges, the $H$-rank $\rank_H(T) = (s_1,\dots,s_{c'})$ is much larger than the $G$-rank $\rank_G(T) = (r_1,\dots, r_c)$ in the sense that
\[
s_1 + \dots + s_{c'} \gg r_1 + \dots + r_c.
\]
\end{enumerate}
Here ``$\gg$'' indicates a difference in the order of magnitude. In particular, the gap between the ranks can be arbitrarily large.
\end{theorem}
\begin{proof}
We first let $d=3$ and later extend our construction to arbitrary $d > 3$. Set $\mathbb{V}_3 = \mathbb{C}^{n \times n}$, the $n^2$-dimensional vector space of complex $n \times n$ matrices and $\mathbb{V}_1 = \mathbb{V}_2 = \mathbb{V}_3^*$, its dual space. Let $T = \mu_n \in (\mathbb{C}^{n \times n})^* \otimes (\mathbb{C}^{n \times n})^* \otimes \mathbb{C}^{n \times n} \cong \mathbb{C}^{n^2 \times n^2 \times n^2} $ be the structure tensor of matrix-matrix product \cite{YLstruct}, i.e.,
\begin{equation}\label{eq:strassen}
\mu_n =\sum_{i,j,k=1}^{n} E^*_{ik}\otimes E^*_{kj}\otimes E_{ij}
\end{equation}
where $E_{ij}=e_{i}e_{j}^\tp \in \mathbb{C}^{n \times n}$, $i,j=1,\dots,n$, is the standard basis with dual basis $E^*_{ij}: \mathbb{C}^{n \times n} \to \mathbb{C}$, $A \mapsto a_{ij}$, $i,j=1,\dots,n$.
It is well-known that $\rank(\mu_n ) \ge n^2$ as it is not possible to multiply two $n \times n$ matrices with fewer than $n^2$ multiplications. It is trivial to see that
\begin{equation}\label{eq:mrankT}
\mrank(\mu_n) = (n^2, n^2, n^2).
\end{equation}
Let $P_3$ and $C_3$ be the path graph and cycle graph on three vertices in Figures~\ref{fig:P} and \ref{fig:C} respectively. Attach vector spaces $\mathbb{V}_1,\mathbb{V}_2,\mathbb{V}_3$ to the vertices of both graphs. Then $\mu_n\in \tns (C_3; n,n,n ;\mathbb{V}_1,\mathbb{V}_2,\mathbb{V}_3)$
and it is clear that $\mu_n \notin \tns (C_3; r_1,r_2,r_3;\mathbb{V}_1,\mathbb{V}_2,\mathbb{V}_3)$ if at least one of $r_1 \le n$, $r_2 \le n$, $r_3 \le n$ holds strictly. Hence
\[
\rank_{C_3} (\mu_n) = (n,n,n).
\]
On the other hand, we also have
\[
\rank_{P_3}(\mu_n) = (n^2,n^2).
\]
See Theorems~\ref{thm:C_3-rank of matrix multiplication} and \ref{thm:P_3-rank of matrix multiplication} for more details on computing $\rank_{C_3} (\mu_n) $ and $\rank_{P_3}(\mu_n) $.
The required conclusions follow from
\begin{align*}
\lVert \rank_{C_3} (\mu_n) \rVert_1 &= 3n  \ll  n^2  \le  \rank (\mu_n), \\
\lVert \rank_{C_3} (\mu_n) \rVert_1 &= 3n  \ll  3n^2  =   \lVert \mrank (\mu_n) \rVert_1, \\
\lVert \rank_{C_3} (\mu_n) \rVert_1 &= 3n \ll  2n^2  = \lVert \rank_{P_3} (\mu_n) \rVert_1.
\end{align*}
To extend the above to $d > 3$, let $0 \ne v_i \in \mathbb{V}_i$, $i=4,\dots, d$, and set 
\[
T_d = \mu_n \otimes v_4 \otimes \cdots \otimes v_d \in \mathbb{V}_1\otimes \cdots \otimes \mathbb{V}_d.
\]
Clearly, its tensor rank and multilinear rank are
\[
\rank(T_d) = \rank(\mu_n) \ge n^2, \qquad
\smash{\mrank(T_d)  = (n^2,n^2,n^2,\overbrace{1,\dots, 1}^{d-3})}.
\]
Next we compute $\rank_{P_d} (T_d)$ and $\rank_{C_d} (T_d)$.  Relabeling  $v_{ij1} = E_{ij}$ and $v^{(k)}_{11} =v_{k}$, we get
\[
T_d = \Bigl(\sum_{i,j,k=1}^n E^\ast_{ik} \otimes E^\ast_{kj} \otimes E_{ij}\Bigr) \otimes v_4 \otimes \dots \otimes v_d 
= \Bigl(\sum_{i,j,k=1}^n E^\ast_{ik} \otimes E^\ast_{kj} \otimes v_{ij1} \Bigr) \otimes v^{(4)}_{11} \otimes \dots \otimes v^{(d)}_{11},
\]
where it follows immediately that
\[
\smash{T_d \in \tns(P_d; n^2,n^2, \overbrace{1,\dots, 1}^{d-2}; \mathbb{V}_1,\dots, \mathbb{V}_d)}.
\]
Since $\rank_{P_3}(\mu_n) = (n^2,n^2)$, by Theorem~\ref{thm:P_3-rank of matrix multiplication},
\[
\smash{\rank_{P_d}(T_d) = (n^2,n^2,\underbrace{1,\dots, 1}_{d-2})}.
\]
Now rewrite $T_d$ as 
\begin{align*}
T_d &= \Bigl(\sum_{i,j,k=1}^n E^\ast_{ik} \otimes E^\ast_{kj} \otimes E_{ij} \Bigr) \otimes v_4 \otimes \dots \otimes v_d \\
& = \Bigl(\sum_{i,j,k,l=1}^n E^\ast_{ik} \otimes E^\ast_{kj} \otimes v_{lj} \Bigr) \otimes v^{(4)}_{l1} \otimes v^{(5)}_{11}\otimes \dots \otimes v^{(d-1)}_{11} \otimes v^{(d)}_{1i},
\end{align*}
where
\[
v_{lj} = \begin{cases} E_{ij} & l = i,\\
 0 & l\ne i,\end{cases} \quad v_{l1}^{(4)} =  \begin{cases} v_4 & l = i,\\
 0 & l\ne i,\end{cases} \quad v_{11}^{(5)} = v_5, \dots, v_{11}^{(d-1)} = v_{d-1}; \quad v_{1i}^{(d)} = v_d,\; i=1,\dots, n.
\]
It follows that $\smash{T_d \in \tns(C_d; n,n,n,n,\underbrace{1,\dots, 1}_{d-4};\mathbb{V}_1,\dots, \mathbb{V}_d)}$
and so $\smash{\rank_{C_d}(T_d) \le (n,n,n,n,\underbrace{1,\dots, 1}_{d-4})}$.
Hence we obtain 
\begin{align*}
\lVert \rank_{C_d}(T_d) \rVert_1 &\le 4n + d-4  \ll n^2 \le \rank(T_d), \\
\lVert \rank_{C_d}(T_d) \rVert_1 &\le 4n + d-4 \ll 3n^2 + d-3 =  \lVert \mrank(T_d)\rVert_1,\\
\lVert \rank_{C_d}(T_d) \rVert_1 &\le 4n + d-4  \ll  2n^2 + d - 2 = \lVert \rank_{P_d}(T_d) \rVert_1 .\qedhere
\end{align*}
\end{proof}

\subsection{Comparison with matrix rank}

The matrix rank of a matrix can also be arbitrarily higher than its $G$-rank when regarded as a $3$-tensor. We will make this precise below.

Every $d$-tensor may be regarded as a $d'$-tensor for any $d' \le d$ via \emph{flattening} \cite{L, HLA}. The most common case is when $d' = 2$ and in which case the flattening map
\[
\flat_k : \mathbb{V}_1 \otimes \dots   \otimes  \mathbb{V}_d \to ( \mathbb{V}_1 \otimes \dots \otimes \mathbb{V}_k )\otimes ( \mathbb{V}_{k+1} \otimes \dots \otimes \mathbb{V}_d), \qquad k=2,\dots,d-1,
\]
takes a $d$-tensor and sends it to a $2$-tensor by `forgetting' the tensor product structures in $\mathbb{V}_1 \otimes \dots \otimes \mathbb{V}_k $ and $\mathbb{V}_{k+1} \otimes \dots \otimes \mathbb{V}_d$. The converse of this operation also holds in the following sense. Suppose the dimensions of the vector spaces $\mathbb{V}$ and $\mathbb{W}$ factor as
\[
\dim (\mathbb{V}) = n_1 \cdots n_k, \qquad \dim (\mathbb{W}) = n_{k+1} \cdots n_d
\]
for integers $n_1,\dots,n_d \in \mathbb{N}$. Then we may impose tensor product structures on $\mathbb{V}$ and $\mathbb{W}$ so that
\begin{equation}\label{eq:sharp1}
\mathbb{V} \cong  \mathbb{V}_1 \otimes \dots \otimes \mathbb{V}_k, \qquad 
\mathbb{W} \cong  \mathbb{V}_{k+1} \otimes \dots \otimes \mathbb{V}_d,
\end{equation}
where $\dim \mathbb{V}_i = n_i$, $i= 1,\dots,d$, and where $\cong$ denotes vector space isomorphism. In which case the \emph{sharpening} map
\[
\sharp_k : \mathbb{V}\otimes  \mathbb{W} \to  \mathbb{V}_1 \otimes \dots \otimes \mathbb{V}_k \otimes  \mathbb{V}_{k+1} \otimes \dots \otimes \mathbb{V}_d, \qquad k=2,\dots,d-1,
\]
takes a $2$-tensor and sends it to a $d$-tensor by imposing the tensor product structures chosen in \eqref{eq:sharp1}. Note that both $\flat_k$ and $\sharp_k$ are vector space isomorphisms. Applying this to matrices,
\[
\mathbb{C}^{n_1 \cdots n_{k-1} \times n_k \cdots n_d} \cong
\mathbb{C}^{n_1 \cdots n_{k-1}} \otimes \mathbb{C}^{n_k \cdots n_d} \xrightarrow{\sharp_k}
\mathbb{C}^{n_1} \otimes \dots \otimes \mathbb{C}^{n_d} \cong
\mathbb{C}^{n_1 \times \dots \times n_d},
\]
and we see that any $n_1 \cdots n_{k-1} \times n_k \cdots n_d$ matrix may be regarded as an $n_1 \times \dots \times n_d$ hypermatrix.
Theorem~\ref{thm:compare} applies to matrices (i.e., $d=2$) in the sense of the following corollary.
\begin{corollary}\label{cor:compare}
There exists a matrix in $ \mathbb{C}^{mn \times p}$ whose matrix rank is arbitrarily larger than its $C_3$-rank when regarded as a hypermatrix in $\mathbb{C}^{m \times n \times p}$.
\end{corollary}
\begin{proof}
Let $\mu_n \in \mathbb{C}^{n^2 \times n^2 \times n^2} $ be a hypermatrix representing the structure tensor in Theorem~\ref{thm:compare}. Consider any flattening \cite{HLA} of $\mu_n$, say, $\beta_1(\mu_n) \in \mathbb{C}^{n^4 \times n^2}$. Then by \eqref{eq:mrankT}, its matrix rank is
\[
\rank\bigl(\beta_1(\mu_n)\bigr) = n^2 \gg 3n =  \lVert \rank_{C_3} (\mu_n)  \rVert_1. \qedhere
\]
\end{proof}

\subsection{Comparing number of parameters}\label{sec:comparedim}

One might argue that the comparisons in Theorem~\ref{thm:compare} and Corollary~\ref{cor:compare} are not completely fair as, for instance, a rank-$r$ decomposition of $T$ may still require as many parameters as  a $G$-rank-$(r_1,\dots,r_c)$ decomposition of $T$, even if $r \gg r_1 + \dots + r_c$. We will show  that this is not the case: if we measure the complexities of these decompositions by a strict count of parameters, the conclusion that $G$-rank can be much smaller than matrix, tensor, or multilinear ranks remain unchanged.

Let  $\mu_n$  be the structure tensor for matrix-matrix product as in the proof of Theorem~\ref{thm:compare}, which also shows that
\[
\rank_{C_3}(\mu_n)= (n,n,n), \qquad \rank_{P_3}(\mu_n) = (n^2,n^2), \qquad
\mrank(\mu_n) = (n^2,n^2,n^2).
\]
Let $r \coloneqq  \rank (\mu_n) $, the exact value of which is open but its current best known lower bound \cite{ME2013} is
\begin{equation}\label{eq:lbd}
r  \ge  3n^2 - 2\sqrt{2} n^{3/2} - 3n,
\end{equation}
which will suffice for our purpose.

Geometrically, the number of parameters is the dimension. So the number of parameters required to decompose $\mu_n$ as a point in $\tns(C_3;n,n,n;\mathbb{V}_1,\mathbb{V}_2,\mathbb{V}_3)$, $\tns(P_3;n^2,n^2;\mathbb{V}_1,\mathbb{V}_2,\mathbb{V}_3)$, $\Sub_{n^2,n^2, n^2} (\mathbb{V}_1,\mathbb{V}_2,\mathbb{V}_3)$, and $\sigma_{r} (\Seg (\mathbb{V}_1,\mathbb{V}_2,\mathbb{V}_3))$ are given by their respective dimensions:
\begin{align}
\dim \tns(C_3;n,n,n;\mathbb{V}_1,\mathbb{V}_2,\mathbb{V}_3) &= 3n^4 - 3n^2, \label{eq:C3} \\
\dim \tns(P_3;n^2,n^2;\mathbb{V}_1,\mathbb{V}_2,\mathbb{V}_3) & = n^6, \label{eq:P3} \\
\dim \Sub_{n^2,n^2, n^2}(\mathbb{V}_1,\mathbb{V}_2,\mathbb{V}_3) &= n^6, \label{eq:Sub} \\
\dim \sigma_r (\Seg (\mathbb{V}_1,\mathbb{V}_2,\mathbb{V}_3)) &\ge 9n^4 - 6\sqrt{2} n^{7/2} - 9n^3 - 6n^2 +4\sqrt{2} n^{3/2} + 6n -1. \label{eq:sigma}
\end{align}
The dimensions in \eqref{eq:C3} and \eqref{eq:P3} follow from \cite[Theorem 5.3]{dtn} and that  in \eqref{eq:Sub} follows from \eqref{eq:dimSub}. The lower bound on the tensor rank in \eqref{eq:lbd} gives us the lower bound on the dimension in \eqref{eq:sigma} by \cite[Theorem~5.2]{Abo2009}.

In conclusion, a $C_3$-rank decomposition of  $\mu_n$ requires fewer parameters than its $P_3$-rank decomposition, its multilinear rank decomposition, and its tensor rank decomposition.

\section{Properties of tensor network rank}\label{sec:prop}

We will establish some fundamental properties of $G$-rank in this section.  We begin by showing that like tensor rank and multilinear rank, $G$-ranks are independent of the choice of the ambient space, i.e., for a fixed $G$ and any  vector spaces $\mathbb{W}_i \subseteq \mathbb{V}_i$, $i = 1,\dots,d$, a tensor in $\mathbb{W}_1 \otimes \dots \otimes \mathbb{W}_d$ has the same $G$-rank whether it is regarded as an element of $\mathbb{W}_1 \otimes \dots \otimes \mathbb{W}_d$  or of $\mathbb{V}_1 \otimes \dots \otimes \mathbb{V}_d$. The proof is less obvious and more involved than for tensor rank or multilinear rank, a consequence of Lemma~\ref{lemma: subspace TNS}.
\begin{theorem}[Inheritance property]\label{thm:$G$-rank subspace}
Let $G$ be a connected graph with $d$ vertices and $c$ edges.  Let $\mathbb{W}_i\subseteq \mathbb{V}_i$ be a linear subspace,  $i=1,\dots, d$, such that $T\in \mathbb{W}_1\otimes \dots \otimes \mathbb{W}_d$. Then $(r_1,\dots,r_c) \in \mathbb{N}^c$ is a $G$-rank of $T$ as an element in $\mathbb{W}_1\otimes \dots \otimes \mathbb{W}_d$ if and only if  it is a $G$-rank of $T$ as an element in $\mathbb{V}_1\otimes \dots \otimes \mathbb{V}_d$.
\end{theorem}
\begin{proof}
Let $\rank_G(T) = (r_1,\dots, r_c)$ as an element in $\mathbb{W}_1\otimes \dots \otimes \mathbb{W}_d$ and $\rank_G(T) = (s_1,\dots, s_c)$ as an element in $\mathbb{V}_1\otimes \dots \otimes \mathbb{V}_d$. Then $(s_1,\dots, s_c) \le (r_1,\dots, r_c)$. Suppose they are not equal, then $s_i < r_i$ for at least one $i \in \{1,\dots,c\}$. Since $T \in  \tns(G; s_1,\dots,s_c; \mathbb{V}_1,\dots, \mathbb{V}_d)$ and $T\in \mathbb{W}_1\otimes \dots \otimes \mathbb{W}_d$, we must have $T \in  \tns(G; s_1,\dots,s_c; \mathbb{W}_1,\dots, \mathbb{W}_d)$ by Lemma~\ref{lemma: subspace TNS}, contradicting our assumption that $\rank_G(T) = (r_1,\dots, r_c)$ as an element in $\mathbb{W}_1\otimes \dots \otimes \mathbb{W}_d$.
\end{proof}
It is well-known that Theorem~\ref{thm:$G$-rank subspace} holds true for tensor rank and multilinear rank \cite[Proposition~3.1]{DSL}; so this is yet another way $G$-ranks resemble the usual notions of ranks. This inheritance property has often been exploited in the calculation of tensor rank and similarly  Theorem~\ref{thm:$G$-rank subspace} provides a useful simplification in the calculation of $G$-ranks:  Given $T\in \mathbb{V}_1\otimes \dots \otimes \mathbb{V}_d$, we may find linear subspaces $\mathbb{W}_i\subseteq \mathbb{V}_i$, $i=1,\dots, d$,  such that $T \in \mathbb{W}_1\otimes \dots \otimes \mathbb{W}_d$ and determine the $G$-rank of $T$ as a tensor in the smaller space $\mathbb{W}_1\otimes \dots \otimes \mathbb{W}_d$. With this in mind, we introduce the following terminology.
\begin{definition}\label{def:degen}
$T\in \mathbb{V}_1\otimes \dots \otimes \mathbb{V}_d$ is \emph{degenerate} if there exist subspaces $\mathbb{W}_i\subseteq \mathbb{V}_i$, $i=1,\dots,d$, with at least one strict inclusion, such that  $T\in \mathbb{W}_1\otimes \dots \otimes \mathbb{W}_d$. Otherwise $T$ is \emph{nondegenerate}.
\end{definition}

Theorem~\ref{thm:$G$-rank subspace} tells us the behavior of $G$-ranks with respect to subspaces. The next result tells us about the behavior of $G$-ranks with respect to subgraphs.
\begin{proposition}[Subgraph]\label{prop:subgraph generic rank}
Let $G$ be a connected graph with $d$ vertices and $c$ edges. Let $H$ be a connected subgraph of $G$ with $d$ vertices and $c'$ edges.
\begin{enumerate}[\upshape (i)]
\item Let $(s_1,\dots,s_{c'}) \in \mathbb{N}^{c'}$ be a generic $H$-rank of $\mathbb{V}_1\otimes \cdots \otimes \mathbb{V}_d$. Then there exists $(r_1,\dots, r_c) \in \mathbb{N}^c$ with  $r_i \le s_i$, $i=1,\dots, c'$, such that $(r_1,\dots, r_c)$ is a generic $G$-rank of $\mathbb{V}_1\otimes \cdots \otimes \mathbb{V}_d$.

\item Let $T\in  \mathbb{V}_1\otimes \cdots \otimes \mathbb{V}_d$ and $\rank_H(T) = (s_1,\dots,s_{c'})$. Then there exists $(r_1,\dots,r_c) \in \mathbb{N}^c$ with $r_i\le s_i, i =1,\dots, c'$, such that $\rank_G(T) = (r_1,\dots,r_{c})$.
\end{enumerate}
\end{proposition}
\begin{proof}
By Proposition~\ref{prop:remove an edge}, we have 
\begin{equation}\label{eq:s1}
\tns(H;s_1,\dots,s_{c'}; \mathbb{V}_1,\dots, \mathbb{V}_d) = \smash{\tns(G;s_1,\dots, s_{c'},\overbrace{1,\dots,1}^{c-c'}; \mathbb{V}_1,\dots, \mathbb{V}_d)}.
\end{equation}
Since $(s_1,\dots, s_{c'})$ is a generic $H$-rank of $\mathbb{V}_1\otimes \cdots \otimes \mathbb{V}_d$, 
\[
\overline{\tns}(G;s_1,\dots, s_{c'},1,\dots,1; \mathbb{V}_1,\dots, \mathbb{V}_d) = 
\overline{\tns}(H;s_1,\dots,s_{c'}; \mathbb{V}_1,\dots, \mathbb{V}_d) = \mathbb{V}_1\otimes \cdots \otimes \mathbb{V}_d,
\]
implying that $ \mathbb{V}_1\otimes \cdots \otimes \mathbb{V}_d$ has a generic $G$-rank with $r_i \le s_i$, $i =1,\dots, c'$.
The same argument and \eqref{eq:s1} show that  $\rank_{G}(T) = (s_1,\dots, s_{c'},1,\dots,1) \in \mathbb{N}^c$.
\end{proof}

\begin{corollary}
Let $T\in \mathbb{V}_1\otimes \cdots \otimes \mathbb{V}_d$. Then among all graphs $G$ with $d$ vertices $T$ has the smallest $G$-rank when $G = K_d$, the complete graph on $d$ vertices.
\end{corollary}

Theorem~\ref{thm:compare} tells us that some tensors have much lower $G$-ranks relative to their tensor rank, multilinear rank, or $H$-rank for some other graph $H$. We now prove a striking result that essentially says that for some $G$, almost all tensors have much  lower $G$-ranks relative to the dimension of the tensor space. In fact, the gap is exponential  in this case: For a tensor space of dimension $O(n^d)$, the $G$-rank of almost every tensor in it would only be $O\bigl(n(d-1)\bigr)$; to see the significance, note that almost all tensors in such a space would have tensor rank $O\bigl(n^d/(nd - d +1)\bigr)$.

\begin{theorem}[Almost all tensors have exponentially low $G$-rank]\label{thm:expsmall}
There exists a connected graph $G$ such that $\lVert \rank_G(T) \rVert_1 \ll  \dim (\mathbb{V}_1\otimes \cdots \otimes \mathbb{V}_d)$ for all $T$ in a Zariski dense subset of $\mathbb{V}_1\otimes \cdots \otimes \mathbb{V}_d$.
\end{theorem}
\begin{proof}
Let $G = S_d$, the star graph on $d$ vertices in Figure~\ref{fig: star graph}. Let $\dim \mathbb{V}_i = n_i$, $i=1,\dots,d$. Without loss of generality, we let the  center vertex of $S_d$ be vertex $1$ and associate $\mathbb{V}_1$ to it. Clearly, any $T \in \mathbb{V}_1\otimes \cdots \otimes \mathbb{V}_d$ has $\rank_{S_d}(T) = (r_1,\dots, r_{d-1})$ where $r_i \le n_{i+1}$, $i=1,\dots, d-1$. Moreover,
\[
\{T \in\mathbb{V}_1\otimes \cdots \otimes \mathbb{V}_d : \rank_{S_d}(T) = (n_2,\dots, n_d) \}
\]
is a Zariski open dense subset of $\mathbb{V}_1\otimes \cdots \otimes \mathbb{V}_d$. Now observe that
\[
\lVert \rank_{S_d}(T)  \rVert_1 = r_1 +\dots+r_{d-1} \le n_2 + \dots + n_d \ll n_1 \cdots n_d = \dim (\mathbb{V}_1\otimes \cdots \otimes \mathbb{V}_d).
\]
In particular, if $n_i = n$, $i=1,\dots, d$, then the exponential gap becomes evident:
\[
\lVert \rank_{S_d}(T)  \rVert_1\le n(d-1) \ll n^d = \dim (\mathbb{V}_1\otimes \cdots \otimes \mathbb{V}_d). \qedhere
\]
\end{proof}

\begin{proposition}[Bound for $G$-ranks]\label{prop:upper bound G-ranks}
Let $G$ be a  connected graph with $d$ vertices and $c$ edges. If $T\in \mathbb{V}_1\otimes \dots \otimes \mathbb{V}_d$ is nondegenerate and  $\rank_G(T) =  (r_1,\dots, r_c)\in \mathbb{N}^c$, then we must have 
\[
\prod_{j\in \In(i)\cup \Out(i)} r_j \ge \dim\mathbb{V}_i, \qquad i =1,\dots, d .
\]
\end{proposition}
\begin{proof}
Suppose there exists some $i \in \{1,\dots,d\}$ such that 
\[
\prod_{j\in \In(i)\cup \Out(i)} r_j  <  \dim \mathbb{V}_i.
\]
By Proposition~\ref{prop:reduction supercritical},  $\mathbb{V}_i$ may be replaced by a subspace of dimension $\prod_{j\in \In(i)\cup \Out(i)} r_j$, showing that $T$ is degenerate, a contradiction.
\end{proof}

While we have formulated our discussions in a coordinate-free manner, the notion of $G$-rank applies to hypermatrices by making a choice of bases so that $\mathbb{V}_i = \mathbb{C}^{n_i}$, $i=1,\dots,d$. In which case $\mathbb{C}^{n_1} \otimes \dots \otimes \mathbb{C}^{n_d} \cong \mathbb{C}^{n_1 \times \dots \times n_d}$ is the space of $n_1\times \cdots \times n_d$ hypermatrices.
\begin{corollary} Let $A \in \mathbb{C}^{n_1 \times \dots \times n_d}$ and $G$ be a $d$-vertex graph.
\begin{enumerate}[\upshape (i)]
\item\label{GLinv} Let $(M_1,\dots,M_d) \in \GL_{n_1}( \mathbb{C}) \times \dots \times \GL_{n_d}(\mathbb{C})$. Then
\begin{equation}\label{eq:mlm}
\rank_G\bigl( (M_1,\dots, M_d) \cdot A \bigr) = \rank_G(A).
\end{equation}

\item\label{subsp} Let $n_1' \ge n_1, \dots, n_d' \ge n_d$. Then $\rank_G(A)$ is the same whether we regard $A$ as an element of $ \mathbb{C}^{n_1 \times \dots \times n_d}$ or as an element of $ \mathbb{C}^{n_1' \times \dots \times n_d'}$.
\end{enumerate}
\end{corollary}
\begin{proof}
The operation $\cdot$ denotes multilinear matrix multiplication \cite{HLA}, which is exactly the change-of-basis transformation for $d$-hypermatrices.
\eqref{GLinv} follows from the fact that the definition of $G$-rank is basis-free and \eqref{subsp} follows from Theorem~\ref{thm:$G$-rank subspace}.
\end{proof}

\section{Tensor trains}\label{sec:tt}

Tensor trains and \textsc{tt}-rank (i.e., $P_d$-rank) are the simplest instances of tensor networks and tensor network ranks. They are a special case of both \textsc{ttns} in Section~\ref{sec:ttns} (since $P_d$ is a tree) and \textsc{mps} in Section~\ref{sec:mps} (see \eqref{eq:tt=mps}). However, we single them out as \textsc{tt}-rank generalizes matrix rank and may be related to multilinear rank and tensor rank in certain cases; furthermore, we may determine the dimension of the set of tensor trains and, in some cases, the generic and maximal \textsc{tt}-ranks. We begin with two examples.

\begin{example}[Matrix rank]\label{example:TT2}
Let $G = P_2$, the path graph on two vertices $1$ and $2$ (see Figure~\ref{fig:P}). This yields the simplest tensor network states:
$\tns(P_2; r; m,n) $ is simply the set rank-$r$ matrices, or more precisely,
\begin{equation}\label{eq:ttmat}
\tns(P_2; r; m,n) = \{ T \in \mathbb{C}^{m \times n} : \rank(T) \le r \}, 
\end{equation}
and so matrix rank is just $P_2$-rank. Moreover, observe that
\[
\tns(P_2; r; m,n)  \cap \tns(P_2; s; m,n) = \tns(P_2; \min\{r,s\}; m,n),
\]
a property that we will generalize  in Lemma~\ref{lemma: intersection TNS}  to arbitrary $G$-ranks for acyclic $G$'s.
\end{example}

\begin{example}[Multilinear rank]\label{example:TT3}
Let $G = P_3$ with vertices $1,2,3$, which is the next simplest case. Orient $P_3$ by $1\rightarrow 2\rightarrow 3$. Let $(r_1,r_2) \in \mathbb{N}^2$ satisfy
$r_1\le m$,  $ r_1 r_2\le n$,  $r_2\le p$. In this case
\begin{equation}\label{eq:tt3ten}
\tns(P_3; r_1, r_2; m,n,p) = \{ T \in \mathbb{C}^{m \times n \times p} : \mrank(T) \le (r_1,r_1r_2,r_2) \},
\end{equation}
and so
\[
\rank_{P_3}(T) = (r_1, r_2) \quad \text{iff} \quad \mrank(T) = (r_1, r_1r_2, r_3).
\]
The $P_3$-rank of any  $T\in \mathbb{C}^{m \times n \times p}$ is unique, a consequence of  Theorem~\ref{thm:uniqueness G-rank}. But this may be deduced directly: Suppose $T$ has two $P_3$-ranks $(r_1,r_2)$ and $(s_1,s_2)$. Then $T\in\tns(P_3; r_1, r_2; m,n, p)  \cap \tns(P_3; s_1, s_2; m,n, p) $. We claim that there exists $(t_1, t_2) \in \mathbb{N}^2$ such that 
\[
T\in\tns(P_3; t_1, t_2; m,n, p) \subseteq \tns(P_3; r_1, r_2; m,n, p)  \cap \tns(P_3; s_1, s_2; m,n, p) 
\] 
Without loss of generality, we may assume\footnote{We cannot have $(r_1, r_2) \le (s_1, s_2)$ or $(s_1,s_2) \le (r_1, r_2)$ since both are assumed to be $P_3$-ranks of $T$. So that leaves either (i) $r_1 \le s_1$, $r_2 \ge s_2$ or (ii) $r_1 \ge s_1$, $r_2 \le s_2$ --- we pick (i) if $r_1r_2 \le s_1s_2$ and (ii) if $s_1s_2 \le r_1r_2$. By symmetry the subsequent arguments are identical.} that $r_1\le s_1$, $r_2 \ge s_2$, and that $r_1 r_2\le s_1 s_2$. By \eqref{eq:tt3ten} and the observation that
\[
\Sub_{r_1,r_1 r_2,r_2} (m,n, p)\cap \Sub_{s_1,s_1 s_2,s_2} (m,n, p)= \Sub_{r_1,r_1 r_2,s_2}(m,n, p),
\]
the assumption that $r_2\ge s_2$ allows us to conclude that
\[
\Sub_{r_1,r_1 r_2,s_2} (m,n, p)= \Sub_{r_1,r_1 s_2,s_2}(m,n, p) = \tns(P_3;r_1,s_2;m,n,p)
\]
and therefore we may take $(t_1, t_2) = (r_1,s_2)$. So $(r_1,r_2)=\rank_{P_3}(T) \le (r_1,s_2)$ and we must have $r_2 = s_2$; similarly $(s_1,s_2)=\rank_{P_3}(T) \le (r_1,s_2)$ and we must have $r_1 = s_1$.
\end{example}

\begin{example}[Rank-one tensors]\label{example:TTd}
The set of decomposable tensors of order $d$, i.e.,  rank-$1$ or multilinear rank-$(1,\dots,1)$ tensors (or the zero tensor), are exactly tensor trains of $P_d$-rank $(1,\dots,1)$.
\begin{equation}\label{eq:ttrank1}
\tns(P_d;1,\dots,1;n_1,\dots,n_d) = \{T \in \mathbb{C}^{n_1 \times \dots \times n_d} : \rank(T) \le 1 \}.
\end{equation}
\end{example}

The equalities \eqref{eq:ttmat}, \eqref{eq:tt3ten}, \eqref{eq:ttrank1} are obvious from definition and may also be deduced from the respective dimensions given in \cite[Theorem~4.8]{dtn}.
\begin{theorem}[Dimension of tensor trains]\label{thm:dimension of tensor train}
Let $P_d$ be the path graph with $d\ge 2$ vertices and $d-1$ edges. Let $(r_1,\dots,r_{d-1})\in \mathbb{N}^{d-1}$ be such that $\tns(P_d; r_1,\dots,r_{d-1}; n_1,\dots, n_d)$ is supercritical or critical. Then 
\begin{multline}\label{eqn:dimension of TT}
\dim \tns(P_d; r_1,\dots,r_{d-1}; n_1,\dots, n_d)  = r_{d/2}^2 + \sum_{i=1}^{d} r_{i-1}r_i(n_i-r_{i-1}r_i) \\
+ \sum_{j=1}^{\lfloor d/2 \rfloor-1} r_{j+1}^2(r_j^2-1) + r_{d-{j-1}}^2(r_{d-j}^2-1),
\end{multline}
where $r_0 = r_d \coloneqq 1$ and 
\[
r_{d/2} \coloneqq
\begin{cases}
r_{d/2} &\text{for} \; d\; \text{even},\\
r_{(d-1)/2} r_{(d+1)/2} &\text{for} \; d \; \text{odd}.
\end{cases}
\]
\end{theorem}

If we set $k_i = m_i = r_{i-1}r_i$, $i=1,\dots, d$, in \eqref{eq:dimSub}, then
\[
\dim \Sub_{k_1,\dots,k_d}(\mathbb{V}_1,\dots,\mathbb{V}_d)= \sum_{i=1}^d r_{i -1}r_i(n_i-r_{i-1}r_i) + \prod_{j=1}^{d-1} r_j^2,
\]
and with this, we have the following corollary of \eqref{eqn:dimension of TT}.
\begin{corollary}\label{cor:dimension of TT}
Let $P_d$ be the path graph of $d\ge 2$ vertices and $d-1$ edges. Let $(r_1,\dots,r_{d-1})\in \mathbb{N}^{d-1}$ be such that $\tns(P_d; r_1,\dots,r_{d-1}; \mathbb{V}_1,\dots, \mathbb{V}_d)$ is supercritical or critical and $m_i = r_{i-1}r_i$, $i=1,\dots, d$. Then
\[
\tns(P_d; r_1,\dots,r_{d-1}; \mathbb{V}_1,\dots, \mathbb{V}_d) \subseteq \Sub_{m_1,\dots,m_d}(\mathbb{V}_1,\dots, \mathbb{V}_d)
\]
is a subvariety of codimension
\[
\prod_{j=1}^{d-1} r_j^2 - \Bigl(\sum_{j=1}^{\lfloor d/2\rfloor-1} r_{j+1}^2(r_j^2-1) + r_{d-{j-1}}^2(r_{d-j}^2-1) + m^2_{d/2} \Bigr).
\]
In particular, we have
\begin{align*}
\tns(P_2; r; \mathbb{V}_1,\mathbb{V}_2) &= \sigma_{r}\bigl(\Seg(\mathbb{V}_1,\mathbb{V}_2)\bigr),\\
\tns(P_3; r_1, r_2; \mathbb{V}_1,\mathbb{V}_2,\mathbb{V}_3) &= \Sub_{r_1,r_1 r_2,r_2}(\mathbb{V}_1,\mathbb{V}_2,\mathbb{V}_3),\\
\tns(P_d; \underbrace{1,\dots,1}_{d-1};\mathbb{V}_1,\dots, \mathbb{V}_d) &= \Seg(\mathbb{V}_1, \dots,\mathbb{V}_d),
\end{align*}
where $r,r_1,r_2 \in \mathbb{N}$. For all other $d$ and $r$, we have a strict inclusion
\[
\tns(P_d; r_1,\dots,r_{d-1}; \mathbb{V}_1,\dots, \mathbb{V}_d) \subsetneq \Sub_{m_1,\dots,m_d}(\mathbb{V}_1,\dots, \mathbb{V}_d).
\]
\end{corollary}

We now provide a few examples of generic and maximal \textsc{tt}-ranks, in which $G$ is the path graph $P_2$, $P_3$, or $P_4$ in Figure~\ref{fig:P}. Again, these represent the simplest instances of more general results for tree tensor networks in Section~\ref{sec:ttns} and are intended to be instructive.
In the following let $\mathbb{V}_i$ be a vector space of dimension $n_i$, $i=1,2,3,4$.
\begin{example}[Generic/maximal \textsc{tt}-rank of $\mathbb{C}^{n_1 \times n_2}$]
In this case maximal and generic $G$-ranks are equivalent since $G$-rank and border $G$-rank are equal for acyclic graphs (see Corollary~\ref{cor:border G-rank}).
By \eqref{eq:ttmat}, the generic $P_2$-rank of $\mathbb{V}_1\otimes \mathbb{V}_2 \cong \mathbb{C}^{n_1 \times n_2}$ is $\min\{n_1,n_2\}$, i.e., the generic matrix rank.
\end{example}
\begin{example}[Generic/maximal \textsc{tt}-rank of $\mathbb{C}^{n_1 \times n_2 \times n_3}$]\label{example:P_3 generic rank}
Assume for simplicity that $n_1 n_2 \ge n_3$, we will show that the generic $P_3$-rank of $\mathbb{V}_1\otimes \mathbb{V}_2 \otimes \mathbb{V}_3 \cong \mathbb{C}^{n_1 \times n_2 \times n_3}$ is $(n_1,n_3)$. Let $(g_1,g_2)\in \mathbb{N}^2$. By Corollary~\ref{cor:dimension supercritical}, if $\tns(P_3; g_1,g_2; \mathbb{V}_1,\mathbb{V}_2,\mathbb{V}_3)$ is supercritical at vertices $1$ and $3$, then
\[
\tns(P_3; g_1,g_2; \mathbb{V}_1,\mathbb{V}_2, \mathbb{V}_3) \subsetneq \mathbb{V}_1\otimes \mathbb{V}_2\otimes \mathbb{V}_3.
\]
So we may assume that $g_1$ and $g_2$ are large enough so that $\tns(P_3; g_1,g_2; \mathbb{V}_1,\mathbb{V}_2, \mathbb{V}_3)$ is critical or subcritical at  vertex $1$ or vertex $3$. Thus we must have  $g_1\ge n_1$ or $g_2\ge n_3$.
By Proposition~\ref{prop:reduction valence one},
\[
\tns(P_3 ; n_1,n_2;\mathbb{V}_1,\mathbb{V}_2, \mathbb{V}_3) = \mathbb{V}_1\otimes \mathbb{V}_2\otimes \mathbb{V}_3
\]
and hence the generic $P_3$-rank of $\mathbb{V}_1\otimes \mathbb{V}_2\otimes \mathbb{V}_3$ is $(n_1,n_3)$.
\end{example}
\begin{example}[Generic/maximal \textsc{tt}-rank of $\mathbb{C}^{2\times 2 \times 2 \times 2}$]\label{example:P_4 generic rank}
Let $n_1 = n_2 = n_3 = n_4 = 2$. Let $(g_1,g_2,g_3) \in \mathbb{N}^3$ be the generic $P_4$-rank of $\mathbb{V}_1 \otimes \mathbb{V}_2 \otimes \mathbb{V}_3 \otimes \mathbb{V}_4\cong \mathbb{C}^{2 \times 2 \times 2 \times 2} $. By the definition of $P_4$-rank we must have 
\[
(g_1,g_2,g_3) \le (2,4,2).
\]
Suppose that either $g_1 = 1$ or $g_3 = 1$ --- by symmetry, suppose $g_1 = 1$. In this case a $4$-tensor in
$\tns (P_4;1,g_2,g_3;\mathbb{V}_1, \mathbb{V}_2, \mathbb{V}_3, \mathbb{V}_4)$
has rank at most one when regarded as a matrix in $\mathbb{V}_1 \otimes (\mathbb{V}_2 \otimes \mathbb{V}_3 \otimes \mathbb{V}_4)$. However, a generic element in $\mathbb{V}_1 \otimes \mathbb{V}_2 \otimes \mathbb{V}_3 \otimes \mathbb{V}_4$ has rank two when regarded as a matrix in $\mathbb{V}_1 \otimes (\mathbb{V}_2 \otimes \mathbb{V}_3 \otimes \mathbb{V}_4)$, a contradiction. Thus $g_1 = g_3 = 2$. Now by Proposition~\ref{prop:reduction valence one}, if $g_2 \le 3$, then
\[
\tns(P_4;2,g_2,2; \mathbb{V}_1, \mathbb{V}_2, \mathbb{V}_3, \mathbb{V}_4) = \tns(P_2; g_2; \mathbb{V}_1\otimes \mathbb{V}_2, \mathbb{V}_3 \otimes \mathbb{V}_4) \subsetneq  \mathbb{V}_1 \otimes \mathbb{V}_2 \otimes \mathbb{V}_3 \otimes \mathbb{V}_4,
\]
since $\tns(P_2;g_2;4,4)$ is the set of all $4\times 4$ matrices of rank at most three.
Hence the generic $P_4$-rank of $ \mathbb{V}_1 \otimes \mathbb{V}_2 \otimes \mathbb{V}_3 \otimes \mathbb{V}_4 $ must be $(2,4,2)$.
\end{example}

\section{Tree tensor networks}\label{sec:ttns}

We will now discuss \textsc{ttns}-ranks, i.e., $G$-ranks where $G$ is a tree (see Figure~\ref{fig:T}). Since $G$ is assumed to be connected and every connected acyclic graph is a tree, this includes all acyclic $G$ with  tensor trains ($G= P_d$) and star tensor network states ($G = S_d$) as special cases. A particularly important result in this case is that \textsc{ttns}-rank is always \emph{unique}  and  is easily computable as matrix ranks of various flattenings of tensors.

We first establish the intersection property that we saw in Examples~\ref{example:TT2} and \ref{example:TT3} more generally.
\begin{lemma}\label{lemma: intersection TNS}
Let $G$ be a tree with $d$ vertices and $c$ edges. Let $(r_1,\dots,r_c)$ and $(s_1,\dots,s_c) \in \mathbb{N}^{c}$ be such that $\tns(G; r_1,\dots,r_c; \mathbb{V}_1,\dots, \mathbb{V}_d)$ and $\tns(G; s_1,\dots,s_c; \mathbb{V}_1,\dots, \mathbb{V}_d)$ are subcritical. Then 
\[
 \tns(G; r_1,\dots,r_c; \mathbb{V}_1,\dots, \mathbb{V}_d) \cap \tns(G; s_1,\dots,s_c; \mathbb{V}_1,\dots, \mathbb{V}_d)
= \tns(G; t_1,\dots, t_c; \mathbb{V}_1,\dots, \mathbb{V}_d),
\]
where $(t_1,\dots, t_c)\in \mathbb{N}^{c}$ is given by $t_j = \min\{r_j,s_j\}$, $j=1,\dots,c$.
\end{lemma}
\begin{proof}
Without loss of generality, let the vertex $1$ be a degree-one vertex (which must exist in a tree) and let the edge $e_1$ be adjacent to the vertex $1$. It is straightforward to see that 
\[
\tns(G; t_1,\dots, t_c; \mathbb{V}_1,\dots, \mathbb{V}_d) \subseteq 
 \tns(G; r_1,\dots,r_c; \mathbb{V}_1,\dots, \mathbb{V}_d) \cap \tns(G; s_1,\dots,s_c; \mathbb{V}_1,\dots, \mathbb{V}_d).
\]
To prove the opposite inclusion, we proceed by induction on $d$. The required inclusion holds for $d\le 3$ by our calculations in Examples~\ref{example:TT2} and \ref{example:TT3}. Assume that it holds for $d-1$. Now observe that by Proposition~\ref{prop:reduction valence one},
\[
\tns(G; r_1,r_2,\dots, r_c; \mathbb{V}_1,\dots, \mathbb{V}_d) = \tns(G'; r_2,\dots, r_c; \mathbb{V}_1\otimes \mathbb{V}_2,\mathbb{V}_3,\dots, \mathbb{V}_d), 
\]
where $G'$ is the graph obtained by removing vertex $1$ and its only edge $e_1$. Similarly,
\[
\tns(G; s_1, s_2,\dots, s_{c}; \mathbb{V}_1,\dots, \mathbb{V}_d) = \tns(G'; s_2,\dots, s_{c}; \mathbb{V}_1\otimes \mathbb{V}_2,\mathbb{V}_3,\dots, \mathbb{V}_d).
\]
Therefore,
\begin{multline*}
\tns(G; r_1, r_2,\dots,r_c; \mathbb{V}_1,\dots, \mathbb{V}_d)  \cap \tns(G; s_1, s_2, \dots,s_c; \mathbb{V}_1,\dots, \mathbb{V}_d) = \\
\tns(G'; r_2,\dots,r_c; \mathbb{V}_1\otimes \mathbb{V}_2,\dots, \mathbb{V}_d) \cap \tns(G'; s_2,\dots,s_c; \mathbb{V}_1\otimes \mathbb{V}_2,\dots, \mathbb{V}_d).
\end{multline*}
Given  $T\in\tns(G'; r_2,\dots,r_c; \mathbb{V}_1\otimes \mathbb{V}_2,\dots, \mathbb{V}_d) \cap \tns(G'; s_2,\dots,s_c; \mathbb{V}_1\otimes \mathbb{V}_2,\dots, \mathbb{V}_d)$, there must be some subspace $\mathbb{W} \subseteq \mathbb{V}_1\otimes \mathbb{V}_2$ such that $\dim \mathbb{W}\le \min\{r_2,s_2\}$ and thus
\[
T\in \tns(G'; r_2,\dots,r_c; \mathbb{W},\mathbb{V}_3, \dots, \mathbb{V}_d) \cap \tns(G'; s_2,\dots,s_c; \mathbb{W},\mathbb{V}_3, \dots, \mathbb{V}_d).
\]
By the induction hypothesis, we have
\begin{multline*}
\tns(G'; r_2,\dots,r_c; \mathbb{W},\mathbb{V}_3, \dots, \mathbb{V}_d) \cap \tns(G'; s_2,\dots,s_c; \mathbb{W},\mathbb{V}_3, \dots, \mathbb{V}_d)\\
 = \tns(G'; t_2,\dots,t_c; \mathbb{W},\mathbb{V}_3, \dots, \mathbb{V}_d).
\end{multline*}
Since both $\tns(G; r_1,r_2,\dots,r_c; \mathbb{V}_1,\dots, \mathbb{V}_d)$ and $\tns(G; s_1,s_2,\dots,s_c; \mathbb{V}_1,\dots, \mathbb{V}_d)$ are subcritical,  $\tns(G; t_1,t_2,\dots,t_c; \mathbb{V}_1,\dots, \mathbb{V}_d)$ is also subcritical. By \eqref{eq:subsp2} and Proposition~\ref{prop:reduction valence one},
\begin{align*}
T\in \tns(G'; t_2,\dots,t_c; \mathbb{W},\mathbb{V}_3, \dots, \mathbb{V}_d)&\subseteq \tns(G'; t_2,\dots,t_c; \mathbb{V}_1\otimes \mathbb{V}_2,\dots, \mathbb{V}_d) \\
&=\tns(G; t_1,t_2,\dots,t_c ; \mathbb{V}_1,\dots, \mathbb{V}_d),
\end{align*}
showing that the inclusion also holds for $d$, completing our induction proof.
\end{proof}

We are now ready to prove a more general version of Lemma~\ref{lemma: intersection TNS}, removing the subcriticality requirement. Note  that Lemma~\ref{lemma: intersection TNS} is inevitable since our next proof relies on it.
\begin{theorem}[Intersection of \textsc{ttns}]\label{thm:intersection of TNS}
Let $G$ be a tree with $d$ vertices and $c$ edges. Let $(r_1,\dots,r_c)$ and $(s_1,\dots,s_c) \in \mathbb{N}^{c}$. Then 
\[
 \tns(G; r_1,\dots,r_c; \mathbb{V}_1,\dots, \mathbb{V}_d) \cap \tns(G; s_1,\dots,s_c; \mathbb{V}_1,\dots, \mathbb{V}_d)
= \tns(G; t_1,\dots, t_c; \mathbb{V}_1,\dots, \mathbb{V}_d),
\]
where $(t_1,\dots, t_c)\in \mathbb{N}^{c}$ is given by $t_j = \min\{r_j,s_j\}$, $j=1,\dots,c$.
\end{theorem}
\begin{proof}
We just need to establish `$\subseteq$'  as `$\supseteq$' is obvious.
Let $T\in \tns(G; r_1,\dots,r_c; \mathbb{V}_1,\dots,\mathbb{V}_d) \cap  \tns(G; s_1,\dots,s_c; \mathbb{V}_1,\dots,\mathbb{V}_d)$. Then there exist subspaces $\mathbb{W}_1\subseteq \mathbb{V}_1,\dots,\mathbb{W}_c\subseteq \mathbb{V}_c$ such that both $\tns(G; r_1,\dots,r_c; \mathbb{W}_1,\dots,\mathbb{W}_d)$ and $\tns(G; s_1,\dots,s_c; \mathbb{W}_1,\dots,\mathbb{W}_d)$ are subcritical and
\begin{multline*}
T\in \tns(G; r_1,\dots,r_c; \mathbb{W}_1,\dots,\mathbb{W}_d) \cap  \tns(G; s_1,\dots,s_c; \mathbb{W}_1,\dots,\mathbb{W}_d) \\
= \tns(G; t_1,\dots,t_c; \mathbb{W}_1,\dots,\mathbb{W}_d)  \subseteq \tns(G; t_1,\dots,t_c; \mathbb{V}_1,\dots,\mathbb{V}_d),
\end{multline*}
where the equality follows from Lemma~\ref{lemma: intersection TNS} and the inclusion from \eqref{eq:subsp2}.
\end{proof}
Note that subspace varieties also satisfy the intersection property in Theorem~\ref{thm:intersection of TNS}, i.e.,
\[
\Sub_{r_1,\dots,r_d}(\mathbb{V}_1,\dots, \mathbb{V}_d)  \cap 
\Sub_{s_1,\dots,s_d}(\mathbb{V}_1,\dots, \mathbb{V}_d)  =  
\Sub_{t_1,\dots,t_d}(\mathbb{V}_1,\dots, \mathbb{V}_d).
\]
However neither  Lemma~\ref{lemma: intersection TNS} nor Theorem~\ref{thm:intersection of TNS} holds for graphs containing cycles, as we will see in Example~\ref{example:non uniqueness} and Proposition~\ref{prop:intersectMPS}.

We now establish the uniqueness of $G$-rank for any acyclic $G$, i.e., for a given $d$-tensor $T$, the $G$-rank of $T$ is a unique $d$-tuple in $\mathbb{N}^d$, as opposed to a subset of a few $d$-tuples in $\mathbb{N}^d$. In particular, the \textsc{tt}-rank  and \textsc{stns}-rank of a tensor are both unique.
\begin{theorem}[Uniqueness of \textsc{ttns}-rank]\label{thm:uniqueness G-rank}
The $G$-rank of a tensor $T\in \mathbb{V}_1\otimes \dots \otimes \mathbb{V}_d$ is unique if $G$ is a $d$-vertex tree.
\end{theorem}
\begin{proof}
We may assume that $T$ is nondegenerate; if not, we may replace $\mathbb{V}_1,\dots,\mathbb{V}_d$ by appropriate subspaces without affecting the $G$-rank of $T$, by Theorem~\ref{thm:$G$-rank subspace}. Let $(r_1,\dots,r_d)$ and $(s_1,\dots,s_d) \in \mathbb{N}^d$ be two $G$-ranks of $T$. By Proposition~\ref{prop:upper bound G-ranks}, $T$ lies in the intersection of $\tns(G; r_1,\dots,r_d; \mathbb{V}_1,\dots, \mathbb{V}_d)$ and $\tns(G; s_1,\dots,s_d; \mathbb{V}_1,\dots, \mathbb{V}_d)$, and both of them are subcritical since $T$ is nondegnerate. By Lemma~\ref{lemma: intersection TNS}, $T$ lies in $\tns(G; t_1,\dots,t_d; \mathbb{V}_1,\dots, \mathbb{V}_d)$ where $t_i \le r_i$ and $t_i \le s_i$, $i=1,\dots,d$. By the minimality in the definition of $G$-rank, we must have $r_i = s_i = t_i$, $i=1,\dots,d$, as required.
\end{proof}

\begin{corollary}[Uniqueness of generic \textsc{ttns}-rank]\label{cor:uniqueness generic G-rank}
Let $G$ be a tree with $d$ vertices. For any vector spaces $\mathbb{V}_1,\dots, \mathbb{V}_d$, there is a unique generic $G$-rank for $\mathbb{V}_1\otimes \dots \otimes \mathbb{V}_d$.
\end{corollary}
As we saw in Examples~\ref{example:P_3 generic rank} and \ref{example:P_4 generic rank}, the unique generic $P_3$-rank of $\mathbb{C}^{m \times mn \times n}$ is $(m,n)$ while the unique generic $P_4$-rank of $\mathbb{C}^{2\times 2 \times 2 \times 2}$ is $(2,4,2)$. There will be many examples in Sections~\ref{sec:mps} and \ref{sec:eg} showing that neither Theorem~\ref{thm:uniqueness G-rank} nor Corollary~\ref{cor:uniqueness generic G-rank} holds for graphs containing cycles.

The next result is an important one. It guarantees that whenever $G$ is acyclic, $G$-rank is  upper semicontinuous and thus the kind of illposedness issues in \cite{DSL} where a tensor may lack a best low-rank approximation do not happen.
\begin{theorem}[\textsc{ttns} are closed]\label{thm:tns of trees are closed}
Let $G$ be a tree with $d$ vertices. For any vector spaces $\mathbb{V}_1,\dots, \mathbb{V}_d$ and any $(r_1,\dots,r_c) \in \mathbb{N}^c$, the set  $\tns(G; r_1,\dots,r_d; \mathbb{V}_1,\dots, \mathbb{V}_d)$ is Zariski closed in $\mathbb{V}_1\otimes \dots \otimes \mathbb{V}_d$. 
\end{theorem}
\begin{proof}
We proceed by induction on $d$. The statement holds trivially when $d = 1$ by \eqref{eq:isolate}. Suppose it holds for all trees with at most $d-1$ vertices. Let $G$ be a $d$-vertex tree. Applying Proposition~\ref{prop:reduction supercritical} to $\tns(G; r_1,\dots,r_c; \mathbb{V}_1,\dots, \mathbb{V}_d)$, there is a subbundle $\mathcal{E}$ of the bundle $\mathcal{S}_1\times \dots \times \mathcal{S}_d$ on $\Gr(k_1,n_1)\times \dots \times \Gr(k_d,n_d)$ whose fiber over a point $([\mathbb{W}_1], \dots, [\mathbb{W}_d])$ is 
\[
F: = \tns(G; r_1,\dots,r_c; \mathbb{W}_1,\dots, \mathbb{W}_d),
\]
with the surjective birational map $\pi: \mathcal{E}\to  \tns(G; r; \mathbb{V}_1,\dots, \mathbb{V}_d)$ induced by the projection map 
\[
\operatorname{pr}_2: \bigl[\Gr(k_1,n_1)\times \dots \times \Gr(k_d,n_d)\bigr] \times \bigl[\mathbb{V}_1\times \dots \times \mathbb{V}_d\bigr] \to \mathbb{V}_1\times \dots \times \mathbb{V}_d.
\]
$\operatorname{pr}_2$ is a closed map since Grassmannian varieties are projective. Thus $\pi$ is also a closed map. To show that $\tns(G; r_1,\dots,r_c; \mathbb{V}_1,\dots, \mathbb{V}_d)$ is Zariski closed, it suffices to show that $\mathcal{E}$ is Zariski closed in $\bigl[\Gr(k_1,n_1)\times \dots \times \Gr(k_d,n_d)\bigr] \times \bigl[\mathbb{V}_1\times \dots \times \mathbb{V}_d\bigr]$. As $\mathcal{E}$ is a fiber bundle on $\Gr(k_1,n_1)\times \dots \times \Gr(k_d,n_d)$ with fiber $F$, it in turn suffices to show that $F$ is Zariski closed in $\mathbb{V}_1\otimes \dots \otimes \mathbb{V}_d$. Since $F = \tns(G; r_1,\dots,r_c; \mathbb{W}_1,\dots, \mathbb{W}_d)$
is critical, we may apply Proposition~\ref{prop:reduction valence one} to $F$ and regard $F$ as the tensor network of a tree with $d-1$ vertices. Therefore $F$ is Zariski closed by the induction hypothesis.
\end{proof}

Theorems~\ref{thm:uniqueness G-rank} and \ref{thm:tns of trees are closed} together yield the following corollary, which is in general false when $G$ is not acyclic (see
 Theorem~\ref{thm:brner}).
\begin{corollary}[Border \textsc{ttns}-rank]\label{cor:border G-rank}
For any tree $G$, border $G$-rank equals $G$-rank and is unique.
\end{corollary}

It is well-known that tensor rank is NP-hard but multilinear rank is polynomial-time computable. We will next see that like multilinear rank, $G$-rank is polynomial-time computable whenever $G$ is acyclic. We begin with some additional notations.  Let $G = (V,E)$ be a connected tree with $d$ vertices and $c$ edges.  Since $G$ is a connected tree,  removing any edge $\{i,j\} \in E$ results in a disconnected graph with two components.\footnote{\label{fn4}These components are also connected trees and so the subsequent argument may be repeated on each of them.} Let $V(i)$ denote the set of vertices in the component connected to the vertex $i$. Then we have a disjoint union $V = V(i) \sqcup V(j)$. Now for any vector spaces $\mathbb{V}_1,\dots,\mathbb{V}_d$, we may define a \emph{flattening map} associated with each edge $\{i,j\} \in E$,
\begin{equation}\label{eq:flat}
\flat_{ij} : \mathbb{V}_1\otimes \cdots \otimes \mathbb{V}_d \to \Bigl(\bigotimes_{h\in V(i)} \mathbb{V}_h\Bigr) \otimes \Bigl(\bigotimes_{h \in V(j)} \mathbb{V}_h \Bigr).
\end{equation}
Note that $\rank\bigl(\flat_{ij} (T)\bigr)$ is polynomial-time computable as matrix rank for any  $T\in \mathbb{V}_{1}\otimes \cdots \otimes \mathbb{V}_d$.

\begin{lemma}\label{lemma:G-rank of ttns}
Let $\mathbb{V}_1,\mathbb{V}_2,\mathbb{E}$ be vector spaces and let $\kappa: (\mathbb{V}_1 \otimes \mathbb{E} ) \times  (\mathbb{E}^\ast \otimes \mathbb{V}_2) \to \mathbb{V}_1\otimes \mathbb{V}_2$ defined by contracting factors in $\mathbb{E}$ with factors in $\mathbb{E}^\ast$. For any $T_1\in \mathbb{V}_1\otimes \mathbb{E}$ and $T_2\in \mathbb{E}^\ast \otimes \mathbb{V}_2$, the rank of $\kappa( (T_1,T_2))$ is at most $\dim \mathbb{E}$.
\end{lemma}
\begin{proof}
We denote by $r$ the dimension of $\mathbb{E}$ and we take a basis $e_1,\dots,e_r$ of $\mathbb{E}$ with dual basis $e_1^\ast,\dots, e_r^\ast$. By definition, we may write 
\[
T_1 = \sum_{i=1}^r e_i \otimes x_i,\quad T_2 = \sum_{i=1}^r e_i^\ast \otimes y_i, 
\]
for some $x_i\in \mathbb{V}_1,y_i\in \mathbb{V}_2,i=1,\dots,r$. Hence we have 
\[
\kappa((T_1,T_2)) = \sum_{i=1}^r x_i \otimes y_i,
\]
and this shows that the rank of $\kappa((T_1,T_2))$ is at most $r$.
\end{proof}

\begin{theorem}[\textsc{ttns}-rank is polynomial-time computable]\label{thm:G-rank of ttns}
Let $G$ be a tree with $d$ vertices and $c$ edges labeled as in \eqref{eq:graphlabels}, i.e., $V=\{1,\dots,d\}$ and  $E =\bigl\{\{i_1,j_1\},\dots, \{i_c,j_c\} \bigr\}$. Then for any  $T\in \mathbb{V}_1\otimes \cdots \otimes \mathbb{V}_d$, the $G$-rank of $T$ is given by
\begin{equation}\label{eq:ttnsrank}
\rank_G(T)  = (r_1,\dots,r_c), \qquad r_p = \rank\bigl(\flat_{i_p j_p} (T)\bigr), \qquad p=1,\dots,c.
\end{equation}
\end{theorem}
\begin{proof}
We will show that for $(r_1,\dots,r_c)$ as defined in \eqref{eq:ttnsrank}, (i)  $T\in \tns(G;r_1,\dots,r_c;\mathbb{V}_1,\dots, \mathbb{V}_d)$, i.e., $\rank_G(T) \le (r_1,\dots,r_c)$; and (ii) it is minimal in $\mathbb{N}_0^c$ such that (i) holds. Together, (i) and (ii) imply that $\rank_G(T) = (r_1,\dots,r_c)$.

Let $p$ be an integer such that $1\le p \le c$. Since $r_p = \rank \bigl(\flat_{i_p j_p}(T)\bigr)$, we may write 
\[
T = \sum_{i=1}^{r_p} R_i \otimes S_i, \quad R_i \in \bigotimes_{h \in V(i_p)} \mathbb{V}_h, \quad S_i\in \bigotimes_{h \in V(j_p)} \mathbb{V}_h, \quad i=1,\dots,r_p.
\]
Let $\mathbb{E}_p$ be a vector space of dimension $r_p$ attached to the edge $\{i_p, j_p\}$.  Let  $e_1,\dots, e_{r_p}$ be a basis of  $\mathbb{E}_p$ and $e_1^\ast,\dots, e_{r_p}^\ast$ be the corresponding dual basis of $\mathbb{E}_p^\ast$. Then
\[
T = \kappa_G \Bigl( \Bigl[\sum_{i=1}^{r_p}R_i\otimes e_i\Bigr] \otimes \Bigl[\sum_{j=1}^{r_p} e_j^\ast \otimes S_j \Bigr] \Bigr).
\]
We let $R_i$ (resp.\ $S_j$) take the role of $T$ and repeat the argument. Let $\{i_q,j_q\} \in E$ be such that $i_q, j_q$ are both in $V(i_p)$. Then $V(i_p)$ is the disjoint union  $V(i_p,i_q) \sqcup V(i_p,j_q)$ where $V(i_p, \ast)$ denotes the subset of $V(i_p)$ comprising all vertices in the component of vertex $\ast$ upon removal of $\{i_q, j_q\}$ (see Footnote~\ref{fn4}). Since $R_i \in \bigotimes_{h \in V(i_p)} \mathbb{V}_h$, we may write
\begin{equation}\label{eq:Qki}
R_i =\sum_{k=1}^{r} P_{ik} \otimes Q_{ki}, \quad P_{ik}\in \bigotimes_{h \in V(i_p,i_q)} \mathbb{V}_h, \quad Q_{kj} \in \bigotimes_{h\in V(i_p,j_q)} \mathbb{V}_h, \quad k =1,\dots,r,
\end{equation}
for some $r \in \mathbb{N}$. We claim that we may choose $r \le r_q$. Since $ r_q =\rank \bigl(\flat_{i_qj_q} (T)\bigr)$, 
\[
\dim \operatorname{span}\{P_{ik} : i=1,\dots, r_p,\; k=1,\dots, r\} = r_q,
\]
and so we may find $P_1,\dots, P_{r_q}$ such that each $P_{ik}$ is a linear combination of $P_1,\dots, P_{r_q}$. Thus for each $i= 1,\dots, r_p$, $R_i$ can be written as
\[
R_i =\sum_{k=1}^{r_q} P_k \otimes Q'_{ki}, \quad P_k \in \bigotimes_{h \in V(i_p,i_q)} \mathbb{V}_h, \quad Q'_{ki} \in \bigotimes_{h\in V(i_p,j_q)} \mathbb{V}_h, \quad  k=1,\dots, r_q,
\]
where each $Q'_{ki}$ is a linear combination of the $Q_{ki}$'s in \eqref{eq:Qki}.
Then we may write $T$ as 
\begin{align*}
T &= \kappa_G \Bigl( \Bigl[\sum_{i=1}^{r_p} \Bigl(\sum_{k=1}^{r_q}  P_k \otimes f_k\Bigr)\otimes \Bigl( \sum_{l=1}^{r_q} f_l^\ast \otimes Q'_{li} \Bigr) \otimes e_i \Bigr] \otimes \Bigl[\sum_{j=1}^{r_p} e_j^\ast \otimes S_j \Bigr] \Bigr) \\
& = \kappa_G \Bigl(\Bigl[ \sum_{k=1}^{r_q} P_k\otimes f_k \Bigr] \otimes \Bigl[\sum_{i,l=1}^{r_p,r_q} f_l^\ast \otimes Q'_{li} \otimes e_i \Bigr] \otimes \Bigl[\sum_{j=1}^{r_p}  e_j^\ast \otimes S_j \Bigr] \Bigr),
\end{align*}
where $f_1,\dots,f_{r_q}$ is a basis of $\mathbb{E}_q$ and $f_1^\ast,\dots, f_{r_q}^\ast$ is the corresponding dual basis of $\mathbb{E}_q^*$.

Repeating the process in the previous paragraph until we exhaust all edges, we obtain $T = \kappa_G(T_1\otimes \cdots \otimes T_d)$ for some
\[
T_i\in \Bigl(\bigotimes_{j\in \In(i)} \mathbb{E}_j \Bigr) \otimes \mathbb{V}_i \otimes \Bigl(\bigotimes_{k\in \Out(i)} \mathbb{E}_k^\ast\Bigr), \quad \dim \mathbb{E}_j = r_j, \quad i=1,\dots,d,\; j =1,\dots, c.
\]
and thus $T \in \tns(G;r_1,\dots,r_c;\mathbb{V}_1,\dots, \mathbb{V}_d)$, establishing (i).

Suppose $ (s_1,\dots,s_c)\le (r_1,\dots, r_c) $ is such that   $T \in \tns(G;s_1,\dots,s_c;\mathbb{V}_1,\dots, \mathbb{V}_d)$, i.e., $T = \kappa_G(T_1\otimes \cdots \otimes T_d)$ for some
\[
T_i\in \Bigl(\bigotimes_{j\in \In(i)} \mathbb{F}_j \Bigr) \otimes \mathbb{V}_i \otimes \Bigl(\bigotimes_{k\in \Out(i)} \mathbb{F}_k^\ast\Bigr), \quad \dim \mathbb{F}_j = s_j, \quad i=1,\dots,d,\; j =1,\dots, c.
\]
However, for each $p=1,\dots, c$, we can also write 
\[
T = \kappa_G\Bigl( \Bigl[\bigotimes_{h\in V(i_p)} T_h \Bigr] \otimes \Bigl[ \bigotimes_{h \in V(j_p)} T_h \Bigr] \Bigr)
\]
where
\begin{align*}
\Bigl[\bigotimes_{h\in V(i_p)} T_h \Bigr] &\in \Bigl(\bigotimes_{h \in V(i_p)}      
\Bigl( \bigotimes_{j\in \In(h)} \mathbb{F}_j \otimes \mathbb{V}_h \otimes  \bigotimes_{j\in \Out(h),j\ne p} \mathbb{F}_j^\ast \Bigr) \Bigr) \otimes \mathbb{F}_p^\ast, \\
 \Bigl[ \bigotimes_{h \in V(j_p)} T_h \Bigr] &\in  \mathbb{F}_p \otimes \Bigl(\bigotimes_{k \in V(j_p) }          \Bigl( \bigotimes_{j\in \In(k),j\ne p} \mathbb{F}_j  \otimes \mathbb{V}_k  \otimes  \bigotimes_{j\in \Out(k)} \mathbb{F}_j^\ast \Bigr)  \Bigr).
\end{align*}
This together with Lemma \ref{lemma:G-rank of ttns} imply that $r_p = \rank\bigl(\flat_{i_pj_p}(T) \bigr) \le s_p$ and therefore $r_p = s_p$, establishing (ii).
\end{proof}
A weaker form of Theorem~\ref{thm:G-rank of ttns} that establishes  the upper bound $\rank_G(T)  \le (r_1,\dots,r_c)$  in \eqref{eq:ttnsrank}  under the assumption that  $\mathbb{V}_1,\dots,\mathbb{V}_d$ are Hilbert spaces appeared in \cite[Theorem~3.3]{BSU}. An immediate consequence of Theorem~\ref{thm:G-rank of ttns} is the following.
\begin{corollary}[\textsc{ttns} as an algebraic variety]\label{cor:ideal of ttns}
Let $G = (V,E)$ be a tree with $d$ vertices and $c$ edges. For any $(n_1,\dots,n_d) \in \mathbb{N}^d$ and $(r_1,\dots,r_c)\in \mathbb{N}^{c}$, $\tns(G;r_1,\dots,r_c;n1,\dots,n_d)$ is  an irreducible algebraic variety in $\mathbb{C}^{n_1 \times \dots \times n_d}$ with vanishing ideal generated by all $(r_p + 1)\times (r_p + 1)$ minors of the flattening map \eqref{eq:flat} for $\{i_p,j_p\} \in E$, taken over all $p =1,\dots, c$.
\end{corollary}
We will see in the next section that when $G$  contains a cycle, $G$-rank cannot be computed as matrix ranks of flattening maps and $\tns(G;r_1,\dots,r_c;n_1,\dots, n_d)$ is not Zariski closed  in general.

\section{Matrix product states}\label{sec:mps}

We will restrict our attention in this section to the case where  $G = C_d$, the cyclic graph on $d$ vertices in Figure~\ref{fig:C}. This gives us the matrix product states --- one of the  most widely used class of tensor network states. We  start by stating the dimensions of \textsc{mps} in the supercritical and subcritical cases. Theorem~\ref{thm:dimension of MPS} and Corollary~\ref{cor: dimension of MPS} appear as \cite[Theorem~4.10 and Corollary~4.11]{dtn}.
\begin{theorem}[Dimension of \textsc{mps}]\label{thm:dimension of MPS}
Let $C_d$ be the cycle graph with $d\ge 3$ vertices and $d$ edges. Let $(r_1,\dots,r_d)\in \mathbb{N}^{d}$ be such that $\tns(C_d; r_1,\dots,r_d; n_1,\dots, n_d)$ is supercritical or critical. Then 
\[
\dim \tns(C_d; r_1,\dots,r_d; n_1,\dots, n_d) = \sum_{i=1}^d r_i r_{i+1}n_i  - \sum_{i=1}^d r_i^2 + 1,
\]
where $r_{d+1} \coloneqq r_1$.
\end{theorem}
It follows from the above dimension count that every element in $\mathbb{C}^{ m \times n \times mn}$ is an \textsc{mps} state.
\begin{corollary}\label{cor: dimension of MPS}
Let $C_3$ be the three-vertex cycle graph and
\[
\dim \mathbb{V}_1 = n_1, \quad \dim \mathbb{V}_2 = n_2,\quad \dim \mathbb{V}_3 = n_1 n_2.
\]
Then
\[
\overline{\tns}(C_3; n_1,n_2,1; \mathbb{V}_1,  \mathbb{V}_2, \mathbb{V}_3) = \mathbb{V}_1\otimes \mathbb{V}_2 \otimes \mathbb{V}_3.
\]
\end{corollary}

\begin{theorem}[Generic $C_3$-ranks of \textsc{mps}]\label{thm: generic border rank of MPS}
Let $C_3$ be the cycle graph on three vertices. If $\dim \mathbb{V}_i= n_i$, $i =1,2,3$ and $n_2n_3 \ge n_1,n_1n_3\ge n_2$, then $\mathbb{V}_1\otimes \mathbb{V}_2 \otimes \mathbb{V}_3$ has generic $C_3$-rank  $(n_1,n_2,1)$.
\end{theorem}
\begin{proof}
First we have 
\[
\tns(C_3; n_1,n_2,1; \mathbb{V}_1,\mathbb{V}_2,\mathbb{V}_3) =  \tns(P_3;n_1,n_2;\mathbb{V}_1,\mathbb{V}_3,\mathbb{V}_2) = \mathbb{V}_1\otimes \mathbb{V}_2 \otimes \mathbb{V}_3,
\]
where the first equality follows from Proposition~\ref{prop:remove an edge} and the second follows from Proposition~\ref{prop:reduction valence one}. Next, we claim that there does not exist $(r_1,r_2,1) \in \mathbb{N}^3$ such that 
\[
\overline{\tns}(C_3; r_1,r_2,1; \mathbb{V}_1,\mathbb{V}_2,\mathbb{V}_3) = \mathbb{V}_1\otimes \mathbb{V}_2 \otimes \mathbb{V}_3
\]
and that $r_1 \le n_1$, $r_2 \le n_2$ with at least one strict inequality. Take  $r_1< n_1$ for example (the other case may be similarly argued). Again by Proposition~\ref{prop:remove an edge},
\[
\dim \tns(C_3; r_1,r_2,1; \mathbb{V}_1,\mathbb{V}_2,\mathbb{V}_3) =  \dim \tns(P_3;r_1,r_2;\mathbb{V}_1,\mathbb{V}_3,\mathbb{V}_2).
\]
By Proposition~\ref{prop:reduction supercritical},
\begin{align*}
\dim \tns(P_3;r_1,r_2;\mathbb{V}_1,\mathbb{V}_3,\mathbb{V}_2) &= \dim \Gr(r_1,n_1) + \dim \tns(P_2;r_2;r_1n_3,n_2) \\
&\le r_1(n_1-r_1) + r_1n_2n_3 < n_1n_2n_3
\end{align*}
as $n_2n_3 \ge n_1 > r_1$. Thus
\[
\tns(C_3; r_1,r_2,1; \mathbb{V}_1,\mathbb{V}_2,\mathbb{V}_3) \subsetneq \mathbb{V}_1\otimes \mathbb{V}_2 \otimes \mathbb{V}_3.
\]
Hence $(n_1,n_2,1)$ is a generic $C_3$-rank of $\mathbb{V}_1\otimes \mathbb{V}_2\otimes \mathbb{V}_3$. 
\end{proof}

\begin{corollary}\label{cor: generic border rank of MPS}
If $\dim \mathbb{V}_1= \dim \mathbb{V}_2 = \dim \mathbb{V}_3 = n$, then $(n,n,1)$, $(1,n,n)$, and $(n,1,n)$ are all the generic $C_3$-ranks of $\mathbb{V}_1\otimes \mathbb{V}_2\otimes \mathbb{V}_3$.
\end{corollary}
\begin{proof}
Apply Theorem~\ref{thm: generic border rank of MPS} to the case $n_1 = n_2 =n_3 = n$ to see that $(n,n,1)$ is a generic rank of $\mathbb{V}_1\otimes \mathbb{V}_2\otimes \mathbb{V}_3$. Now we may permute $\mathbb{V}_1,\mathbb{V}_2$ and $\mathbb{V}_3$ to obtain the other two generic $C_3$-ranks.
\end{proof}

In case the reader is led to the false belief that the sums of entries of generic  $C_3$-ranks are always equal, we give an example to show that this is not the case.
\begin{example}
Let $\mathbb{V}_1, \mathbb{V}_2, \mathbb{V}_3$ be of dimensions $n_1, n_2, n_3$ where 
\[
n_2 \ne n_3 \quad \text{and}\quad n_i n_j \ge n_k \;\text{whenever}\; \{i,j,k\} = \{1,2,3\}.
\]
By Theorem~\ref{thm: generic border rank of MPS}, we see that $(1,n_2,n_1)$ and $(n_1,1,n_3)$ are both generic $C_3$-ranks of $\mathbb{V}_1\otimes \mathbb{V}_2 \otimes \mathbb{V}_3$ but $1+ n_2 + n_1 \ne n_1 + 1 + n_3$.
\end{example}

The following provides a necessary condition for generic $C_d$-rank of supercritical \textsc{mps}.
\begin{theorem}[Test for generic $C_d$-rank]\label{thm: generic border g rank of MPS 1}
Let $(n_1,\dots, n_d) \in \mathbb{N}^d$ and $ (r_1,\dots, r_d) \in \mathbb{N}^d$ be such that
$\tns (C_d; r_1,\dots, r_d; n_1,\dots,n_d)$
is supercritical, i.e., $n_ i \ge  r_i r_{i+1}, i=1,\dots, d$, where $r_{d+1} \coloneqq r_1$. 
If $ (r_1,\dots, r_d) $ is a generic $C_d$-rank of $\mathbb{V}_1\otimes \cdots \otimes \mathbb{V}_d$, then there exists  $j \in \{1,\dots, d\}$ such that
\begin{equation}\label{eq:ndn}
\prod_{i \ne j} n_i < d n_j.
\end{equation}
\end{theorem}
\begin{proof}
Fix $(r_1,\dots, r_d) \in \mathbb{N}^d$ and consider the function 
\[
f(n_1,\dots, n_d) \coloneqq \prod_{i=1}^d n_i - \sum_{i=1}^d r_i r_{i+1}n_i + \sum_{i=1}^d r_i^2 - 1.
\]
We have 
\[
f(n_1,\dots, n_d) \ge \prod_{i=1}^d n_i - \sum_{i=1}^d n^2_i + \sum_{i=1}^d r_i^2 -1
= \sum_{i=1}^d \frac{1}{d} \prod_{i=1}^d (n_i - n_i^2) + \sum_{i=1}^d r_i^2 -1.
\]
If $(r_1,\dots, r_d)$ is a generic $C_3$-rank, then $f(n_1,\dots,n_d) = 0 $ by Theorem~\ref{thm:dimension of MPS}. This implies that for some $j=1,\dots ,d$, we must have \eqref{eq:ndn}.
\end{proof}

The next example shows that intersection of \textsc{mps} are more intricate than that of \textsc{ttns}. In particular, Lemma~\ref{lemma: intersection TNS} does not hold for \textsc{mps}.
\begin{example}[Intersection of \textsc{mps}]\label{example:non uniqueness}
Let $\mathbb{U}$, $\mathbb{V}$, $\mathbb{W}$ be two-dimensional vector spaces associated respectively to vertices $1$, $2$, $3$ of $C_3$. Let $r_i \in \mathbb{N}$ be the weight of the edge \emph{not} adjacent to the vertex $i \in V = \{ 1,2,3\}$.
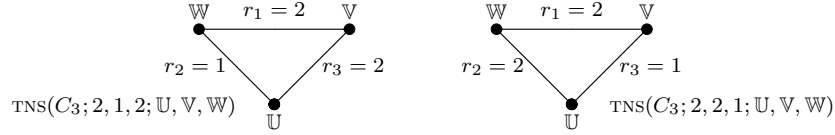
\[
\begin{tikzpicture}
\filldraw
    (-2,0.25) node[align=right, below] {$\tns(C_3;2,1,2;\mathbb{U},\mathbb{V},\mathbb{W})$}
    (0,0) circle (2pt) node[align=center, below] {$\mathbb{U}$}
 -- (0.5,0.5) circle (0pt) node[align=center, right] {$r_3=2$} 
 -- (1,1) circle (2pt)  node[align=center, above] {$\mathbb{V}$}  
 -- (0,1) circle (0pt) node[align=center, above] {$r_1 = 2$} 
 -- (-1,1) circle (2pt)  node[align=center, above] {$\mathbb{W}$}  
  -- (-0.5,0.5) circle (0pt) node[align=center, left] {$r_2=1$} 
 -- (0,0) circle (2pt) node[align=center, below]{}
-- cycle;
\end{tikzpicture}
\qquad
\begin{tikzpicture}
\filldraw
    (2,0.25) node[align=left, below] {$\tns(C_3;2,2,1;\mathbb{U},\mathbb{V},\mathbb{W})$}
    (0,0) circle (2pt) node[align=center, below] {$\mathbb{U}$}
 -- (0.5,0.5) circle (0pt) node[align=center, right] {$r_3 = 1$} 
 -- (1,1) circle (2pt)  node[align=center, above] {$\mathbb{V}$}  
 -- (0,1) circle (0pt) node[align=center, above] {$r_1 = 2$} 
 -- (-1,1) circle (2pt)  node[align=center, above] {$\mathbb{W}$}  
  -- (-0.5,0.5) circle (0pt) node[align=center, left] {$r_2 = 2$} 
 -- (0,0) circle (2pt) node[align=center, below]{}
-- cycle;
\end{tikzpicture}
\]
By Corollary~\ref{cor: generic border rank of MPS}, we see that 
\[
\tns(C_3;2,1,2;\mathbb{U},\mathbb{V},\mathbb{W}) = \tns(C_3;2,2,1;\mathbb{U},\mathbb{V},\mathbb{W}) = \mathbb{U} \otimes \mathbb{V} \otimes \mathbb{W}.
\]
Observe that an element in $ \tns(C_3;2,1,2;\mathbb{U},\mathbb{V},\mathbb{W})$ takes the form 
\[
u_1\otimes v_{1}\otimes w_1 + u_1\otimes v_{2}\otimes w_2 + u_2\otimes v_{3}\otimes w_1 + u_2\otimes v_{4}\otimes w_2
\]
where $u_1,u_2\in \mathbb{U}$,  $v_1, v_2, v_3,  v_4 \in \mathbb{V}$, $w_1, w_2 \in \mathbb{W}$.
By symmetry, an element in $\tns(C_3;2,2,1;\mathbb{U},\mathbb{V},\mathbb{W})$ takes the form
\[
u_1\otimes v_{1}\otimes w_1 + u_1\otimes v_{2}\otimes w_2 + u_2\otimes v_{1}\otimes w_3 + u_2\otimes v_{2}\otimes w_4
\]
where $u_1,u_2\in \mathbb{U}$,  $v_1, v_2  \in \mathbb{V}$, $w_1, w_2, w_3, w_4 \in \mathbb{W}$.
However, $\tns(C_3;1,1,2;\mathbb{U},\mathbb{V},\mathbb{W})$ is a proper subset of $\mathbb{U}\otimes \mathbb{V}\otimes \mathbb{W}$ since an element in $\tns(C_3;1,1,2;\mathbb{U},\mathbb{V},\mathbb{W})$ takes the form
\[
u_1 \otimes v_1\otimes w + u_2 \otimes v_2\otimes w
\]
where $u_1,u_2\in \mathbb{U}$,  $v_1, v_2  \in \mathbb{V}$, $w \in \mathbb{W}$. Hence
\[
\tns(C_3;2,1,2;\mathbb{U},\mathbb{V},\mathbb{W}) \cap \tns(C_3;2,2,1;\mathbb{U},\mathbb{V},\mathbb{W}) \supsetneq \tns(C_3;2,1,1;\mathbb{U},\mathbb{V},\mathbb{W})
\]
and thus Lemma~\ref{lemma: intersection TNS} does not hold for $C_3$.
\end{example}
The main difference between tensor network states associated to trees and those associated to cycle graphs is that one may apply Proposition~\ref{prop:reduction valence one} to trees but not to graphs containing cycles. In particular, the induction argument used to prove Lemma~\ref{lemma: intersection TNS} fails for non-acyclic graphs. 

Example~\ref{example:non uniqueness} generalizes to the following proposition, i.e., Theorem~\ref{thm:intersection of TNS} is always false for \textsc{mps}.
\begin{proposition}[Intersection of \textsc{mps}]\label{prop:intersectMPS}
Let $d\ge 3$ and $C_d$ be the cycle graph with $d$ vertices. Then there exists $\mathbb{V}_1,\dots, \mathbb{V}_d$ and $\underline{r}= (r_1,\dots,r_d)$, $\underline{s} =(s_1,\dots,s_d)\in \mathbb{N}^d$ such that 
\begin{equation}\label{eq:mpsintersect}
\tns(C_d; \underline{t}; \mathbb{V}_1,\dots,\mathbb{V}_d) \subsetneq \tns(C_d; \underline{r}; \mathbb{V}_1,\dots,\mathbb{V}_d) \cap \tns(C_d; \underline{s}; \mathbb{V}_1,\dots,\mathbb{V}_d), 
\end{equation}
where $\underline{t}= (t_1,\dots,t_d)\in \mathbb{N}^d$ is given by $t_i = \min (r_i,s_i)$,  $i =1,\dots, d$.
\end{proposition}
\begin{proof}
Let all $\mathbb{V}_i$'s be two-dimensional. If $d$ is odd, set
\[
\underline{r} = (2, 1, 2, 1, \dots, 2, 1, 2 ), \quad
\underline{s} = (2, 2, 1, 2, \dots, 1, 2, 1), \quad
\underline{t} = (2, 1, 1, 1, \dots, 1, 1, 1),
\]
and it is easy to check that \eqref{eq:mpsintersect} holds with strict inclusion.
If $d \equiv 0 \pmod 4$, write $d = 4m $, set
\[
\underline{r}  = (\overbrace{1, 2, \dots, 1, 2}^{2m}, \overbrace{2, 1, \dots, 2, 1}^{2m}), \quad
\underline{s}  = (\overbrace{2, 1, \dots, 2, 1}^{2m}, \overbrace{1, 2, \dots, 1, 2}^{2m});
\]
if $d \equiv 2 \pmod 4$, write $d = 4m +2$, set
\[
\underline{r} = (\overbrace{1, 2, \dots, 1, 2}^{2m}, 1, 1, \overbrace{2,1,\dots, 2,1}^{2m}),\quad
\underline{s} = (\overbrace{2, 1, \dots, 2, 1}^{2m}, 2, 2, \overbrace{1, 2,\dots, 1,2}^{2m}).
\]
In both cases, we have $\underline{t} = (1,\dots,1)$ and it is easy to verify \eqref{eq:mpsintersect}.
\end{proof}

We saw that generic $C_d$-rank is not unique. The next example shows, nonconstructively, that there are tensors with nonunique $C_d$-ranks. We give explicitly constructed examples in Section~\ref{sec:eg}.
\begin{example}[\textsc{mps}-rank not unique up to permutation]
Let $\dim \mathbb{V}_1 = \dim \mathbb{V}_2 = 2$ and  $\dim \mathbb{V}_3 = 3$. Consider the \textsc{mps}'s of $C_3$-ranks $\underline{r} = (2,1,2)$, $\underline{s} = (1,2,3)$, and $\underline{t} = (1,1,2)$ respectively:
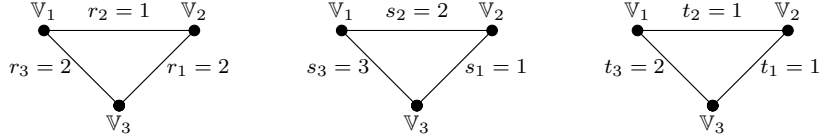
\[
\begin{tikzpicture}
\filldraw
    (0,0) circle (2pt) node[align=center, below] {$\mathbb{V}_3$}
 -- (0.5,0.5) circle (0pt) node[align=center, right] {$r_1 = 2$} 
 -- (1,1) circle (2pt)  node[align=center, above] {$\mathbb{V}_2$}  
 -- (0,1) circle (0pt) node[align=center, above] {$r_2 = 1$} 
 -- (-1,1) circle (2pt)  node[align=center, above] {$\mathbb{V}_1$}  
  -- (-0.5,0.5) circle (0pt) node[align=center, left] {$r_3 = 2$} 
 -- (0,0) circle (2pt) node[align=center, below]{}
-- cycle;
\end{tikzpicture}
\qquad
\begin{tikzpicture}
\filldraw
    (0,0) circle (2pt) node[align=center, below] {$\mathbb{V}_3$}
 -- (0.5,0.5) circle (0pt) node[align=center, right] {$s_1 = 1$} 
 -- (1,1) circle (2pt)  node[align=center, above] {$\mathbb{V}_2$}  
 -- (0,1) circle (0pt) node[align=center, above] {$s_2 = 2$} 
 -- (-1,1) circle (2pt)  node[align=center, above] {$\mathbb{V}_1$}  
  -- (-0.5,0.5) circle (0pt) node[align=center, left] {$s_3 = 3$} 
 -- (0,0) circle (2pt) node[align=center, below]{}
-- cycle;
\end{tikzpicture}
\qquad
\begin{tikzpicture}
\filldraw
    (0,0) circle (2pt) node[align=center, below] {$\mathbb{V}_3$}
 -- (0.5,0.5) circle (0pt) node[align=center, right] {$t_1 = 1$} 
 -- (1,1) circle (2pt)  node[align=center, above] {$\mathbb{V}_2$}  
 -- (0,1) circle (0pt) node[align=center, above] {$t_2 = 1$} 
 -- (-1,1) circle (2pt)  node[align=center, above] {$\mathbb{V}_1$}  
  -- (-0.5,0.5) circle (0pt) node[align=center, left] {$t_3 = 2$} 
 -- (0,0) circle (2pt) node[align=center, below]{}
-- cycle;
\end{tikzpicture}
\]
It is straightforward to see that 
\begin{gather*}
\tns(C_3; 1,1,2; \mathbb{V}_1,\mathbb{V}_2,\mathbb{V}_3) \subsetneq \tns(C_3; 2,1,2; \mathbb{V}_1,\mathbb{V}_2,\mathbb{V}_3) \cap \tns(C_3; 1,2,3; \mathbb{V}_1,\mathbb{V}_2,\mathbb{V}_3),\\
\overline{\tns}(C_3; 2,1,2; \mathbb{V}_1,\mathbb{V}_2,\mathbb{V}_3) = \overline{\tns}(C_3; 1,2,3; \mathbb{V}_1,\mathbb{V}_2,\mathbb{V}_3) = \mathbb{V}_1 \otimes \mathbb{V}_2 \otimes \mathbb{V}_3.
\end{gather*}
Thus a generic $T \in \mathbb{V}_1 \otimes \mathbb{V}_2 \otimes \mathbb{V}_3$ such that $T  \notin \tns(C_3; 1,1,2; \mathbb{V}_1,\mathbb{V}_2,\mathbb{V}_3)$  has at least two $C_3$-ranks $(2,1,2)$ and $(1,2,3)$.
\end{example}

As we saw in Theorem~\ref{thm:tns of trees are closed} and Corollary~\ref{cor:border G-rank}, for an acyclic $G$, $G$-rank is closed, i.e., border $G$-rank and $G$-rank are equivalent. We will see here that this is always false when $G$ is not acyclic.  In the following, we show that \textsc{mps}-rank is never closed by constructing a $d$-tensor whose border $C_d$-rank is strictly less than $C_d$-rank for each $d \ge 3$, extending \cite[Theorem~2]{LQY}.

Let $\mathbb{V} = \mathbb{C}^{n\times n}$ and let $\{E_{ij}\in \mathbb{C}^{n \times n} : i,j =1,\dots,n\}$ be the standard basis as in the proof of Theorem~\ref{thm:compare}. For each $d \ge 3$, define
\begin{equation}\label{eqn:border example}
T \coloneqq \sum_{i,j,k=1}^n (E_{ij} \otimes E_{jj} + E_{ii} \otimes E_{ij}) \otimes E_{jk} \otimes R_{ki} \in \mathbb{V}^{\otimes d},
\end{equation}
where for each $k,i =1,\dots, n$,
\[
R_{ki} \coloneqq \sum_{j_1,\dots, j_{d-4} =1}^n E_{k j_1} \otimes  E_{j_1 j_2} \otimes  \cdots \otimes E_{j_{d-5}  j_{d-4}} \otimes  E_{j_{d-4} i} \in \mathbb{V}^{\otimes (d-3)}.
\]
We adopt the convention that  $E_{jk}\otimes R_{ki} = E_{ji}$ in \eqref{eqn:border example} when $d =3$ and $R_{ki} = E_{ki}$ if $d = 4$. The following is a straightforward generalization of \cite[Theorem~2]{LQY}, with a similar proof.

\begin{theorem}\label{thm:boerder example}
Let $d \ge 3$ and $T$ be defined as above. Then (i) $T \in \overline{\tns}(C_d;n,\dots, n;\mathbb{V},\dots,\mathbb{V})$;  (ii) $T \notin \tns(C_d;n,\dots, n;\mathbb{V},\dots,\mathbb{V})$; (iii) $T \notin \Sub_{m_1,\dots,m_d} (\mathbb{V},\dots, \mathbb{V})$ whenever $m_i \le n^2$, $i=1,\dots, d$, with at least one strict inequality.
\end{theorem}

\begin{corollary}\label{cor:border example}
$\brank_{C_d}(T) = (n,\dots, n)$ but  $\rank_{C_d}(T) \ne (n,\dots, n)$.
\end{corollary}
\begin{proof}
By Theorem~\ref{thm:boerder example}, we have $\rank_{C_d}(T) \ne (n,\dots, n)$ and $\brank_{C_d}(T) \le (n,\dots, n)$. It remains to establish equality in the latter.
Suppose not, then $\brank_{C_d}(T)  = (r_1,\dots,r_d)$ where $r_i \le n$, $i=1,\dots, d$, with at least one strict inequality. Assume without loss of generality that $r_1 < n$. Then $r_1 r_2 < n^2$ and thus
\[
T \in \Sub_{n^2,r_1r_2,n^2,\dots, n^2} (\mathbb{V}, \dots, \mathbb{V}),
\]
contradicting Theorem~\ref{thm:boerder example}(iii).
\end{proof}

\begin{theorem}[Nonacyclic $G$-rank is not closed]\label{thm:brner}
Let $G = (V, E)$ be a connected graph with  $d$ vertices and $c$ edges that contains a cycle subgraph $C_{b}$ for some $b \le d$, i.e., there exist $b$ vertices $i_1,\dots, i_{b} \in V$ such that the $b$ edges $(i_1,i_2),\dots, (i_{b-1},i_{b}), (i_{b},i_1) \in E$. Then there exists $S\in \mathbb{V}^{\otimes d}$ such that
\[
\brank_{G}(S) = (s_1,\dots, s_c) \le  ( r_1,\dots, r_c) = \rank_{G}(S),
\]
with $s_i < r_i$ for at least one $i \in \{1,\dots, c\}$.
\end{theorem}
\begin{proof}
Relabeling the vertices if necessary, we may assume that $i_1=1,\dots, i_b=b$, i.e., the first $b$ vertices of $G$ form the cycle subgraph $C_b$. Relabeling the edges if necessary, we may also assume that $r_1,\dots, r_{b}$ are the weights associated to $(1,2),\dots, (b,1)$, i.e., the edges of $C_b$. Let $r_1 = \dots = r_{b} = n$ and $r_{b+1} = \dots = r_c = 1$. Let $T\in \mathbb{V}^{\otimes b}$ be as defined in \eqref{eqn:border example} (with $b$ in place of $d$) and let $S = T \otimes v^{\otimes (d - b)} \in \mathbb{V}^d$  where $v\in \mathbb{V}$ is a nonzero vector. Then 
\[
S \in \overline{\tns}(G;\underbrace{n,\dots, n}_{b}, \underbrace{1,\dots, 1}_{c-b}; n^2,\dots,n^2).
\]
So $\brank_{G}(S) \le (n,\dots, n, 1,\dots, 1) \in \mathbb{N}^b \times \mathbb{N}^{c - b} =\mathbb{N}^c$.
On the other hand, if $(r_1,\dots,r_{b},1,\dots, 1)\in \mathbb{N}^c$ is a border $G$-rank of $S$ such that $r_i \le n$, $i=1,\dots, b$, with at least one strict inequality, then $S\in \Sub_{m_1,\dots, m_d} (\mathbb{V},\dots, \mathbb{V})$ where $m_i \le n^2$, $i=1,\dots, b$, with at least one strict inequality. But this contradicts Theorem~\ref{thm:boerder example}(iii). Hence $\brank_{G}(S) = (n,\dots, n, 1,\dots, 1) \in \mathbb{N}^b \times \mathbb{N}^{c-b} =\mathbb{N}^c$. Lastly, by the way $S$ is constructed,  if $(r_1,\dots,r_b) \in \mathbb{N}^{b}$ is a $C_{b}$-rank of $T$, then $(r_1,\dots,r_{b},1,\dots, 1)\in \mathbb{N}^c$ is a $G$-rank of $S$. Since by Corollary~\ref{cor:border example}, $\rank_{C_b}(T) = (r_1,\dots, r_{b})\in \mathbb{N}^{b}$ where $n \le r_i $, $ i=1,\dots, b$, with at least one strict inequality,  this completes the proof.
\end{proof}

\section{Tensor network ranks of common tensors}\label{sec:eg}

We will compute some $G$-ranks of some well-known tensors from different fields:
\begin{description}
\item[Algebra] $G$-rank of decomposable tensors and monomials, $S_n$-ranks of decomposable symmetric and skew-symmetric tensors (Section~\ref{sec:algebra}); 

\item[Physics] $P_d$-rank and $C_d$-rank of the $d$-qubit W and GHZ states (Section~\ref{sec:physics});

\item[Computing] $P_3$-rank and $C_3$-rank of the structure tensor for matrix-matrix product (Section~\ref{sec:computing}).
\end{description}

\subsection{Tensors in algebra}\label{sec:algebra}

The following shows that the term `rank-one' is unambiguous --- all rank-one tensors have $G$-rank one regardless of $G$, generalizing Example~\ref{example:TTd}. 
\begin{proposition}
Let $G$ be a connected graph with $d$ vertices and $c$ edges. Let $0 \ne v_1 \otimes \dots \otimes v_d\in \mathbb{V}_1 \otimes \cdots \otimes \mathbb{V}_d $ be a rank-one tensor. Then the $G$-rank of $v_1 \otimes \dots \otimes v_d$ is unique and equals $(1,\dots,1) \in \mathbb{N}^c$.
\end{proposition}
\begin{proof}
As usual, let $\dim \mathbb{V}_i=n_i$, $i =1,\dots, n_d$. 
It follows easily from Definition~\ref{def:tns} that
\begin{equation}\label{eq:G-rank of rank one tensors}
\tns(G;1,\dots, 1;n_1,\dots, n_d) = \Seg(\mathbb{V}_1, \dots,\mathbb{V}_d),
\end{equation}
and so $(1,\dots, 1)$ is a $G$-rank of $v_1 \otimes \dots \otimes v_d$. Conversely, if $ (r_1,\dots, r_c)$ is a $G$-rank of $v_1 \otimes \dots \otimes v_d$, then $ (r_1,\dots, r_c)$ is minimal in $\mathbb{N}^c$ such that  $v_1 \otimes \dots \otimes v_d \in \tns(G;r_1,\dots, r_c;n_1,\dots ,n_d)$.
However, by \eqref{eq:G-rank of rank one tensors}, $v_1 \otimes \dots \otimes v_d \in \tns(G;1,\dots, 1;n_1,\dots ,n_d)$,
and obviously $1\le r_i$, $i=1,\dots, c$, implying that $r_i = 1$, $i=1,\dots, c$.
\end{proof}

We next discuss \emph{decomposable symmetric tensors} and \emph{decomposable skew-symmetric tensors}. For those unfamiliar with these notions, they are defined respectively as
\begin{align*}
v_1 \circ \cdots \circ v_d  &\coloneqq \frac{1}{d!}  \sum_{\sigma\in \mathfrak{S}_d} v_{\sigma(1)} \otimes \dots \otimes v_{\sigma(d)}  \in \mathsf{S}^d(\mathbb{V}) \subseteq \mathbb{V}^d,\\
v_1\wedge \cdots \wedge v_d &\coloneqq  \frac{1}{d!} \sum_{\sigma\in \mathfrak{S}_d} \sgn(\sigma) v_{\sigma(1)} \otimes \dots \otimes v_{\sigma(d)}  \in \mathsf{\Lambda}^d(\mathbb{V}) \subseteq \mathbb{V}^d,
\end{align*}
where  $v_1,\dots, v_d \in \mathbb{V}$, an $n$-dimensional vector space, and 
where $\sgn(\sigma)$ denotes the sign of the permutation $\sigma\in \mathfrak{S}_d$.

\begin{theorem}[\textsc{stns}-rank of decomposable (skew-)symmetric tensors]\label{thm:Sn-rank}
Let $S_n$ be the star graph with vertices $1,\dots, n$ and with $1$ as the root vertex. If $d = n$ and $v_1,\dots, v_n \in \mathbb{V}$ are linearly independent, then
 $v_1 \circ \dots \circ v_n $ and $ v_1 \wedge \dots \wedge v_n $ both have $S_n$-rank $(n,\dots,n) \in \mathbb{N}^n$.
\end{theorem}
\begin{proof}
Since $\tns(S_n;n,\dots, n; n,\dots ,n) = \mathbb{V}^{\otimes n}$, 
\begin{equation}\label{eq: S_n-rank of Lambda_n, Sigma_n}
v_1 \circ \dots \circ v_n , \; v_1 \wedge \dots \wedge v_n  \in \tns(S_n;n,\dots, n; n,\dots ,n).
\end{equation} 
It remains to show that there does not exist $(r_1,\dots ,r_n)$ such that $r_i\le n$, $i=1,\dots, n$, with at least one strict inequality, such that  $v_1 \circ \dots \circ v_n$ and $ v_1 \wedge \dots \wedge v_n   \in \tns(S_n;r_1,\dots, r_n;n,\dots ,n)$. But this is clear as $v_1 \circ \dots \circ v_n $ and $ v_1 \wedge \dots \wedge v_n $  are nondegenerate tensors in $\mathbb{V}^{\otimes n}$.
\end{proof}
It is easy to construct explicit $S_n$-decompositions of $v_1 \circ \dots \circ v_n $ and $ v_1 \wedge \dots \wedge v_n $ in Theorem~\ref{thm:Sn-rank}. Let $\mathbb{E}$ be $n$-dimensional with basis $e_1,\dots, e_n$ and dual basis $e_1^\ast,\dots, e_n^\ast$. Consider
\begin{align*}
\Sigma &= \frac{1}{n!} \sum_{\sigma\in \mathfrak{S}_n} v_{\sigma(1)} \otimes  e_{\sigma(2)} \otimes \dots \otimes e_{\sigma(n)} \in \mathbb{V} \otimes \mathbb{E}^{\otimes (n-1)},\\
\Lambda &= \frac{1}{n!} \sum_{\sigma\in \mathfrak{S}_n} \sgn(\sigma) v_{\sigma(1)} \otimes e_{\sigma(2)} \otimes \dots \otimes e_{\sigma(n)}  \in \mathbb{V} \otimes \mathbb{E}^{\otimes (n-1)},
\end{align*}
and $M = \sum_{i=1}^n e_j^\ast \otimes  v_j \in \mathbb{E}^\ast \otimes \mathbb{V}$.
Then
\[
\kappa_{S_n}(\Sigma \otimes M^{\otimes (n-1)}) = v_1 \circ \dots \circ v_n \qquad \text{and} \qquad \kappa_{S_n}(\Lambda \otimes M^{\otimes (n-1)}) =  v_1 \wedge \dots \wedge v_n.
\]

A \emph{monomial} of degree $d$ in $n$ variables $x_1,\dots, x_n$ may be regarded \cite{L} as a decomposable symmetric $d$-tensor over an $n$-dimensional vector space $\mathbb{V}$:
\[
x_1^{p_1} \cdots x_n^{p_n} = v_1^{\otimes p_1} \circ \cdots \circ v_n^{\otimes p_n}\in \mathsf{S}^d (\mathbb{V}),
\]
where $p_1,\dots, p_n$ are nonnegative integers such that  $p_1 + \cdots + p_n = d$.
If $d = n$ and $p_1=\dots = p_d = 1$, then $ v_1 \circ \dots \circ v_d = x_1\cdots x_d$, i.e., a decomposable symmetric tensor is a special case.

\begin{proposition}[$G$-rank of monomials]\label{prop: G-rank of monomials}
Let $G$ be a connected graph with $d$ vertices and $c$ edges. Let $n \le d$ and  $p_1 + \cdots + p_n = d$. Let $\mathbb{U}$ be a $d$-dimensional vector space with basis $u_1,\dots, u_d$ and $\mathbb{V}$ an $n$-dimensional vector space with basis $v_1,\dots, v_n$. If
$u_1 \circ \dots \circ u_d  \in \tns(G;r_1,\dots, r_c;d,\dots, d)$,
then
$v_1^{\otimes p_1} \circ \cdots \circ v_n^{\otimes p_n} \in \tns(G;r_1,\dots, r_c; n,\dots,n)$, i.e.,
\begin{equation}\label{eq:bound}
\rank_G(v_1^{\otimes p_1} \circ \cdots \circ v_n^{\otimes p_n} ) \le \rank_G(u_1 \circ \dots \circ u_d).
\end{equation}
\end{proposition}
\begin{proof}
Let $\varphi : \mathbb{U} \to \mathbb{V}$ be the linear map  that sends $u_{p_i+1},\dots, u_{p_{i+1}}$ to $v_{i+1}$, $i=0,\dots, n-1$, where $p_0 \coloneqq 0$. Let $\varphi^{\otimes d}: \mathbb{U}^{\otimes d} \to \mathbb{V}^{\otimes d}$ be the linear map induced by $\varphi$. Observe that 
$\varphi^{\otimes d} (u_1 \circ \dots \circ u_d) =  v_1^{\otimes p_1} \circ \cdots \circ v_n^{\otimes p_n}$.
Also, $\varphi^{\otimes d} \bigl(\tns(G;r_1,\dots,r_c; d,\dots, d)\bigr) \subseteq \tns(G;r_1,\dots,r_c; d,\dots, d)$. Hence
$v_1^{\otimes p_1} \circ \cdots \circ v_n^{\otimes p_n}  \in \tns(G;r_1,\dots,r_c; d,\dots, d) \cap (\mathbb{V}^{\otimes p_1}\otimes \cdots \otimes \mathbb{V}^{\otimes p_n})
= \tns(G;r_1,\dots, r_c; n,\dots ,n)$.
\end{proof}

The case where number of variables is larger than degree, i.e., $n > d$, reduces to Proposition~\ref{prop: G-rank of monomials}. In this case, a monomial $x_1^{p_1} \cdots x_n^{p_n}$ of degree $d$ will not involve all variables $x_1,\dots, x_n$ and, as a tensor, $v_1^{\otimes p_1} \circ \cdots \circ v_n^{\otimes p_n} \in \mathsf{S}^d (\mathbb{W})$ where $\mathbb{W} \subseteq \mathbb{V}$ is an appropriate subspace of dimension $ \le d$.

For comparison, the \emph{Waring rank}\footnote{The Waring rank of a polynomial $f$ of degree $d$ is the smallest $r$ such that $f = \sum_{i=1}^r l_i^d$ for linear forms $l_1,\dots,l_r$. Its Waring border rank is the smallest $r$ such that $f$ is a limit of a sequence of polynomials of Waring rank $r$.} or \emph{symmetric tensor rank} of a monomial  $v_1^{\otimes p_1} \circ \cdots \circ v_n^{\otimes p_n} $ where $p_1 \ge p_2 \ge \dots \ge p_n > 0 $, is $\prod_{i=1}^{n-1} ( p_i + 1)$, whereas its \emph{Waring border rank} is $\prod_{i=2}^{n} ( p_i + 1)$ \cite{LT2010,Oeding2016}; its multilinear rank is easily seen\footnote{The first flattening $\flat_1( v_1^{\otimes p_1} \circ \cdots \circ v_n^{\otimes p_n} )$, as a linear map, sends $v_i$ to $v_1^{p_1} \circ \dots \circ v_{i}^{p_i - 1} \circ \cdots v_n^{p_n}$, $i=1,\dots,n$. It has full rank $n$ since $v_1^{p_1} \circ \dots \circ v_{i}^{p_i - 1} \circ \cdots v_n^{p_n}$, $i=1,\dots,n$, are linearly independent. Ditto for other flattenings.} to be $(n,\dots, n)$. The monomials include $v_1 \circ \dots \circ v_n $ as a special case. As for $v_1 \wedge \dots \wedge v_n $, tensor rank and border rank are still open but its multilinear rank is also easily seen to be $(n,\dots, n)$.

\subsection{Tensors in physics}\label{sec:physics}
Let $\mathbb{V}$ be two-dimensional and $v_1,v_2$ be a basis. For any $d \ge 3$, the $d$-tensors in $\mathbb{V}^d$ defined by
\[
\W_d  \coloneqq \sum_{i=1}^d v_1^{\otimes (i-1)} \otimes v_2 \otimes  v_1^{\otimes (d-i)} \qquad\text{and}\qquad
\GHZ_d \coloneqq v_1^{\otimes d} + v_2^{\otimes d}, 
\]
are called the $d$-qubit \emph{W state} and \emph{GHZ state} (for Werner and Greenberger--Horne--Zeilinger) respectively.

Observe that  $\W_d = v_1^{d-1} \circ v_2$ is a decomposable symmetric tensor corresponding to the monomial $x_1^{d-1}x_2$. By \eqref{eq: S_n-rank of Lambda_n, Sigma_n} and Proposition~\ref{prop: G-rank of monomials},  we obtain  $\W_d \in \tns(S_d; d,\dots, d; 2,\dots, 2)$ but this also trivially follows from $\tns(S_d; d,\dots, d; 2,\dots, 2) =\tns(S_d; 2,\dots, 2; 2,\dots, 2) = \mathbb{V}^{d}$, which shows that
 the inequality can be strict in \eqref{eq:bound}.

We start with the $P_d$-rank of $W_d$. Let the vertices of $P_d$ be  $1,\dots, d$ and edges be oriented $(1,2),(2,3),\dots,(d-1,d)$.
Let $\mathbb{E}_i \simeq \mathbb{E}$ be a two-dimensional vector space associated to $(i,i+1)$, $i=1,\dots, d-1$, with basis $e_1,e_2$ and dual basis $e_1^\ast,e_2^\ast $. Note that  $\W_d = \kappa_{P_d} (A\otimes B^{\otimes (d-2)} \otimes C) \ \in \tns(P_d;2,\dots,2;2,\dots,2)$ with
$A = v_1 \otimes e_1^\ast + v_2\otimes e_2^\ast \in \mathbb{V}\otimes \mathbb{E}^\ast$,
$B  = e_1 \otimes ( v_1 \otimes e_1^\ast + v_2 \otimes e_2^\ast) + e_2 \otimes v_1 \otimes  e_2^\ast\in \mathbb{V}\otimes \mathbb{E}^\ast \otimes \mathbb{E}$, $C=  v_1 \otimes e_2 + v_2\otimes e_1 \in  \mathbb{V} \otimes \mathbb{E}$. 
So if $\rank_{P_d}(\W_d) = (r_1,\dots, r_{d-1})$, then we must have $r_i  = 1$ or $2$ for $i=1,\dots, d-1$. Let  $r_0 = r_d = 1$. Suppose $r_{i-1} r_{i} = 1$ for some $i \in \{1,\dots,d\}$. Then $\W_d \in \mathbb{V}^{\otimes (i-1)} \otimes \mathbb{W} \otimes \mathbb{V}^{\otimes (d-i)}$ where $\mathbb{W}\subseteq \mathbb{V}$ is a one-dimensional subspace, which is impossible by the definition of $\W_d$. Thus we obtain the following.
\begin{lemma}\label{lem: P_d rank of W_d}
If $\rank_{P_d}(\W_d) = (r_1,\dots, r_{d-1})$, then $r_i \in \{1,2\}$ and $r_{i-1} r_{i} \ge 2$, $i=1,\dots, d$, where $r_0 = r_d = 1$.
\end{lemma}

\begin{theorem}[\textsc{tt}-rank of W state]\label{thm:P_d-rank of W_d}
Let $d \ge 3$. Then
$\rank_{P_d}(\W_d) = (2,\dots, 2)\in \mathbb{N}^{d-1}$.
\end{theorem}
\begin{proof}
Suppose not. By Lemma~\ref{lem: P_d rank of W_d}, there exists $ i \in \{2, \dots, d-2\}$ such that  $r_i =1$, $r_j= 2$ for all $j \ne i$, and
\[
\W_d \in \smash{\tns(P_{d};\overbrace{2,\dots,2}^{i-1},1,\overbrace{2,\dots,2}^{d-i-1}; 2,\dots,2).}
\]
This implies that $\W_d  = X \otimes Y$ for some $X\in \mathbb{V}^{\otimes i}$ and $Y\in \mathbb{V}^{\otimes (d-i)}$, i.e., $\W_d$ is a rank-one $2$-tensor in $\mathbb{V}^{\otimes i} \otimes \mathbb{V}^{\otimes (d-i)} $, which is impossible by the definition of $\W_d$.
\end{proof}

We next deduce the $C_d$-ranks of $\W_d$.  
Let  the vertices of $C_d$ be $1,\dots, d$ and edges be oriented $(1,2),(2,3),\dots,(d,d+1)$, with $d+1 \coloneqq 1$.
Let $\mathbb{E}_i\simeq \mathbb{E}$ be a two-dimensional vector space associated to $(i,i+1)$, $i=1,\dots, d$, with basis $e_1,e_2$ and dual basis $e_1^\ast,e_2^\ast$.  Note that  $\W_d = \kappa_{C_d} (A\otimes B^{\otimes (d-2)} \otimes C) \in \tns(C_d;2,\dots,2;2,\dots,2)$ with
$A = e_1\otimes v_1 \otimes e_1^\ast + e_2 \otimes  v_2 \otimes  e_2^\ast$,
$B  = e_1 \otimes ( v_1 \otimes e_1^\ast + v_2 \otimes e_2^\ast) + e_2 \otimes  v_1 \otimes  e_2^\ast$, $C = e_2 \otimes v_1 \otimes  (e_1^\ast + e_2^\ast) +  e_1 \otimes v_2 \otimes  e_1^\ast$, all in $\mathbb{E} \otimes \mathbb{V} \otimes \mathbb{E}^\ast$.

\begin{theorem}[\textsc{mps}-rank of W state]\label{thm:C_d-rank of W_d}
Let $d \ge 3$. Then $\rank_{C_d}(\W_d) = (r_1,\dots, r_d) \in \mathbb{N}^d$ if and only if $r_i = 1$ for some $i\in \{1,\dots, d\}$ and all other $r_j = 2$, $ j\ne i$.
\end{theorem}
\begin{proof}
The ``if'' part follows from Theorem~\ref{thm:P_d-rank of W_d}. Since $\W_d \in \tns(C_d;2,\dots,2;2,\dots,2)$, if $(r_1,\dots,r_d)$ is a $C_d$-rank of $\W_d$ with $r_i \ge 2$ for all $i=1,\dots, d$, then $r_i =  2$ for all $i=1,\dots,d$. However, we also have  $\tns(C_d;r_1,\dots,r_d;2,\dots, 2) \subseteq \tns(C_d;2,\dots,2;2,\dots,2)$
for $1\le r_1,\dots, r_d \le 2$. This implies that $(2,\dots,2)$ cannot be a $C_d$-rank of $\W_d$, showing the ``only if'' part.
\end{proof}

We now proceed to the GHZ state. $\GHZ_2 = v_1^{\otimes 2} + v_2^{\otimes 2}\in \mathbb{V}\otimes \mathbb{V}$ is known as the \emph{Bell state}, a rank-two $2\times 2$ matrix. For the only connected graph with two vertices, $P_2$, and it is clear that $\rank_{P_2}(\GHZ_2) = 2$. For $d \ge 3$, the arguments for deducing the $P_d$-rank and $C_d$-rank of $\GHZ_d$ are very similar to those used for $W_d$ and we will be brief.  First observe that $\GHZ_d = \kappa_{P_d}(A\otimes B^{\otimes (d-2)} \otimes C)  \in \tns(P_d;2,\dots,2;2,\dots ,2)$ with
$A = v_1 \otimes e_1^\ast + v_2 \otimes e_2^\ast$, $B = e_1 \otimes v_1 \otimes e_1^\ast  + e_2 \otimes v_2 \otimes e_2^\ast,$ $C = e_1 \otimes v_1 + e_2 \otimes v_2$,
and we may obtain the following analogue of Theorem~\ref{thm:P_d-rank of W_d}.
\begin{theorem}[\textsc{tt}-rank of GHZ state]
Let $d \ge 3$. Then $\rank_{P_d}(\GHZ_d) = (2,\dots, 2)\in \mathbb{N}^{d-1}$.
\end{theorem}
Likewise, $\GHZ_d = \kappa_{C_d} (D^{\otimes d}) \in \tns(C_d;2,\dots, 2;2,\dots,2)$ with $D = e_1 \otimes v_1 \otimes e_1^\ast  + e_2 \otimes v_2 \otimes e_2^\ast$, and we obtain the following analogue of Theorem~\ref{thm:C_d-rank of W_d}
\begin{theorem}[\textsc{mps}-rank of GHZ state]
Let $d \ge 3$. Then $\rank_{C_d}(\GHZ_d) = (r_1,\dots, r_d) \in \mathbb{N}^d$ if and only if $r_i = 1$ for some $i\in \{1,\dots, d\}$ and all other $r_j = 2$, $ j\ne i$.
\end{theorem}

For comparison, note that  $\W_d$ and $\GHZ_d$ are respectively the monomial $x^{d-1}y$  and  the binary form $x^d + y^d$ regarded as symmetric tensors. By our discussion at the end of Section~\ref{sec:algebra}, the Waring rank of $\W_d$ is $d$ while that of $\GHZ_d$ is at most $2(d+1)$; the border rank and multilinear rank of both states are $2$ and  $(2,\dots,2)$ respectively.

\subsection{Tensors in computing}\label{sec:computing}

Let $\mathbb{U} = \mathbb{C}^{m\times n}$, $\mathbb{V}=\mathbb{C}^{n \times p}$, and $\mathbb{W} =\mathbb{C}^{m \times p}$.
Let $\mu_{m,n,p} \in \mathbb{U}^* \otimes \mathbb{V}^* \otimes\mathbb{W} \cong \mathbb{C}^{mn \times np \times mp} $ be the structure tensor for the product of $m\times n$ and $n \times p$ rectangular matrices \cite{YLstruct}, i.e.,
\begin{equation}\label{eq:strassen2}
\mu_{m,n,p}  =\sum_{i,j,k=1}^{n} u^*_{ik}\otimes v^*_{kj}\otimes w_{ij}
\end{equation}
where  $\{u_{ij} \in \mathbb{U}: i=1,\dots, m,\;  j=1,\dots, n \}$, $\{v_{jk}\in \mathbb{V} :  j = 1,\dots, n,\; k= 1,\dots, p \}$, $\{w_{ki} \in \mathbb{W} : k =1,\dots, p,\; i=1,\dots, p \}$ are the standard bases of the respective spaces (e.g., $u_{ij}$ is the $m \times n$ matrix with one in the $(i,j)$ entry and zeroes everywhere else). The reader might remember that we have encountered a special case of this tensor in \eqref{eq:strassen} --- that is the structure tensor for product of square matrices, i.e., $m = n = p$ and we wrote $\mu_n =\mu_{n,n,n}$.

The structure tensor for matrix-matrix product is widely regarded as the most important tensor in algebraic computational complexity theory; its tensor rank quantifies the optimum complexity for matrix-matrix product and has a current best-known bound of $O(n^{2.3728639})$ \cite{LeGall2014}. 
We will establish the $P_3$-rank and $C_3$-rank of $\mu_{m,n,p}$ in the following. For comparison, note that its multilinear rank is  $(mn,np,mp)$.
\begin{theorem}[\textsc{tt}-rank of Strassen tensor]\label{thm:P_3-rank of matrix multiplication}
Let $m, n, p \ge 2$. Then $\rank_{P_3}(\mu_{m,n,p}) =(mn,mp)$. 
\end{theorem}
\begin{proof}
Clearly $\mu_{m,n,p} \in  \tns(P_3;mn,mp;mn,np,mp) = \mathbb{U}^* \otimes \mathbb{V}^* \otimes\mathbb{W}$.
As $\mu_{m,n,p}$ is nondegenerate,  $\mu_{m,n,p} \notin \tns(P_3;r_1,r_2;mn,np,mp)$ if $r_1<mn,  r_2 = mp$ or if $r_1 = mn,  r_2 < mp$. 
\end{proof}
\begin{theorem}[\textsc{mps}-rank of Strassen tensor]\label{thm:C_3-rank of matrix multiplication}
Let $m, n, p \ge 2$. Then $(m,n,p)$, $(mn,mp,1)$, $(mn,1,np)$, $(1,mp,np)$ are all $C_3$-ranks of $\mu_{m,n,p}$.
\end{theorem}
\begin{proof}
By Proposition~\ref{prop:remove an edge} and Theorem~\ref{thm:P_3-rank of matrix multiplication}, we see that $(mn,mp,1)$, $(mn,1,np)$, $(1,mp,np)$ are $C_3$-ranks of $\mu_{m,n,p}$. It remains to show that $(m,n,p)$ is also a $C_3$-rank of $\mu_{m,n,p}$. 
Let $\{e_1, \dots, e_m\}$,  $\{f_1, \dots, f_n\}$, $\{g_1,\dots, g_p\}$ be any bases of vector spaces $\mathbb{E} $, $\mathbb{F}$, $\mathbb{G}$ respectively. Then $\mu_{m,n,p} = \kappa_{C_3}(A\otimes B \otimes C)  \in \tns(C_3;m,n,p;mn,np,mp)$ with
\begin{gather*}
A= \sum_{i,j=1}^{m,n} e_i \otimes u_{ij} \otimes f_{j}^\ast \in \mathbb{E} \otimes \mathbb{U} \otimes \mathbb{F}^\ast, \quad
B = \sum_{j,k=1}^{n,p} f_j \otimes v_{jk} \otimes g_{k}^\ast \in \mathbb{F} \otimes \mathbb{V} \otimes \mathbb{G}^\ast, \\
C = \sum_{k,i=1}^{p,m} g_k \otimes w_{ki} \otimes e_{i}^\ast \in \mathbb{G} \otimes \mathbb{W} \otimes \mathbb{E}^\ast.
\end{gather*}
Now let $(r_1,r_2,r_3)\le (m,n,p)$ such that  $\mu_{m,n,p} \in \tns(C_3; r_1,r_2,r_3;mn,np,mp)$. If, for example, $r_1 < m$, then $r_1 r_2 < mn$ and thus  $\mu_{m,n,p} \in \tns(C_3; r_1,r_2,r_3;r_1r_2,np,mp) \subseteq 
\mathbb{C}^{r_1r_2} \otimes \mathbb{C}^{np} \otimes \mathbb{C}^{mp}$, which is impossible by the definition of $\mu_{m,n,p}$. Similarly, we may exclude other cases, concluding that  $(r_1,r_2,r_3) = (m,n,p)$.
\end{proof}
In general, we do not know if  there might be other $C_3$-ranks of  $\mu_{m,n,p}$ aside from the four in Theorem~\ref{thm:C_3-rank of matrix multiplication} although for the case $m=n=p=2$, we do have 
\[
\rank_{C_3}(\mu_{2,2,2}) = \{(2,2,2), (1,4,4),(4,1,4),(4,4,1)\}.
\]
To see this, note that if $(r_1,r_2,r_3)$ is a $C_3$-rank of $\mu_{2,2,2}$ with $r_i \ge 2$, $i=1,2,3$, then $r_1=r_2=r_3=2$ by minimality of $C_3$-ranks; whereas if $r_i = 1$ for some $i$, say $r_1 = 1$, then $\mu_{2,2,2} \in \tns(C_3;1,r_2,r_3;2,2,2) = \tns(P_3;r_2,r_3;2,2,2)$, and as $\rank_{P_3}(\mu_{2,2,2}) =(4,4)$, we get $r_2 = r_3 = 4$. 

One may wonder if proofs of Theorems~\ref{thm:P_3-rank of matrix multiplication} and \ref{thm:C_3-rank of matrix multiplication} could perhaps give a new algorithm for matrix-matrix product along the lines of Strassen's famous algorithm. The answer is no: the proofs in fact only rely on the decomposition of $\mu_{m,n,p}$ given by the standard algorithm for matrix-matrix product.

\section{Tensor network ranks versus tensor rank and multilinear rank}\label{sec:compare}

In Section~\ref{sec:Grank}, we saw that $G$-ranks may be regarded as `interpolants' between tensor rank and multilinear rank. We will conclude this article by showing that they are nevertheless  distinct notions, i.e., tensor and multilinear ranks cannot be obtained as $G$-ranks. This is already evident in $3$-tensors and we may limit our discussions to this case. Since $d =3$ and there are only two connected graphs  with three vertices, we have only two choices for $G$ --- either $C_3$ or $P_3$.
\begin{proposition}\label{prop:ranknotGrank}
Let $\mathbb{V}_1,\mathbb{V}_2$, $\mathbb{V}_3$ be of dimensions $\ge 4$ and let the following sets be in  $\mathbb{V}_1 \otimes \mathbb{V}_2 \otimes \mathbb{V}_3$. There exists $r\in \mathbb{N}$ such that
\begin{equation}\label{eq:secant}
\overline{\{T : \rank(T) \le r\}}
\end{equation}
is not equal to
\begin{equation}\label{eq:C3P3}
\overline{\{T :  \rank_{P_3} \le (r_1,r_2)\}}
\quad \text{or} \quad
\overline{\{T :  \rank_{C_3}(T)\le (r_1,r_2,r_3)\}} 
\end{equation}
for any $r_1,r_2,r_3 \in \mathbb{N}$.
\end{proposition}
\begin{proof}
Note that the set on the left of \eqref{eq:C3P3} is  $\overline{\tns}(P_3;r_1,r_2;\mathbb{V}_1,\mathbb{V}_2,\mathbb{V}_3) \eqqcolon X_{r_1, r_2} $, the one on the right is $\overline{\tns}(C_3;r_1,r_2,r_3;\mathbb{V}_1,\mathbb{V}_2,\mathbb{V}_3) \eqqcolon Y_{r_1,r_2,r_3}$, and the set  in \eqref{eq:secant} is  $\sigma\bigl(\Seg(\mathbb{V}_1,\mathbb{V}_2,\mathbb{V}_3)\bigr) \eqqcolon \Sigma_r$. 

It suffices to take $r=2$.  Suppose $\Sigma_2 = X_{r_1, r_2}$ for some positive integers $r_1,r_2$. Then a generic element in $X_{r_1, r_2}$ must have rank $2$, which implies that $r_1, r_2 \le 2$. By Example~\ref{example:TT3}, we see that
\[
X_{r_1,r_2} = \Sub_{2,4,2}(\mathbb{V}_1,\mathbb{V}_2,\mathbb{V}_3),
\]
which implies that a generic $T\in X_{r_1,r_2}$ has $\mrank(T) = (2,4,2)$; but this gives a contradiction as any $T\in \Sigma_2$ must have $\mrank(T) \le (2,2,2)$.

Next we show that $ \Sigma_2 \ne Y_{r_1,r_2,r_3}$. We may assume that $r_1,r_2,r_3\ge 2$ or otherwise $Y_{r_1,r_2,r_3}$ becomes $X_{r_1, r_2}$, $X_{r_1, r_3}$, or $X_{r_2, r_3}$  by \eqref{eq:tt=mps}. So we have
\[
Y_{2,2,2} \subseteq Y_{r_1,r_2, r_3},
\]
which implies that the structure tensor $\mu_2$ for $2\times 2$ matrix-matrix product (cf.\  \eqref{eq:strassen} and \eqref{eq:strassen2}) is contained in $Y_{r_1,r_2, r_3}$. It is well-known \cite{Landsberg2006} that $\brank (T) = 7$  and thus $T\not\in \Sigma_2 $ (since that would mean $\brank(T) \le 2$). Hence $\Sigma_2 \ne Y_{r_1,r_2,r_3}$.
\end{proof}

\begin{proposition}\label{prop:mranknotGrank}
Let $\mathbb{V}$ be of dimension $n\ge 4$ and let the following sets be in  $\mathbb{V} \otimes \mathbb{V} \otimes \mathbb{V}$. There exist  $s_1,s_2,s_3 \in \mathbb{N}$ such that
\begin{equation}\label{eq:subsp}
\{T: \mrank(T) \le (s_1, s_2, s_3)\}
\end{equation}
is not equal to
\[
\overline{\{T :  \rank_{P_3} \le (r_1,r_2)\}}
\quad \text{or} \quad
\overline{\{T :  \rank_{C_3}(T)\le (r_1,r_2,r_3)\}} 
\]
for any $r_1,r_2,r_3 \in \mathbb{N}$.
\end{proposition}
\begin{proof}
We adopt the shorthands in the proof of Proposition~\ref{prop:ranknotGrank}. In addition, note that   the set in \eqref{eq:subsp} is $\Sub_{s_1,s_2,s_3}(\mathbb{V},\mathbb{V},\mathbb{V})  \eqqcolon Z_{s_1,s_2,s_3}$. It suffices to take $(s_1,s_2,s_3) = (2,2,2)$. It is obvious that for any $r_1,r_2 \in \mathbb{N}$,
\[
Z_{2,2,2} \ne Z_{r_1,r_1r_2,r_2} = X_{r_1, r_2}
\]
where the second equality follows from Example~\ref{example:TT3}. Next, suppose 
\[
Z_{2,2,2} = Y_{r_1,r_2,r_3},
\]
then a generic $T$ in $Y_{r_1,r_2,r_3}$ has $\mrank(T) = (2,2,2)$. However since $T \in Y_{r_1,r_2,r_3}$ has the form
\[
T = \sum_{i,j,k=1}^{r_1,r_2,r_3} u_{ij} \otimes v_{jk} \otimes w_{ki},
\]
we have $\mrank(T) = (\min\{r_1r_2, n\}, \min\{r_2r_3,n\},\min \{r_1r_3,n\})$ and therefore
\[
r_1r_2 = r_2 r_3 = r_3r_1 = 2,
\]
which cannot hold for any positive integers $r_1,r_2,r_3$. Hence $Z_{2,2,2} \ne Y_{r_1, r_2, r_3}$.
\end{proof}

\section{Conclusion}

We hope this article provides a convincing explanation as to why $G$-rank can be a better alternative to rank under many circumstances, and how it underlies the efficacy of tensor networks in computational physics and other applications. We also hope that the formalism introduced in this article would help establish a mathematical foundation for tensor networks, the study of which has thus far relied more on physical intuition, computational heuristics, and numerical experiments; but suffers from a lack of mathematically precise results built upon unambiguous definitions and rigorous proofs.

\subsection*{A word about MERA} A notable omission from this article is the \emph{multiscale entanglement renormalization ansatz} or \textsc{mera}, often also regarded as a tensor network state in physics literature. From a mathematical perspective, \textsc{mera}  differs in important ways from other known tensor networks like \textsc{tt}, \textsc{mps}, \textsc{peps}, and every other example discussed in our article ---  these can all be defined purely using tensor contractions but \textsc{mera} will require additional operations known as `isometries' and `disentanglers' in physics \cite{V1,V2,CLJM2008,EG2011}. From a physics perspective, the discussion in \cite[Section~III]{EG2011} also highlights a critical difference: while other tensor network states are derived from the \emph{physical geometry}, \textsc{mera} is derived from the \emph{holographic geometry} of the quantum system.

Although our definition of a tensor network state can be readily adapted to allow for more operations and thereby also include \textsc{mera}, in this article we restricted ourselves to tensor network states that can be constructed out of the three standard operations on tensors --- sums, outer products, contractions --- in multilinear algebra and leave \textsc{mera} to furture work.

\subsection*{Acknowledgments}

The authors would like to thank David Gross and Anna Seigal for helpful pointers. We owe special thanks to Miles Stoudenmire for his careful reading and enlightening feedback.

The work in this article is generously supported by AFOSR FA9550-13-1-0133, DARPA D15AP00109, NSF IIS 1546413, DMS 1209136, and DMS 1057064. In addition, LHL's work is  supported by a DARPA Director's Fellowship and KY's work is supported by the Hundred Talents Program of the Chinese Academy of Science and the Thousand Talents Program.

\end{document}